\documentclass[twoside]{article}

\usepackage[accepted]{aistats2023}
\usepackage[round]{natbib}

\usepackage[utf8]{inputenc} 
\usepackage[T1]{fontenc}    
\usepackage{hyperref}       
\usepackage{url}            
\usepackage{booktabs}       
\usepackage{amsfonts}       
\usepackage{nicefrac}       
\usepackage{microtype}      
\usepackage{xcolor}         
\usepackage{amsmath,amsfonts,amssymb,amsthm,array}
\usepackage{algorithm}
\usepackage[font=small]{caption}
\usepackage{subcaption}
\usepackage{algpseudocode}
\usepackage{bm}
\usepackage{color}
\usepackage{graphicx}
\usepackage{enumitem}
\usepackage{nicefrac}

\def\L{{\cal L}}

\newcommand{\Exp}{\mathbb{E}}
\newcommand{\PP}{\mathbb{P}}
\newcommand{\E}[1]{{\mathbb{E}\left[#1\right] }}    

\def\L{{\cal L}}

\newcommand{\EE}[2]{{\mathbb{E}_{#1}\left[#2\right] }} 

\newcommand{\R}{\mathbb{R}}

\newcommand{\norm}[1]{\left\| #1\right\|}

\usepackage{tikz}

\newcommand{\eqdef}{:=}
\newcommand{\prox}{\text{prox}}

\newcommand{\cC}{{\cal C}}
\newcommand{\cD}{{\cal D}}

\newcommand{\cO}{{\cal O}}

\newcommand{\cQ}{{\cal Q}}

\theoremstyle{plain}

\theoremstyle{remark}

\newcommand{\mA}{{\bf A}}
\newcommand{\mB}{{\bf B}}

\newcommand{\mD}{{\bf D}}

\newcommand{\mI}{{\bf I}}

\newcommand{\mM}{{\bf M}}

\newcommand{\mQ}{{\bf Q}}

\usepackage{mdframed} 
\usepackage{thmtools}

\usepackage[flushleft]{threeparttable} 

\usepackage{caption}
\usepackage{multirow}
\usepackage{colortbl}
\definecolor{bgcolor}{rgb}{0.8,1,1}
\definecolor{bgcolor2}{rgb}{0.8,1,0.8}
\definecolor{niceblue}{rgb}{0.0,0.19,0.56}

\usepackage{hyperref}
\hypersetup{colorlinks,linkcolor={blue},citecolor={niceblue},urlcolor={blue}}

\definecolor{shadecolor}{gray}{0.9}
\declaretheoremstyle[
headfont=\normalfont\bfseries,
notefont=\mdseries, notebraces={(}{)},
bodyfont=\normalfont,
postheadspace=0.5em,
spaceabove=1pt,
mdframed={
  skipabove=8pt,
  skipbelow=0pt,
  hidealllines=true,
  backgroundcolor={shadecolor},
  innerleftmargin=4pt,
  innerrightmargin=4pt}
]{shaded}

\declaretheorem[style=shaded,within=section]{definition}
\declaretheorem[style=shaded,sibling=definition]{theorem}
\declaretheorem[style=shaded,sibling=definition]{proposition}
\declaretheorem[style=shaded,sibling=definition]{assumption}
\declaretheorem[style=shaded,sibling=definition]{corollary}

\declaretheorem[style=shaded,sibling=definition]{lemma}

\allowdisplaybreaks
\usepackage{xspace}
\newcommand{\algname}[1]{{\sf  #1}\xspace}

\usepackage[colorinlistoftodos,bordercolor=orange,backgroundcolor=orange!20,linecolor=orange,textsize=scriptsize]{todonotes}

\begin{document}

%

%
\runningauthor{Aleksandr Beznosikov, Eduard Gorbunov, Hugo Berard, Nicolas Loizou}

\twocolumn[

\aistatstitle{Stochastic Gradient Descent-Ascent: Unified Theory and New Efficient Methods}

\aistatsauthor{ Aleksandr Beznosikov$^*$ \And Eduard Gorbunov$^*$ \And Hugo Berard$^*$ \And Nicolas Loizou }

\aistatsaddress{Innopolis University, \\ HSE University, \\ Yandex 
\And MBZUAI
\And  Mila and DIRO, \\ Universit\'e de Montr\'eal 
\And  AMS and MINDS,\\
  Johns Hopkins University 
} ]

\begin{abstract}
\vspace{-0.75cm}
  Stochastic Gradient Descent-Ascent (\algname{SGDA}) is one of the most prominent algorithms for solving min-max optimization and variational inequalities problems (VIP) appearing in various machine learning tasks. 
The success of the method led to several advanced extensions of the classical \algname{SGDA}, including variants with arbitrary sampling, variance reduction, coordinate randomization, and distributed variants with compression, which were extensively studied in the literature, especially during the last few years.
In this paper, we propose a unified convergence analysis that covers a large variety of stochastic gradient descent-ascent methods, which so far have required different intuitions, have different applications and have been developed separately in various communities. A key to our unified framework is a parametric assumption on the stochastic estimates. Via our general theoretical framework, we either recover the sharpest known rates for the known special cases or tighten them.
Moreover, to illustrate the flexibility of our approach, we develop several new variants of \algname{SGDA} such as a new variance-reduced method (\algname{L-SVRGDA}), new distributed methods with compression (\algname{QSGDA}, \algname{DIANA-SGDA}, \algname{VR-DIANA-SGDA}), and a new method with coordinate randomization (\algname{SEGA-SGDA}). Although variants of the new methods are known for solving minimization problems, they were never considered or analyzed for solving min-max problems and VIPs. We also demonstrate the most important properties of the new methods through extensive numerical experiments.
\end{abstract}

\section{INTRODUCTION}\label{sec:introduction}
Min-max optimization and, more generally, variational inequality problems (VIPs) appear in a wide range of research areas, including but not limited to statistics \citep{bach2019eta}, online learning \citep{cesa2006prediction}, game theory \citep{morgenstern1953theory}, and machine learning \citep{goodfellow2014generative}. Motivated by applications in these areas, in this paper, we focus on solving the following regularized VIP: Find $x^* \in \R^d$ such that
\begin{equation}
\label{eq:VI}
\langle F(x^*), x - x^* \rangle + R(x) - R(x^*) \geq 0\quad \forall x\in \R^d,
\end{equation}
where $F: \R^d \rightarrow \R^d$ is some operator and $R:\R^d \to \R$ is a regularization term (a proper lower semicontinuous convex function), which is assumed to have a simple structure. This problem is quite general and covers a wide range of possible problem formulations. For example, when operator $F(x)$ is the gradient of a convex function $f$, then problem \eqref{eq:VI} is equivalent to the composite minimization problem \citep{beck2017first}, i.e., minimization of $f(x) + R(x)$. Problem~\eqref{eq:VI} is also a more abstract formulation of the min-max problem
\begin{equation}
    \min\limits_{x_1 \in Q_1}\max\limits_{x_2 \in Q_2} f(x_1, x_2), \label{eq:minmax_problem}
\end{equation}
with convex-concave continuously differentiable $f$. In that case, first-order optimality conditions imply that \eqref{eq:minmax_problem} is equivalent to \eqref{eq:VI} with $x= (x_1^\top, x_2^\top)^\top$, $F(x) = (\nabla_{x_1}f(x_1,x_2)^\top , -\nabla_{x_2}f(x_1,x_2)^\top)^\top$, and $R(x) = \delta_{Q_1}(x_1) + \delta_{Q_2}(x_2)$, where $\delta_{Q}(\cdot)$ is an indicator function of the set $Q$~\citep{alacaoglu2021stochastic}. 
In addition, to formulate the constraints, regularization $R$ allows us to enforce some properties to the solution $x^*$, e.g., sparsity \citep{candes2008enhancing, beck2017first}.

More precisely, we are interested in the situations when operator $F$ is accessible through the calls of unbiased stochastic oracle. This is natural when $F$ has an expectation form $ F(x) = \Exp_{\xi \sim \cD}[F_\xi(x)]$ or a finite-sum form $F(x) = \tfrac{1}{n}\sum_{i=1}^n F_i(x)$. In the context of machine learning, $\cD$ corresponds to some unknown distribution on the data, $n$ corresponds to the number of samples, and $F_\xi$, $F_i$ denote vector fields corresponding to the samples $\xi$, and $i$, respectively \citep{gidel2018variational, loizou2021stochastic}.

One of the most popular methods for solving \eqref{eq:VI} is Stochastic Gradient Descent-Ascent\footnote{This name is usually used in the min-max setup. Although we consider a more general problem formulation, we keep the name \algname{SGDA} to highlight the connection with min-max problems.} (\algname{SGDA}) \citep{dem1972numerical, Nemirovski-Juditsky-Lan-Shapiro-2009}. However, besides its rich history, SGDA only recently was analyzed without using strong assumptions on the noise \citep{loizou2021stochastic} such as uniformly bounded variance. 
In the last few years, several powerful algorithmic techniques like variance reduction \citep{palaniappan2016stochastic, yang2020global} and coordinate-wise randomization \citep{sadiev2020zeroth}, were also combined with \algname{SGDA} resulting in better algorithms. However, these methods were analyzed under different assumptions, using different analysis approaches, and required different intuitions. Moreover, to the best of our knowledge, fruitful directions such as communication compression for distributed versions of \algname{SGDA} or linearly converging variants of coordinate-wise methods for regularized VIPs were never considered in the literature before.

All of these facts motivate the importance and necessity of a novel general analysis of \algname{SGDA} unifying several special cases and providing the ability to design and analyze new \algname{SGDA}-like methods filling existing gaps in the theoretical understanding of the method.
\begin{gather*}
    \boxed{\textit{In this work, we develop such unified analysis.}}
\end{gather*}

\subsection{Technical Preliminaries}

Throughout the paper, we assume that \eqref{eq:VI} has at least one solution and operator $F$ is $\mu$-\emph{quasi-strongly monotone} and $\ell$-\emph{star-cocoercive}: there exist constants $\mu \geq 0$ and $\ell > 0$ such that for all $x \in \R^d$
\begin{gather}
\left\langle F(x) - F(x^*),  x-x^*\right\rangle \geq \mu \|x-x^*\|^2, \label{eq:QSM}\\
    \|F(x) - F(x^*)\|^2 \le  \ell \langle F(x) - F(x^*), x - x^* \rangle, \label{eq:cocoercivity}
\end{gather}
where $x^* = \text{proj}_{X^*}(x) \eqdef \arg\min_{y\in X^*}\|y - x\|$ is the projection of $x$ on the solution set $X^*$ of \eqref{eq:VI}. If $\mu=0$, inequality \eqref{eq:QSM} is known as variational stability condition \cite{hsieh2020explore}, which is weaker than standard monotonicity: $\langle F(x) - F(y), x-y\rangle \geq 0$ for all $x,y\in \R^d$. It is worth mentioning that there exist examples of non-monotone operators satisfying \eqref{eq:QSM} with $\mu > 0$ \citep{loizou2021stochastic}. Condition \eqref{eq:cocoercivity} is a relaxation of standard cocoercivity $\|F(x) - F(y)\|^2 \le  \ell \langle F(x) - F(y), x - y \rangle$. At this point, let us highlight that it is possible
for an operator $F$ to satisfy \eqref{eq:cocoercivity} and not be Lipschitz continuous \citep{loizou2021stochastic}. This emphasizes the wider applicability of the $\ell$-\emph{star-cocoercivity} compared to $\ell$-cocoercivity. We emphasize that in our convergence analysis, we do not assume $\ell$-cocoercivity nor $L$-Lipschitzness of $F$.

We consider \algname{SGDA} for solving \eqref{eq:VI} in its general form:
\begin{equation}
    x^{k+1} = \prox_{\gamma_k R}(x^k - \gamma_k g^k), \label{eq:SGDA_general}
\end{equation}
where $g^k$ is an unbiased estimator of $F(x^k)$, $\gamma_k > 0$ is a stepsize at iteration $k$, and $\prox_{\gamma R}(x) \eqdef \arg\min_{y\in \R^d}\left\{R(y) + \nicefrac{\|y-x\|^2}{2\gamma}\right\}$ is a proximal operator defined for any $\gamma > 0$ and $x\in \R^d$. While $g^k$ gives an information about operator $F$ at step $k$, proximal operator is needed to take into account regularization term $R$. We assume that function $R$ is such that $\prox_{\gamma R}(x)$ can be easily computed for all $x\in \R^d$. This is a standard assumption satisfied for many practically interesting regularizers \citep{beck2017first}. By default we assume that $\gamma_k \equiv \gamma > 0$ for all $k \ge 0$.

\subsection{Our Contributions}

\begin{enumerate}[leftmargin=*]
    \item[$\diamond$] \textbf{Unified analysis of \algname{SGDA}.} We propose a general assumption on the stochastic estimates and the problem~\eqref{eq:VI} (Assumption~\ref{as:key_assumption}) and show that several variants of \algname{SGDA}~\eqref{eq:SGDA_general} satisfy this assumption. In particular, through our approach, we cover \algname{SGDA} with arbitrary sampling \citep{loizou2021stochastic}, variance reduction, coordinate randomization, and compressed communications. Under Assumption~\ref{as:key_assumption} we derive general convergence results for quasi-strongly monotone (Theorem~\ref{thm:main_result}), monotone star-cocoercive (Theorem~\ref{thm:main_result_monotone}) and cocoercive problems (Theorem~\ref{thm:main_result_monotone_coco}).
    
    \item[$\diamond$] \textbf{Extensions of known methods and analysis.} As a by-product of the generality of our theoretical framework, we derive new results for the proximal extensions of several known methods such as \textit{proximal} \algname{SGDA-AS} \citep{loizou2021stochastic} and \textit{proximal} \algname{SGDA} with coordinate randomization \citep{sadiev2020zeroth}. Moreover, we close some gaps on the convergence of known methods, e.g., we derive the first convergence guarantees in the monotone case for \algname{SGDA-AS} \citep{loizou2021stochastic} and \algname{SAGA-SGDA} \citep{palaniappan2016stochastic} and we obtain the first result on the convergence of \algname{SAGA-SGDA} for (averaged star-)cocoercive operators.

    \item[$\diamond$] \textbf{Sharp rates for known special cases.} For the known methods fitting our framework our general theorems either recover the best rates known for these methods (\algname{SGDA-AS}) or tighten them (\algname{SGDA-SAGA}, \algname{Coordinate SGDA}).

    \item[$\diamond$] \textbf{New methods.} The flexibility of our approach allows us to develop and analyze several new variants of \algname{SGDA}. Guided by algorithmic advances for solving minimization problems we propose a new variance-reduced method (\algname{L-SVRGDA}), new distributed methods with compression (\algname{QSGDA}, \algname{DIANA-SGDA}, \algname{VR-DIANA-SGDA}), and a new method with coordinate randomization (\algname{SEGA-SGDA}). We show that the proposed new methods fit our theoretical framework and, using our general theorems, we obtain tight convergence guarantees for them. Although the analogs of these methods are known for solving minimization problems \citep{hofmann2015variance, kovalev2019don, alistarh2017qsgd, mishchenko2019distributed, horvath2019stochastic, hanzely2018sega}, they were never considered for solving min-max and variational inequality problems. Therefore, by proposing and analyzing these new methods we close several gaps in the literature on \algname{SGDA}. For example, \algname{VR-DIANA-SGDA} is the first \algname{SGDA}-type \textit{linearly converging} distributed stochastic method with compression and \algname{SEGA-SGDA} is the first \textit{linearly converging} coordinate method for solving \textit{regularized} VIPs. 

    \item[$\diamond$] \textbf{Numerical evaluation.} In numerical experiments, we illustrate the most important properties of the new methods. The results corroborate our theoretical findings.
\end{enumerate}

Throughout the paper, we provide necessary comparisons with closely related work. Additional works relevant to our paper are discussed in Appendix~\ref{AppendixRelatedWork}.

\section{UNIFIED ANALYSIS OF \algname{SGDA}}\label{sec:unif_analysis}

\paragraph{Key assumption.} We start by introducing the next parametric assumption -- a central part of our approach.
\begin{assumption}\label{as:key_assumption}
    We assume that for all $k \ge 0$ the estimator $g^k$ from \eqref{eq:SGDA_general} is unbiased: $\Exp_k\left[g^k\right] = F(x^k)$, where $\Exp_k[\cdot]$ denotes the expectation w.r.t.\ the randomness at iteration $k$. Next, we assume that there exist non-negative constants $A, B, C, D_1, D_1 \ge 0$, $\rho \in (0, 1]$ and a sequence of (possibly random) non-negative variables $\{\sigma_k\}_{k\ge 0}$ such that for all $k\ge 0$
    \begin{eqnarray}
        \Exp_k\left[\|g^k - g^{*,k}\|^2\right]&\leq& 2A\langle F(x^k) - g^{*,k}, x^k - x^{*,k}\rangle
        \nonumber\\
        &&+ B\sigma_k^2 + D_1, \label{eq:second_moment_bound}\\
        \Exp_k\left[\sigma_{k+1}^2\right] &\leq& 2C\langle F(x^k) - g^{*,k}, x^k - x^{*,k}\rangle 
        \nonumber\\
        &&+ (1-\rho)\sigma_k^2 + D_2, \label{eq:sigma_k_bound}
    \end{eqnarray}
    where $x^{*,k} = \text{proj}_{X^*}(x^k)$ and $g^{*,k} = F(x^{*,k})$.
\end{assumption}
While unbiasedness of $g^k$ is a standard assumption, inequalities \eqref{eq:second_moment_bound}-\eqref{eq:sigma_k_bound} are new and require clarifications. For simplicity, assume that $\sigma_k^2 \equiv 0$, $F(x^*) = 0$ for all $x^* \in X^*$, and focus on \eqref{eq:second_moment_bound}. In this case, \eqref{eq:second_moment_bound} gives an upper bound for the second moment of the stochastic estimate $g^k$. For example, such a bound follows from expected cocoercivity assumption \citep{loizou2021stochastic}, where $A$ denotes some expected/averaged (star-)cocoercivity constant and $D_1$ stands for the variance at the solution (see also Section~\ref{sec:sgda}). When $F$ is not necessarily zero on $X^*$, the shift $g^{*,k}$ helps to take this fact into account. Finally, the sequence $\{\sigma_k^2\}_{k\geq 0}$ is typically needed to capture the variance reduction process, parameter $B$ is typically some numerical constant, $C$ is another constant related to (star-)cocoercivity\footnote{Although Assumption \ref{as:key_assumption} does not formally imply star-cocoercivity of $F$, but in all special cases, considered in this work, operator $F$ is star-cocoercive.}, and $D_2$ is the remaining noise that is not handled by the variance reduction process. As we show in the next sections, inequalities \eqref{eq:second_moment_bound}-\eqref{eq:sigma_k_bound} hold for various \algname{SGDA}-type methods.


We point out that Assumption~\ref{as:key_assumption} is inspired by similar assumptions appeared in \citet{gorbunov2020unified, gorbunov2021stochastic}. However, the difference between our assumption and the ones appeared in these papers is significant: \citet{gorbunov2020unified} focuses only on solving minimization problems and as a result, their assumption includes a much simpler quantity (function suboptimality), instead of the $\langle F(x^k) - g^{*,k}, x^k - x^{*,k}\rangle$, in the right-hand sides of \eqref{eq:second_moment_bound}-\eqref{eq:sigma_k_bound}. The assumption proposed in  \citet{gorbunov2021stochastic}, is designed specifically for analyzing vanilla Stochastic \algname{EG}, it does not have $\{\sigma_k^2\}_{k\geq 0}$ sequence (not able to capture variants of Stochastic \algname{EG} with variance reduction, quantization, nor coordinate-wise randomization)  and works only for \eqref{eq:VI} with $R(x) \equiv 0$. For more detailed comparison of our approach and this line of work, see Appendix~\ref{AppendixRelatedWork}.
\paragraph{Quasi-strongly monotone case.} Under Assumption~\ref{as:key_assumption} and quasi-strong monotonicity of $F$, we derive the following general result.
\begin{theorem}\label{thm:main_result}
    Let $F$ be $\mu$-quasi-strongly monotone ($\mu > 0$) and let Assumption~\ref{as:key_assumption} hold. Assume that $0 < \gamma \leq \min\left\{\nicefrac{1}{\mu}, \nicefrac{1}{2(A + CM)}\right\}$ for some $M > \nicefrac{B}{\rho}$ (when $B = 0$, we suppose $M = 0$ and $\nicefrac{B}{M}\eqdef 0$ in all following expressions). Then the iterates of \algname{SGDA} \eqref{eq:SGDA_general}, satisfy:
    \begin{eqnarray}
        \Exp[V_k] &\leq& \left(1 - \min\left\{\gamma\mu, \rho - \frac{B}{M}\right\}\right)^k V_0 \nonumber
        \\
        &&+ \frac{\gamma^2(D_1 + MD_2)}{\min\left\{\gamma\mu, \rho - \nicefrac{B}{M}\right\}}. \label{eq:main_result}
    \end{eqnarray}
    where the Lyapunov function $V_k$ is defined by $V_k = \|x^k - x^{*,k}\|^2 + M\gamma^2 \sigma_k^2$ for all $k\ge 0$.
\end{theorem}

The above theorem states that \algname{SGDA} \eqref{eq:SGDA_general} converges linearly to the neighborhood of the solution. The size of the neighborhood is proportional to the noises $D_1$ and $D_2$. When $D_1 = D_2 = 0$, i.e., the method is variance reduced, it converges linearly to the exact solution in expectation. However, in general, to achieve any predefined accuracy, one needs to reduce the size of the neighborhood somehow. One possible way to do that is to use a proper stepsize schedule. We formalize this discussion in the following result.

\begin{corollary}\label{cor:main_result}
    Let the assumptions of Theorem~\ref{thm:main_result} hold. Consider two possible cases.
    
    \textbf{Case 1.} Let $D_1 = D_2 = 0$. Then, for any $K \ge 0$, $M = \nicefrac{2B}{\rho}$, and $\gamma = \min\left\{\nicefrac{1}{\mu}, \nicefrac{1}{2(A + \nicefrac{2BC}{\rho})}\right\}$, the iterates of \algname{SGDA}, given by \eqref{eq:SGDA_general}, satisfy:
         $  \Exp[V_K] \leq V_0 \exp\left(-\min\left\{\frac{\mu}{2(A + \nicefrac{2BC}{\rho})}, \frac{\rho}{2}\right\}K\right). \notag$
        
    \textbf{Case 2.} Let $D_1 + MD_2 > 0$. For any $K \ge 0$ and $M = \nicefrac{2B}{\rho}$ one can choose $\{\gamma_k\}_{k \ge 0}$ as follows:
\begin{equation*}
    \begin{split}
  \gamma_k &= \frac{1}{h} \quad \text{ if } K \le \frac{h}{\mu} \text{ or } \left(K > \frac{h}{\mu} \text{ and } k < k_0\right) , \\
	\gamma_k &= \frac{2}{\mu(\kappa + k - k_0)} \quad \text{ if } K > \frac{h}{\mu} \text{ and } k \ge k_0, 
\end{split}
\end{equation*}
	where $h = \max\left\{2(A + \nicefrac{2BC}{\rho}), \nicefrac{2\mu}{\rho}\right\}$, $\kappa = \nicefrac{2h}{\mu}$ and $k_0 = \left\lceil \nicefrac{K}{2} \right\rceil$. For this choice of $\gamma_k$, the iterates of \algname{SGDA}, given by \eqref{eq:SGDA_general}, satisfy:
	\begin{eqnarray}
		\Exp[V_K] \le \frac{32h V_0}{\mu}\exp\left(-\frac{\mu}{h}K\right) + \frac{36(D_1 + \nicefrac{2BD_2}{\rho})}{\mu^2 K}.\notag
	\end{eqnarray}
\end{corollary}

\paragraph{Monotone case.} When $\mu = 0$, we additionally assume that $F$ is monotone,
i.e., for all $x, y \in \R^d$
\begin{equation*}
\left\langle F(x) - F(y),  x-y\right\rangle \geq 0.
\end{equation*}
Similar to minimization, in the case of $\mu = 0$, the squared distance to the solution is not a valid measure of convergence. To introduce an appropriate convergence measure, we make the following assumption.

\begin{assumption}\label{as:boundness}
    There exists a compact convex set $\cC$ (with the diameter $\Omega_{\cC} \eqdef \max_{x,y \in \cC}\|x - y\|$) such that $X^* \subset \cC$.
\end{assumption}

In this setting, we focus on the following quantity called a restricted gap-function \citep{nesterov2007dual} defined for any $z \in \R^d$ and any $\cC \subset \R^d$ satisfying Assumption~\ref{as:boundness}: 
\begin{equation}
    \label{eq:gap}
    \text{Gap}_{\cC} (z) \eqdef \max_{u \in \mathcal{C}} \left[ \langle F(u),  z - u  \rangle + R(z) - R(u) \right].
\end{equation}

Assumption~\ref{as:boundness} and function $\text{Gap}_{\cC} (z)$ are standard for the convergence analysis of methods for solving \eqref{eq:VI} with monotone $F$ \citep{nesterov2007dual, alacaoglu2021stochastic}. Additional discussion is left to Appendix~\ref{sec:monotone_appendix}.

Under these assumptions, Assumption~\ref{as:key_assumption}, and star-cocoercivity we derive the following general result.

\begin{theorem}\label{thm:main_result_monotone}
        Let $F$ be monotone, $\ell$-star-cocoercive and let Assumptions~\ref{as:key_assumption}, \ref{as:boundness} hold. Assume that $ 0 < \gamma \leq \nicefrac{1}{2(A + \nicefrac{BC}{\rho})}.$ Then for all $K\ge 0$ the iterates of \algname{SGDA}, given by \eqref{eq:SGDA_general}, satisfy:
\begin{align}
    \label{eq:main_result_monotone}
    \Exp\Bigg[&\text{Gap}_{\cC} \left(\frac{1}{K}\sum\limits_{k=1}^{K}  x^{k}\right)\Bigg] 
    \notag\\
    \leq& \frac{3\left[\max_{u \in \mathcal{C}}\|x^{0} - u\|^2\right]}{2\gamma K} + 9\gamma\max\limits_{x^* \in X^*}\|F(x^*)\|^2 \notag\\
    &+ \frac{8\gamma \ell^2 \Omega_{\mathcal{C}}^2}{K} + \left( 4A + \ell + \nicefrac{8BC}{\rho}\right) \cdot \frac{\|x^0-x^{*,0}\|^2}{K} \notag\\
    &+ \left(4 +  \left( 4A + \ell + \nicefrac{8BC}{\rho}\right) \gamma \right) \frac{\gamma B \sigma_0^2}{\rho K} \notag\\
    &+\gamma (2 + \gamma\left( 4A + \ell + \nicefrac{8BC}{\rho}\right))(D_1 + \nicefrac{2BD_2}{\rho}).
\end{align}
\end{theorem}

The above result establishes $\cO(\nicefrac{1}{K})$ rate of convergence to the accuracy proportional to the stepsize $\gamma$ multiplied by the noise term $D_1 + \nicefrac{2BD_2}{\rho}$ and $\max_{x^*\in X^*} \|F(x^*)\|^2$. We notice that if $R \equiv 0$ in \eqref{eq:VI}, then $F(x^*) = 0$, meaning that in this case, the second term from \eqref{eq:main_result_monotone} equals zero. Otherwise, even in the deterministic case one needs to use small stepsizes to ensure the convergence to any predefined accuracy (see Corollary~\ref{cor:main_result_monotone} in Appendix~\ref{sec:monotone_appendix}). 

\paragraph{Cocoercive case.} The term proportional to $\max_{x^*\in X^*} \|F(x^*)\|^2$ can be removed if we assume that the operator $F$ is not just monotone star-cocoercive \eqref{eq:cocoercivity}, but general cocoercive, i.e., it holds that for all $x,y \in \R^d$
\begin{equation*}
    \|F(x) - F(y)\|^2 \le  \ell \langle F(x) - F(y), x - y \rangle.
\end{equation*}

\begin{theorem}\label{thm:main_result_monotone_coco}
    Let $F$ be $\ell$-cocoercive and Assumptions~\ref{as:key_assumption}, \ref{as:boundness} hold. Assume that 
    $
        0 < \gamma \leq \min\left\{\nicefrac{1}{\ell}, \nicefrac{1}{2(A + \nicefrac{BC}{\rho})}\right\}.
    $
    Then for all $K\ge 0$ the iterates of \algname{SGDA}, given by \eqref{eq:SGDA_general}, satisfy:
\begin{align}
    \label{eq:main_result_monotone_coco}
    \Exp\Bigg[&\text{Gap}_{\cC} \left(\frac{1}{K}\sum\limits_{k=1}^{K}  x^{k}\right)\Bigg]
    \notag\\
    \leq& \frac{3\max_{u \in \mathcal{C}}\|x^{0} - u\|^2}{2\gamma K}  \notag\\
    &+ \left( 6A + 3\ell + \nicefrac{12BC}{\rho}\right) \cdot \frac{\|x^0-x^{*,0}\|^2}{K} \notag\\
    & + \left(6 +  \left( 6A + 3\ell + \nicefrac{12BC}{\rho}\right) \gamma \right) \frac{\gamma B \sigma_0^2}{\rho K} \notag\\
    &
    +\gamma (3 + \gamma\left( 6A + 3\ell + \nicefrac{12BC}{\rho}\right))(D_1 + \nicefrac{2BD_2}{\rho}).
\end{align}
\end{theorem}
In contrast to Theorem~\ref{thm:main_result_monotone}, the above result implies $\cO(\nicefrac{1}{K})$ convergence rate in the deterministic case. See Corollary~\ref{cor:main_result_monotone_coco} in Appendix~\ref{sec:monotone_coco_appendix} for the results of the convergence with a selected stepsize.

\section{\algname{SGDA} WITH ARBITRARY SAMPLING}\label{sec:sgda}
We start our consideration of special cases with a standard \algname{SGDA} \eqref{eq:SGDA_general} with $g^k = F_{\xi^k}(x^k), \xi^k \sim \cD$ under so-called \textit{expected cocoercivity} assumption from \citet{loizou2021stochastic}, which we properly adjust to the setting of regularized VIPs.

\begin{assumption}[Expected Cocoercivity]\label{as:expected_cocoercivity}
    We assume that stochastic operator $F_{\xi}(x), \xi\sim \cD$ is such that for all $x \in \R^d$,
     $  \Exp_{\cD}\left[\|F_{\xi}(x) - F_{\xi}(x^*)\|^2\right] \le \ell_{\cD} \langle F(x) - F(x^*), x - x^* \rangle,$ 
    where $x^* = \text{proj}_{X^*}(x)$.
\end{assumption}

When $R(x) \equiv 0$, this assumption recovers the original one from \citet{loizou2021stochastic}. We also emphasize that for operator $F$ Assumption~\ref{as:expected_cocoercivity} implies only star-cocoercivity.

Following \citet{loizou2021stochastic}, we mainly focus on the finite-sum case and its stochastic reformulation: we consider a random \textit{sampling} vector $\xi = (\xi_1,\ldots,\xi_n)^\top \in \R^n$ having a distribution $\cD$ such that $\Exp_{\cD}[\xi_i] = 1$ for all $i\in [n]$. Using this we can rewrite $F(x) = \tfrac{1}{n}\sum_{i=1}^n F_i(x)$ as
\begin{equation}
    F(x) = \frac{1}{n}\sum\limits_{i=1}^n\Exp_{\cD}[\xi_i F_i(x)] = \Exp_{\cD}\left[F_\xi(x)\right],\label{eq:stoch_reformulation}
\end{equation}
where $F_\xi(x) = \frac{1}{n}\sum_{i=1}^n\xi_i F_i(x)$. Such a reformulation allows to handle a wide range of samplings: the only assumption on $\cD$ is $\Exp_{\cD}[\xi_i] = 1$ for all $i \in [n]$. Therefore, this setup is often referred to as \textit{arbitrary sampling} \citep{richtarik2020stochastic, loizou2020convergence, loizou2020momentum, gower2019sgd,gower2021sgd, hanzely2019accelerated, qian2019saga, qian2021svrg}. We elaborate on several special cases in Appendix~\ref{sec:arb_sampl_sp_cases}.

In this setting, \algname{SGDA} with Arbitrary Sampling (\algname{SGDA-AS})\footnote{For the pseudo-code of \algname{SGDA-AS} see Algorithm~\ref{alg:prox_SGDA} in Appendix~\ref{AppendixSGDA_AS}.} fits our framework.
\begin{proposition}\label{thm:prox_SGDA_convergence}
    Let Assumption~\ref{as:expected_cocoercivity} hold. Then, \algname{SGDA-AS} satisfies Assumption~\ref{as:key_assumption} with $A = \ell_{\cD}$, $D_1 = 2\sigma_*^2 \eqdef 2\max_{x^* \in X^*}\Exp_{\cD}\left[\|F_{\xi}(x^*) - F(x^*)\|^2\right]$, $B = 0$, $\sigma_k^2 \equiv 0$, $C = 0$, $\rho = 1$, $D_2 = 0$.
\end{proposition}

Plugging these parameters to Theorem~\ref{thm:main_result} we recover the result\footnote{In the main part of the paper, we focus on $\mu$-quasi strongly monotone case with $\mu > 0$. For simplicity, we provide here the rates of convergence to the exact solution. Further details, including the rates in monotone case, are left to the Appendix.} from \citet{loizou2021stochastic} when $R(x) \equiv 0$ and generalize it to the case of $R(x) \not\equiv 0$ without sacrificing the rate. Applying Corollary~\ref{cor:main_result}, we establish the rate of convergence to the exact solution.
\begin{corollary}\label{cor:prox_SGDA_convergence}
    Let $F$ be $\mu$-quasi-strongly monotone and Assumption~\ref{as:expected_cocoercivity} hold. Then for all $K > 0$ there exists a choice of $\gamma$ (see \eqref{eq:stepsize_choice_2_QSM_prox_SGDA}) for which the iterates of \algname{SGDA-AS}, satisfy:
	\begin{align*}
		\Exp[\|x^{K} &- x^{*,K}\|^2] 
        \\
        &= \cO\left(\frac{\ell_{\cD}\Omega_0^2}{\mu}\exp\left(-\frac{\mu}{\ell_{\cD}}K\right) + \frac{\sigma_{*}^2}{\mu^2 K}\right),\notag
	\end{align*}
	where $\Omega_0^2 = \|x^0 - x^{*,0}\|^2$.
\end{corollary}
For the different stepsize schedule, \citet{loizou2021stochastic} derive the convergence rate $\cO(\nicefrac{1}{K} + \nicefrac{1}{K^2})$ which is inferior to our rate, especially when $\sigma_*^2$ is small. In addition, \citet{loizou2021stochastic} consider explicitly only uniform minibatch sampling without replacement as a special case of arbitrary sampling. In Appendix~\ref{sec:arb_sampl_sp_cases}, we discuss another prominent sampling strategy called importance sampling. In Section~\ref{sec:numerical_exp}, we provide numerical experiments verifying our theoretical findings and showing the benefits of importance sampling over uniform sampling for SGDA.
\section{\algname{SGDA} WITH VARIANCE REDUCTION}\label{sec:vr_sgda}

In this section, we focus on variance-reduced variants of \algname{SGDA} for solving finite-sum problems $F(x) = \tfrac{1}{n}\sum_{i=1}^n F_i(x)$. We start with the Loopless Stochastic Variance Reduced Gradient Descent-Ascent (\algname{L-SVRGDA}), which is a generalization of the \algname{L-SVRG} algorithm proposed in \citet{hofmann2015variance, kovalev2019don}. \algname{L-SVRGDA} (see Alg.~\ref{alg:prox_L_SVRGDA}) follows the update rule \eqref{eq:SGDA_general} with
\begin{equation}
    \begin{split}
        g^k &= F_{j_k}(x^k) - F_{j_k}(w^k) + F(w^k), \\
    w^{k+1} &= \begin{cases} x^k, & \text{with prob.\ } p,\\ w^k,& \text{with prob.\ } 1-p,\end{cases}
    \end{split}
     \label{eq:L_SVRGDA_w^k}
\end{equation}
where in $k^{th}$ iteration $j_k$ is sampled uniformly at random from $[n]$. Here full operator $F$ is computed once $w^k$ is updated, which happens with probability $p$. Typically, $p$ is chosen as $p \sim \nicefrac{1}{n}$ ensuring that the expected cost of $1$ iteration equals $\cO(1)$ oracle calls, i.e., computations of $F_i(x)$ for some $i\in [n]$.

We introduce the following assumption about operators $F_i$.
\begin{assumption}[Averaged Star-Cocoercivity]\label{as:averaged_cocoercivity}
    We assume that there exists a constant $\widehat{\ell} > 0$ such that for all $x\in \R^d$
    \begin{equation}
        \frac{1}{n}\sum\limits_{i=1}^n\|\Delta_{F_i}(x,x^*)\|^2 \leq \widehat{\ell} \langle F(x) - F(x^*), x - x^*\rangle, \label{eq:averaged_cocoercivity}
    \end{equation}
    where $\Delta_{F_i}(x,x^*) = F_{i}(x) - F_{i}(x^*)$ and $x^* = \text{proj}_{X^*}(x)$.
\end{assumption}
For example, if $F_i$ is $\ell_i$-cocoercive for $i \in [n]$, then \eqref{eq:averaged_cocoercivity} holds with $\widehat{\ell} \leq \max_{i\in [n]}\ell_i$. Next, if $F_i$ is $L_i$-Lipschitz for all $i\in [n]$ and $F$ is $\mu$-quasi strongly monotone, then \eqref{eq:averaged_cocoercivity} is satisfied for $\widehat{\ell} \in [\overline{L}, \nicefrac{\overline{L}^2}{\mu}]$, where $\overline{L}^2 = \tfrac{1}{n}\sum_{i=1}^n L_i^2$.

Moreover, for the analysis of variance-reduced variants of \algname{SGDA} we also use the uniqueness of the solution.
\begin{assumption}[Unique Solution]\label{as:unique_solution}
    We assume that the solution set $X^*$ of problem \eqref{eq:VI} is a singleton: $X^* = \{x^*\}$.
\end{assumption}

These assumptions are sufficient to derive validity of Assumption~\ref{as:key_assumption} for \algname{L-SVRGDA} estimator. 

\begin{proposition}\label{thm:prox_SVRGDA_convergence}
    Let Assumptions~\ref{as:averaged_cocoercivity}~and~\ref{as:unique_solution} hold. Then, \algname{L-SVRGDA} satisfies Assumption~\ref{as:key_assumption} with $A = \widehat{\ell}$, $B = 2$, $\sigma_k^2 = \frac{1}{n}\sum_{i=1}^n\|F_i(w^k) - F_i(x^{*})\|^2$, $C = \nicefrac{p\widehat{\ell} }{2}$, $\rho = p$, $D_1 = D_2 = 0$.
\end{proposition}

Plugging these parameters in our general results on the convergence of \algname{SGDA}-type algorithms we derive the convergence results for \algname{L-SVRGDA}, see Table~\ref{tab:vr_methods} and Appendix~\ref{sec:l_svrgda} for the details. Moreover, in Appendix~\ref{sec:SAGA}, we show that \algname{SAGA-SGDA} \citep{palaniappan2016stochastic} fits our framework and using our general analysis we tighten the convergence rates for this method.

We compare our convergence guarantees  with known results in Table~\ref{tab:vr_methods}. We note that by neglecting importance sampling scenario, in the worst case, our convergence results match the best-known results for \algname{SGDA}-type methods, i.e., ones derived in \citet{palaniappan2016stochastic}. Indeed, this follows from $\widehat{\ell} \in [\overline{L}, \nicefrac{\overline{L}^2}{\mu}]$. Next, when the difference between $\overline{\ell}$ and $\widehat{\ell}$ is not significant, our complexity results match the one derived in \citet{chavdarova2019reducing} for \algname{SVRE}, which is \algname{EG}-type method. Although in general, $\overline{\ell}$ might be smaller than $\widehat{\ell}$, our analysis does not require cocoercivity of each $F_i$ and it works for $R(x) \not\equiv 0$. Finally, \citet{alacaoglu2021stochastic} derive a better rate (when $n = \cO(\nicefrac{\overline{L}^2}{\mu^2})$), but their method is based on \algname{EG}. Therefore, our results match the best-known ones in the literature on \algname{SGDA}-type methods.

\begin{table*}[t]
		\centering
		    \small
		\caption{\small Summary of the complexity results for variance reduced methods for solving \eqref{eq:VI}. By complexity we mean the number of oracle calls required for the method to find $x$ such that $\Exp[\|x - x^*\|^2] \leq \varepsilon$. Dependencies on numerical and logarithmic factors are hidden. By default, operator $F$ is assumed to be \textbf{$\mu$-strongly monotone} and, as the result, the solution is unique. Our results rely on \textbf{$\mu$-quasi strong monotonicity} of $F$ \eqref{eq:QSM}, but we also assume uniqueness of the solution. Methods supporting $R(x) \not\equiv 0$ are highlighted with $^*$. Our results are highlighted in green. Notation: $\overline{\ell}$, $\overline{L}$ = averaged cocoercivity/Lipschitz constants depending on the sampling strategy, e.g., for uniform sampling $\overline{\ell}^2 = \tfrac{1}{n}\sum_{i=1}^n \ell_i^2$, $\overline{L}^2 = \tfrac{1}{n}\sum_{i=1}^n L_i^2$ and for importance sampling $\overline{\ell} = \tfrac{1}{n}\sum_{i=1}^n \ell_i$, $\overline{L} = \tfrac{1}{n}\sum_{i=1}^n L_i$; $\widehat{\ell}$ = averaged star-cocoercivity constant from Assumption~\ref{as:averaged_cocoercivity}.}
		\label{tab:vr_methods}    
		\begin{threeparttable}
			\begin{tabular}{|c|c c c|}
			\hline
				Method & Citation & Assumptions & Complexity\\
				\hline
				\hline
				\algname{SVRE}\tnote{\color{blue}(1)} & \citep{chavdarova2019reducing} & $F_i$ is $\ell_i$-cocoer. & $n + \frac{\overline{\ell}}{\mu}$\\
				 \algname{EG-VR}$^*$\tnote{\color{blue}(1)} & \citep{alacaoglu2021stochastic} & $F_i$ is $L_i$-Lip. & $n + \sqrt{n}\frac{\overline{L}}{\mu}$\\
				\hline\hline
				\algname{SVRGDA}$^*$ & \citep{palaniappan2016stochastic} & $F_i$ is $L_i$-Lip. & $n + \frac{\overline{L}^2}{\mu^2}$\\
				\algname{SAGA-SGDA}$^*$ & \citep{palaniappan2016stochastic} & $F_i$ is $L_i$-Lip. & $n + \frac{\overline{L}^2}{\mu^2}$ \\
				\algname{VR-AGDA} & \citep{yang2020global} & $F_i$ is $L_{\max}$-Lip.\tnote{\color{blue}(2)} & $\min\left\{n + \frac{L_{\max}^9}{\mu^9}, n^{\nicefrac{2}{3}}\frac{L_{\max}^3}{\mu^3}\right\}$\\
			    \rowcolor{bgcolor2}\algname{L-SVRGDA}$^*$ & \textbf{This paper} & As.~\ref{as:averaged_cocoercivity} & $n + \frac{\widehat{\ell}}{\mu}$\\
			    \rowcolor{bgcolor2}\algname{SAGA-SGDA}$^*$ & \textbf{This paper} & As.~\ref{as:averaged_cocoercivity} & $n + \frac{\widehat{\ell}}{\mu}$\\
			    \hline
			\end{tabular}
			        {\small
					\begin{tablenotes}
					    \item [{\color{blue}(1)}] The method is based on Extragradient update rule.
					    \item [{\color{blue}(2)}] \citet{yang2020global} consider saddle point problems satisfying so-called two-sided P{\L} condition, which is weaker than strong-convexity-strong-concavity of the objective function.
					\end{tablenotes}}
		\end{threeparttable}
\vspace{-2mm}
\end{table*}

\section{DISTRIBUTED \algname{SGDA} WITH COMPRESSION}\label{sec:distr_sgda}

In this section, we consider the distributed version of \eqref{eq:VI}, i.e., we assume that $F(x) = \tfrac{1}{n}\sum_{i=1}^n F_i(x)$, where $\{F_i\}_{i=1}^n$ are distributed across $n$ devices connected with parameter-server in a centralized fashion. Each device $i$ has an access to the computation of the unbiased estimate of $F_i$ at the given point. Typically, in these settings, communication is a bottleneck, especially when $n$ and $d$ are huge. This means that in the naive distributed implementations of \algname{SGDA}, communication rounds take much more time than local computations on the clients. Various approaches are used to circumvent this issue. 

One of them is based on the usage of compressed communications. We focus on unbiased compression operators.
\begin{definition}\label{def:quantization}
Operator $\cQ:\R^d\to \R^d$ (possibly randomized) is called \textit{unbiased compressor/quantization} if there exists a constant $\omega \geq 0$ such that for all $x \in \R^d$
\begin{align}
\label{eq:quant}
    \Exp[\cQ(x)] = x,\quad \Exp[\| \cQ(x) - x \|^2] \leq \omega \| x\|^2.
\end{align}
\end{definition}

In this paper, we consider compressed communications in the direction from clients to the server. The simplest method with compression -- \algname{QSGDA} (Alg.~\ref{alg:prox_QSGDA}) -- can be described as \algname{SGDA} \eqref{eq:SGDA_general} with $g^k = \tfrac{1}{n} \sum_{i=1}^n \cQ(g^k_i)$. Here $g_i^k$ are stochastic estimators satisfying the following assumption\footnote{We use this assumption for illustrating the flexibility of the framework. It is possible to consider Arbitrary Sampling setup as well.}. 
\begin{assumption}[Bounded variance]\label{as:bounded_variance}
    All stochastic realizations $g^k_i$ are unbiased and have bounded variance, i.e., for all $i \in [n]$ and $k \ge 0$ the following holds:
    \begin{align}
    \label{eq:variance}
    \Exp[g^k_i] = F_i(x^k),\quad \Exp[\| g^k_i - F_i(x^k) \|^2] \leq \sigma_i^2.
\end{align}
\end{assumption}

Despite its simplicity, \algname{QSGDA} was never considered in the literature on solving min-max problems and VIPs. It turns out that under such assumptions \algname{QSGDA} satisfies our Assumption \ref{as:key_assumption}.
\begin{proposition} \label{thm:prox_QSGDA_convergence}
    Let $F$ be $\ell$-star-cocoercive and Assumptions~ \ref{as:averaged_cocoercivity},~\ref{as:bounded_variance} hold. Then, \algname{QSGDA} satisfies Assumption~\ref{as:key_assumption} with $A = \tfrac{3 \ell}{2} + \frac{9\omega \widehat\ell}{2n}$, $B = 0$, $\sigma_k^2 \equiv 0$, 
    $D_1 = \frac{3(1 + 3\omega)\sigma^2 + 9\omega\zeta_*^2}{n}$, $C=0$, $\rho = 1$, $D_2 = 0$, where $\sigma^2 = \tfrac{1}{n}\sum_{i=1}^n\sigma_i^2, \zeta_*^2 \eqdef \tfrac{1}{n}\max_{x* \in X^*}\sum_{i=1}^n \left\| F_i(x^{*})\right\|^2$.
\end{proposition}

As for the other special cases, we derive the convergence results for \algname{QSGDA} using our general theorems (see Table~\ref{tab:distrib_methods} and Appendix~\ref{sec:qsgda} for the details). The proposed method is simple, but has a significant drawback: even in the deterministic case ($\sigma = 0$), \algname{QSGDA} does not converge linearly unless $\zeta_*^2 = 0$. However, when the data on clients is arbitrarily heterogeneous the dissimilarity measure $\zeta_*^2$ is strictly positive and can be large (even when $R(x) \equiv 0$). 

To resolve this issue, we propose a more advanced scheme based on \algname{DIANA} update \citep{mishchenko2019distributed,horvath2019stochastic} -- \algname{DIANA-SGDA} (Alg.~\ref{alg:prox_DIANA_SGDA}). In a nutshell, \algname{DIANA-SGDA} is \algname{SGDA} \eqref{eq:SGDA_general} with $g^k$ defined as follows:
\begin{align}
\label{eq:DIANA_SGDA_update_1}
\begin{split}
    \Delta^k_i &= g^k_i - h_i^k, \quad h^{k+1}_i = h_i^k + \alpha \cQ(\Delta^k_i), \\
g^k &= h^k + \frac{1}{n} \sum\limits_{i=1}^n \cQ(\Delta^k_i), \\
h^{k+1} &= \frac{1}{n} \sum\limits_{i=1}^n h^{k+1}_i = h^k + \alpha \frac{1}{n} \sum\limits_{i=1}^n \cQ(\Delta^k_i),
\end{split}
\end{align}
where the first two lines correspond to the local computations on the clients and the last two lines -- to the server-side computations. Taking into account the update rule for $h^{k+1}$, one can notice that \algname{DIANA-SGDA} requires workers to send only vectors $\cQ(\Delta_i^k)$ to the server at step $k$, i.e., the method uses only compressed workers-server communications.

As we show next, \algname{DIANA-SGDA} fits our framework.
\begin{proposition}\label{thm:prox_DIANA_convergence}
    Let Assumptions~\ref{as:averaged_cocoercivity}, \ref{as:unique_solution},  \ref{as:bounded_variance} hold. Suppose that $\alpha \leq \nicefrac{1}{(1+\omega)}$. Then, \algname{DIANA-SGDA} with quantization \eqref{eq:quant} satisfies Assumption~\ref{as:key_assumption} with $\sigma_k^2 = \tfrac{1}{n}\sum_{i=1}^n \|h^k_i - F_i(x^*) \|^2$ and $A = \left(\tfrac{1}{2} + \tfrac{\omega}{n}\right)\widehat \ell$, $B = \tfrac{2\omega}{n}$, $D_1 = \frac{(1 + \omega)\sigma^2}{n}$, $C = \frac{\alpha \widehat \ell}{2}$, $\rho = \alpha$, $D_2 = \alpha \sigma^2$, where $\sigma^2 = \tfrac{1}{n}\sum_{i=1}^n\sigma_i^2$.
\end{proposition}

\algname{DIANA-SGDA} can be considered as a variance-reduced method since it reduces the term proportional to $\omega\zeta_*^2$ that the bound for \algname{QSGDA} contains (see Table~\ref{tab:distrib_methods} and Appendix~\ref{sec:diana} for the details). As the result, when $\sigma = 0$, i.e., workers compute $F_i(x)$ at each step, \algname{DIANA-SGDA} enjoys linear convergence to the exact solution.

Next, when local operators $F_i$ have a finite-sum form $F_i(x) = \tfrac{1}{m}\sum_{j=1}^m F_{ij}(x)$, one can combine \algname{L-SVRGDA} and \algname{DIANA-SGDA} as follows: consider the scheme from \eqref{eq:DIANA_SGDA_update_1} with 
\begin{equation}
    \begin{split}
   g_i^k = F_{ij_k}(x^k) - F_{ij_k}(w^k) + F(w_i^k), 
    \\
    w_i^{k+1} = \begin{cases} x^k, & \text{with prob.\ } p,\\ w_i^k,& \text{with prob.\ } 1-p,\end{cases} 
\end{split}
     \label{eq:VR_DIANA_w^k}
\end{equation}
where $j_k$ is sampled uniformly at random from $[m]$. We call the resulting method \algname{VR-DIANA-SGDA} (Alg.~\ref{alg:prox_VR_DIANA_SGDA}) and we note that its analog for solving minimization problems (\algname{VR-DIANA}) was proposed and analyzed in \citet{horvath2019stochastic}.

To cast \algname{VR-DIANA-SGDA} as a special case of our general framework, we need to make the following assumption.
\begin{assumption}\label{as:sum_averaged_cocoercivity}
    We assume that there exists a constant $\widetilde \ell > 0$ such that for all $x\in \R^d$
    \begin{equation}
        \frac{1}{nm}\!\sum\limits_{i,j=1,1}^{n,m} \!\|\!\Delta_{F_{ij}}(x,x^*)\! \|^2 \leq \widetilde \ell \langle F(x)\! - \!F(x^*), x \!-\! x^*\rangle, \label{eq:sum_averaged_cocoercivity}
    \end{equation}
    where $\Delta_{F_{ij}}(x,x^*) = F_{ij}(x) - F_{ij}(x^*)$, $x^* = \text{proj}_{X^*}(x)$.
\end{assumption}

Using Assumption~\ref{as:sum_averaged_cocoercivity} and previously introduced conditions, we get the following result.
\begin{proposition}\label{thm:prox_VR_DIANA_convergence}
    Let $F$  be $\ell$-star-cocoercive and Assumptions~\ref{as:averaged_cocoercivity}, \ref{as:unique_solution}, \ref{as:sum_averaged_cocoercivity} hold. Suppose that $\alpha \leq \min\left\{\tfrac{p}{3},\tfrac{1}{1+\omega}\right\}$.  Then, \algname{VR-DIANA-SGDA} satisfies Assumption~\ref{as:key_assumption} with 
    $A = \tfrac{\ell}{2} + \tfrac{\widetilde \ell}{n} + \tfrac{\omega (\widehat \ell + \widetilde \ell)}{n}$, $B = \tfrac{2(\omega+1)}{n}, \sigma_k^2 = \tfrac{1}{n} \sum_{i=1}^n \| h^k_i - F_i (x^*) \|^2 + \tfrac{1}{nm}\sum_{i=1}^n \sum_{j=1}^m \| F_{ij} (w^k_{i}) - F_{ij} (x^*) \|^2
	$, $C = \tfrac{p \widetilde l}{2} + \alpha(\widetilde \ell + \widehat \ell)$, $\rho = \alpha$,
    $D_1 = D_2 = 0$.
\end{proposition}

Since $D_1 = D_2 = 0$, our general results imply linear convergence of \algname{VR-DIANA-SGDA} when $\mu > 0$ (see the details in Appendix~\ref{sec:vr_diana_sgda}). That is, \algname{VR-DIANA-SGDA} is the \textit{first linearly converging distributed \algname{SGDA}-type method with compression}. We compare it with \algname{MASHA1} \citep{beznosikov2021distributed} in Table~\ref{tab:distrib_methods}. Firstly, let us note that \algname{MASHA1} is a method based on \algname{EG}, and its convergence guarantees depend on the Lipschitz constants. In addition, we note that the complexity of \algname{MASHA1} could be better than the one of \algname{VR-DIANA-SGDA} when cocoercivity constants are large compared to Lipschitz ones. However, our compleixty bound has better dependency on quantization parameter $\omega$, number of clients $n$, and the size of the local dataset $m$. These parameters can be large meaning that the improvement is noticeable.

\begin{table*}[t]
		\centering
		\small
		\captionof{table}{\small Summary of the complexity results for distributed methods with unbiased compression for solving distributed \eqref{eq:VI} with $F = \tfrac{1}{n}\sum_{i=1}^n F_i(x)$. By complexity we mean the number of communication rounds required for the method to find $x$ such that $\Exp[\|x - x^*\|^2] \leq \varepsilon$. Dependencies on numerical and logarithmic factors are hidden. $\Exp$ stands for the setup, when $F_i(x) = \Exp_{\xi_i}[F_{\xi_i}(x)]$; $\Sigma$ denotes the case, when $F_i(x) = \tfrac{1}{m}\sum_{j=1}^m F_{ij}(x)$. 
		Our results rely on \textbf{$\mu$-quasi strong monotonicity} of $F$ \eqref{eq:QSM}, but we also assume the uniqueness of the solution. Methods supporting $R(x) \not\equiv 0$ are highlighted with $^*$. Our results are highlighted in green. Notation: $\sigma^2 = \tfrac{1}{n}\sum_{i=1}^n\sigma_i^2$ -- averaged upper bound for the variance (see Ass.~\ref{as:bounded_variance} for the definition of $\sigma_i^2$); $\omega$ = quantization parameter (see Def.~\ref{def:quantization}); $\zeta_*^2 = \tfrac{1}{n}\max_{x* \in X^*}\sum_{i=1}^n \left\| F_i(x^{*})\right\|^2$; $L_{\max} = \max_{i\in [n]}L_i$; $\widetilde\ell =$ averaged star-cocoercivity constant from Ass.~\ref{as:sum_averaged_cocoercivity}.}
		\label{tab:distrib_methods}    
		\begin{threeparttable}
			\begin{tabular}{|c|c|c c c|}
			\hline
				Setup & Method & Citation & Assumptions & Complexity\\
				\hline
				\hline
				\multirow{2.5}{0.7cm}{\centering ${\Exp}$}&\cellcolor{bgcolor2}\algname{QSGDA}$^*$ &\cellcolor{bgcolor2} \textbf{This paper} &\cellcolor{bgcolor2} As.~\ref{as:averaged_cocoercivity},~\ref{as:bounded_variance} &\cellcolor{bgcolor2} $\frac{\ell}{\mu} + \frac{\omega\widehat{\ell}}{n\mu} + \frac{(1+\omega)\sigma^2 + \omega \zeta_*^2}{n\mu^2\varepsilon}$\\
			    &\cellcolor{bgcolor2}\algname{DIANA-SGDA}$^*$ &\cellcolor{bgcolor2} \textbf{This paper} &\cellcolor{bgcolor2} As.~\ref{as:averaged_cocoercivity},~\ref{as:bounded_variance} & \cellcolor{bgcolor2}$\omega + \frac{\ell}{\mu} + \frac{\omega\widehat{\ell}}{n\mu} + \frac{(1+\omega)\sigma^2}{n\mu^2\varepsilon}$\\
				\hline\hline
			    \multirow{4}{0.7cm}{\centering $\Sigma$}&\algname{MASHA1}$^*$\tnote{{\color{blue}(1)}} & \citep{beznosikov2021distributed} & $F_i$ is $L_i$-Avg.\ Lip.\tnote{\color{blue}(2)} & \begin{tabular}{c}
			         $m +\omega +\frac{L_{\max}\sqrt{\left(m+\omega\right)\left(1 + \frac{\omega}{n}\right)}}{\mu}$
			    \end{tabular}\\
				&\cellcolor{bgcolor2}\algname{VR-DIANA-SGDA}$^*$ & \cellcolor{bgcolor2}\textbf{This paper} & \cellcolor{bgcolor2}As.~\ref{as:averaged_cocoercivity}, \ref{as:sum_averaged_cocoercivity} &\cellcolor{bgcolor2} \begin{tabular}{c}
				    $m + \omega + \frac{\ell}{\mu} + \frac{(1+\omega)(\widehat\ell + \widetilde\ell)}{n\mu}$\\
				    $+ \frac{(1+\omega)\max\{m, \omega\}\widetilde{\ell}}{nm\mu}$
				\end{tabular}\\
			    \hline
			\end{tabular}
			        {\small
					\begin{tablenotes}
					    \item [{\color{blue}(1)}] The method is based on Extragradient update rule.
					    \item [{\color{blue}(2)}] This means that for all $x,y \in \R^d$ and $i\in [n]$ the following inequality holds: $\tfrac{1}{m}\sum_{j=1}^m \|F_{ij}(x) - F_{ij}(y)\|^2 \leq L_i^2\|x-y\|^2$.
					\end{tablenotes}}
		\end{threeparttable}
\vspace{-2mm}
\end{table*}
\section{NUMERICAL EXPERIMENTS}
\label{sec:numerical_exp}
To illustrate our theoretical results, we conduct several numerical experiments on quadratic games, which are defined through the affine operator: $F(x) = \frac{1}{n}\sum_{i=1}^n \mA_i x + b_i$, 
where each matrix $\mA_i \in \R^{d\times d}$ is non-symmetric with all eigenvalues having strictly positive real parts. Enforcing all the eigenvalues to have strictly positive real part ensures that the operator is strongly monotone and cocoercive.
We consider two different settings: (i) problem without constraints, and (ii) problem that has $\ell_1$ regularization and constraints forcing the solution to lie in the $\ell_\infty$-ball of radius $r$. In all experiments, we use a constant stepsize for all methods which was selected manually using a grid search and picking the best-performing stepsize for each method. For further details about the experiments and additional experiments see Appendix~\ref{app:experiment_details}.

\textbf{Uniform sampling (\algname{US}) vs Important sampling (\algname{IS}).} We note that \citet{loizou2021stochastic} which studies \algname{SGDA-AS} does not consider \algname{IS} explicitly. Although we show the theoretical benefits of \algname{IS} in comparison to \algname{US} in Appendix~\ref{sec:arb_sampl_sp_cases}, here we provide a numerical comparison to illustrate the superiority of \algname{IS} (on both constrained and unconstrained quadratic games). We choose the matrices $\mA_i$ such that $\ell_{\max} = \max_i \ell_i \gg \bar \ell$. In this case, our theory predicts that \algname{IS} should perform better than \algname{US}. We provide the results in Fig.~\ref{fig:us_vs_is_1}. We observe that indeed \algname{SGDA} with \algname{IS} converges faster and to a smaller neighborhood than \algname{SGDA} with \algname{US}. This observation perfectly corroborates our theory.

\textbf{Comparison of variance reduced methods.} In this experiment, we test the performance of our proposed \algname{L-SVRGDA}~(Alg.~\ref{alg:prox_L_SVRGDA}) and compare it to other variance-reduced methods on quadratic games, see Fig.~\ref{fig:us_vs_is_2}. In particular, we compare it to \algname{SVRG} \citep{palaniappan2016stochastic}, \algname{SVRE} \citep{chavdarova2019reducing}, \algname{EG-VR} \citep{alacaoglu2021stochastic} and \algname{VR-AGDA} \citep{yang2020global}. In the constrained setting, we only compare \algname{L-SVRGDA} to \algname{SVRG} and \algname{EG-VR}, since they are the only methods from this list that handle constrained settings. For loopless variants, we choose $p = \tfrac{1}{n}$ and for the non-loopless variants we pick the number of inner-loop iterations to be $n$.  We observe that all  methods converge linearly and that \algname{L-SVRGDA} is competitive with the other considered variance-reduced methods, converging slightly faster than all of them.

We point out that we plot the distance to optimality as a function of the number of oracle calls. When using variance-reduced methods we sometimes have to compute the full-batch gradient, and thus have to make $n$ oracle calls. This is why we observe ``steps'' for variance-reduced methods in Fig.~\ref{fig:us_vs_is_2}: we observe a ``step'' every time the full batch gradient is computed.

\textbf{Comparison of distributed methods.} In our last experiment, we consider a distributed version of the quadratic game, in which we assume that $F(x) = \tfrac{1}{n}\sum_{i=1}^n F_i(x)$ with each $\{F_i\}_{i=1}^n$ being constructed similarly to the previous experiments. The information about operator $F_i$ is stored on node $i$ only. We compare the distributed methods proposed in the paper: \algname{QSGDA}, \algname{DIANA-SGDA}, and \algname{VR-DIANA-SGDA}. For the quantization, we use the \algname{RandK} sparsification \citep{beznosikov2020biased} with $K=5$. We show our findings in Fig.~\ref{fig:distributed_exp}, where the performance is measured both in terms of the number of oracle calls and the number of bits communicated from workers to the server. In both figures, we can clearly see the advantage of using quantization in terms of reducing the communication cost compared to the baseline \algname{SGDA}. We also observe that \algname{VR-DIANA-SGDA} achieves linear convergence to the solution. Additional experiments are deferred to Appendix~\ref{app:experiment_details}.



\section*{Acknowledgments}

The research of A. Beznosikov has been supported by The Analytical Center for the Government of the Russian Federation (Agreement No. 70-2021-00143 dd. 01.11.2021, IGK 000000D730321P5Q0002).

\begin{figure*}[h]
    \centering
    \begin{subfigure}[b]{0.99\columnwidth}
         \centering
    \includegraphics[width=0.95\textwidth]{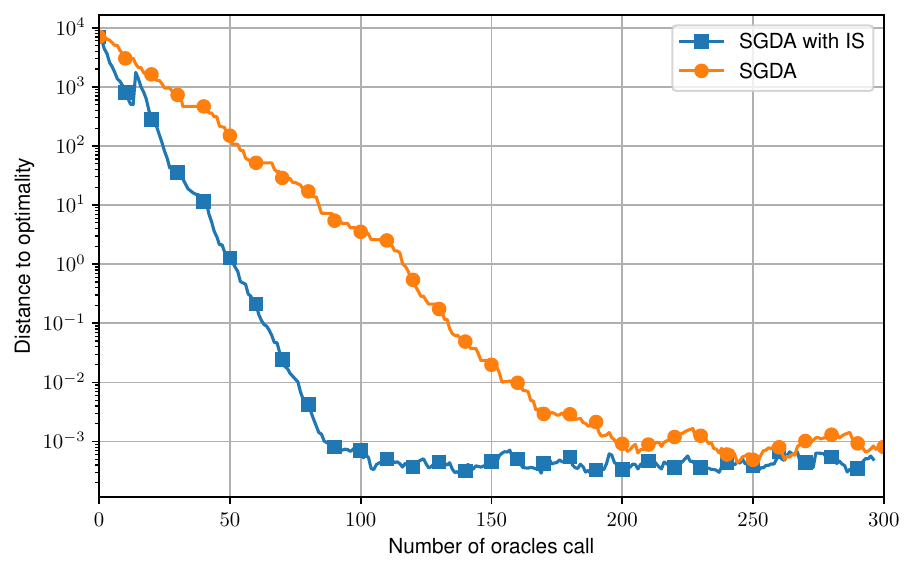}
     \end{subfigure}
     \begin{subfigure}[b]{0.99\columnwidth}
         \centering
    \includegraphics[width=0.95\textwidth]{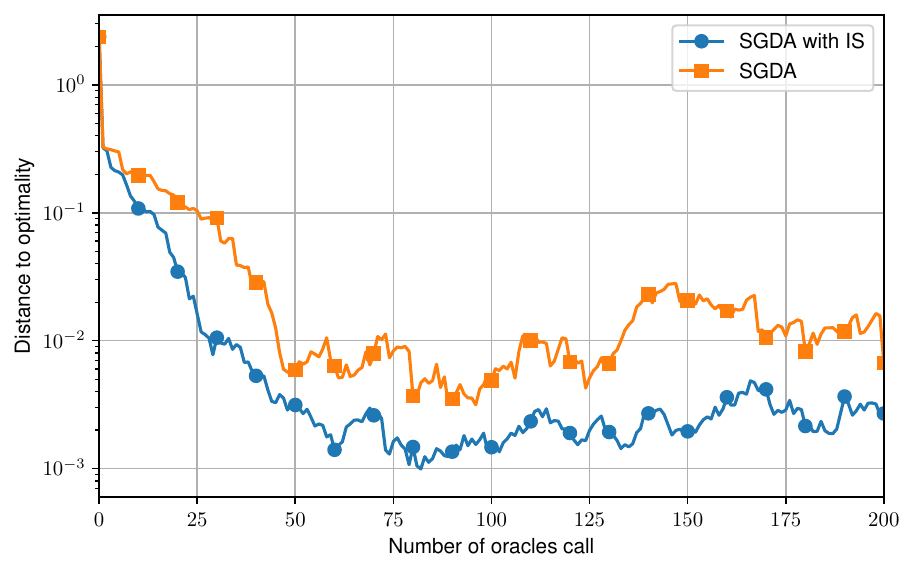}
     \end{subfigure}
    \caption{\small Comparison of Uniform Sampling (\algname{US}) vs Importance Sampling (\algname{IS}). \textbf{Left:} the result for the problem without constraints, \textbf{right:} with constraints. As expected by theory \algname{IS} converges faster and to a smaller neighborhood than \algname{US}.}
    \label{fig:us_vs_is_1}
\end{figure*}
\begin{figure*}[h]
    \centering
    \begin{subfigure}[b]{0.99\columnwidth}
         \centering
    \includegraphics[width=\textwidth]{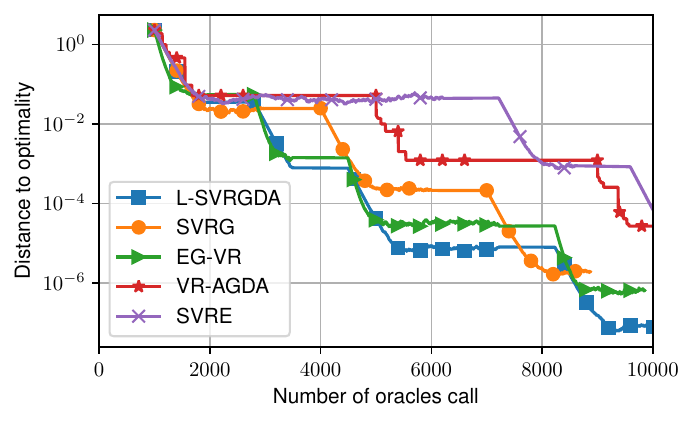}
     \end{subfigure}
     \begin{subfigure}[b]{0.99\columnwidth}
         \centering
    \includegraphics[width=\textwidth]{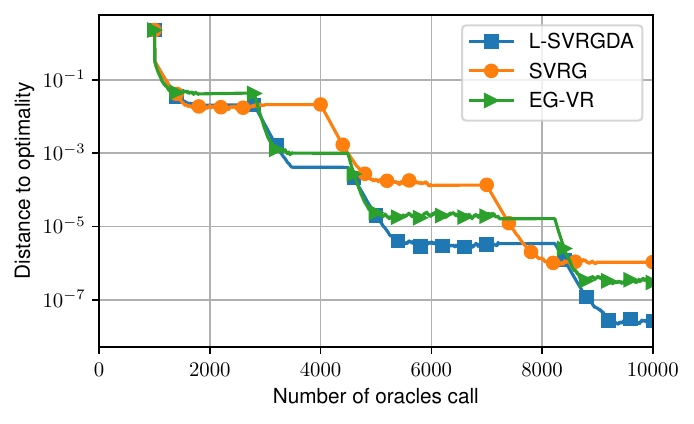}
     \end{subfigure}
    \caption{\small Comparison of variance-reduced methods. \textbf{Left:} the result for the problem without constraints, \textbf{right:} with constraints. Note that \algname{L-SVRGDA} is very competitive, and outperforms all the other methods.}
    \label{fig:us_vs_is_2}
\end{figure*}
\begin{figure*}[h!]
    \centering
    \begin{subfigure}[b]{0.99\columnwidth}
         \centering
    \includegraphics[width=\textwidth]{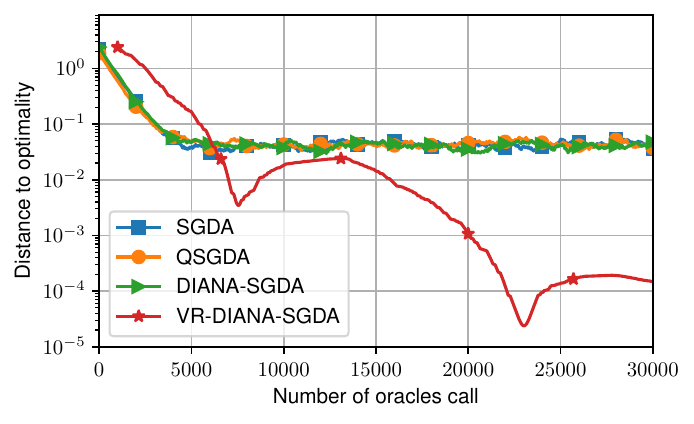}
     \end{subfigure}
     \begin{subfigure}[b]{0.99\columnwidth}
         \centering
    \includegraphics[width=\textwidth]{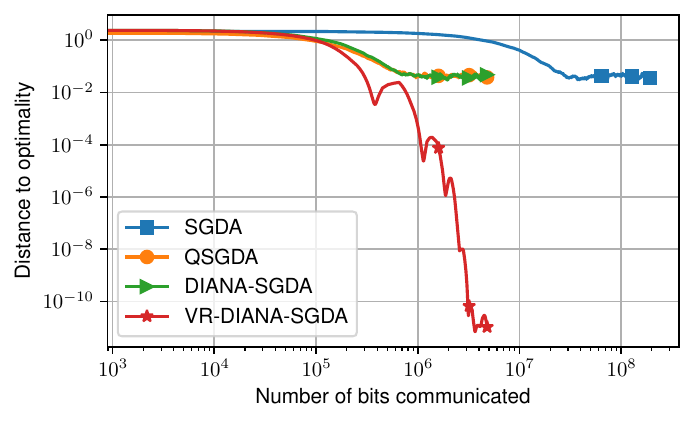}
     \end{subfigure}
    \caption{\small Comparison of algorithms in distributed setting. \textbf{Left:} number of oracle calls, \textbf{right:} number of bits communicated.}
    \label{fig:distributed_exp}
\end{figure*}

\bibliography{biblio1}

\newpage
\appendix
\onecolumn

 \aistatstitle{Stochastic Gradient Descent-Ascent: Unified Theory and New Efficient Methods\\ Supplementary Materials}

{\small\tableofcontents}

\newpage

\section{FURTHER RELATED WORK}
\label{AppendixRelatedWork}


The references necessary to motivate our work and connect it to the most relevant literature are included
in the appropriate sections of the main body of the paper. Here we present a broader view of the
literature, including some more references to papers of the area that are not directly related with our work.

\paragraph{Variants of the key assumption in prior work \& Detailed comparison to our results.} Here we would like to provide more details on the comparison with the closely related works \citep{gorbunov2020unified, gorbunov2021stochastic, loizou2021stochastic}. 

As we mention in the main part of the paper, \citet{gorbunov2020unified} focus on solving the much simpler minimization problems using \algname{SGD}. In particular, their Assumption 4.1 requires a function suboptimality (or Bregman divergence) for the upper bound, a concept that cannot be used in VI problems (there are no functions). Thus, the difference of the two notions does not solely lie on the norm bound, but begins at the deeper, conceptual level. In addition, we focus also on monotone VIs (non-quasi-strongly monotone), while \citet{gorbunov2020unified} consider only the class of quasi-strongly convex minimization problems. 

Next, \citet{gorbunov2021stochastic} provide convergence guarantees for vanilla \algname{SEG} under the arbitrary sampling paradigm. Their analysis is not able to capture \algname{SEG} with variance reduction, quantization, and coordinate-wise randomization. In contrast, our approach covers variants of \algname{SGDA} with variance reduction, quantization and coordinate-wise randomization. We are able to capture these more advanced variants by using sequence $\{\sigma_k^2\}_{k\ge 0}$ (see \eqref{eq:sigma_k_bound}) in our key assumption, and this is a major difference between our approach and the approach of \citet{gorbunov2021stochastic}. In addition, our analysis works for the case $R(x) \not\equiv 0$. Although the generalization of the analysis to the case of non-zero $R$ might be trivial in the quasi-strongly monotone case, for the monotone case this is definitely not straightforward. Finally, for the monotone case, we do not require large batch-sizes to achieve any predefined accuracy, while analysis of \algname{SEG} in \citep{gorbunov2021stochastic} does (see Appendix B in their work). 

Finally, we highlight again that \citet{loizou2021stochastic} focus only on uniform minibatch \algname{SGDA} for solving quasi-strongly monotone problems. This is only a special case of our approach (see Section~\ref{sec:sgda}). We note that even in this scenario, through our analysis  we were able to provide faster convergence by considering \algname{SGDA} with importance sampling (see Appendix~\ref{sec:arb_sampl_sp_cases} and Fig.~\ref{fig:us_vs_is_1}).

\paragraph{Stochastic methods for solving VIPs.} Although this paper is devoted to \algname{SGDA}-type methods, we briefly mention here the works studying other popular stochastic methods for solving VIPs based on different algorithmic schemes such as Extragradient (\algname{EG}) method \citep{korpelevich1976extragradient} and Optimistic Gradient (\algname{OG}) method \citep{popov1980modification}. The first analysis of Stochastic \algname{EG} for solving (quasi-strongly) monotone VIPs was proposed in \citet{juditsky2011solving} and then was extended and generalized in various ways \citep{mishchenko2020revisiting, hsieh2020explore, beznosikov2020distributed, li2021convergence, gorbunov2021stochastic}. Stochastic \algname{OG} was studied in \citet{gidel2018variational, hsieh2019convergence, azizian2021last}. In addition, lightweight second-order methods  like stochastic Hamiltonian methods and stochastic consensus optimization were studied in \cite{loizou2020stochastic}, and \cite{loizou2021stochastic}, respectively.

\paragraph{Analysis of \algname{SGDA}.} \algname{SGDA} is usually analyzed under uniformly bounded variance assumption. That is,  $\Exp[\|g^k - F(x^k)\|^2\mid x^k] \leq \sigma^2$ is typically assumed to get convergence
guarantees \citep{Nemirovski-Juditsky-Lan-Shapiro-2009, mertikopoulos2019learning, yang2020global}. This assumption rarely holds, especially for unconstrained VIPs: it is easy to construct an example of \eqref{eq:VI} with $F$ being a finite sum of linear operators such that the variance is unbounded. \citet{lin2020finite} provide a convergence analysis of SGDA under a relative random noise assumption allowing to handle some special cases not covered by uniformly bounded variance assumption. However, relative noise is also a quite strong assumption and usually requires a special type of noise appearing in coordinate methods\footnote{For example, see inequality \eqref{eq:bhsjcbhjcbjs} from Appendix~\ref{sec:coord_sgda} in the case when there is no regularization term, i.e., when $R(x) \equiv 0$ and, as a result, $F(x^*) = 0$ for all $x^* \in X^*$.} or in the training of overparameterized models \citep{vaswani2018fast}.  
In their recent work, \citet{loizou2021stochastic} proposed a new weak condition called expected cocoercivity. This assumption fits our theoretical framework (see Section~\ref{sec:sgda}) and does not imply strong conditions on the variance of the stochastic estimator but it is stronger than star-cocoercivity of operator $F$.

\paragraph{Variance reduction for VIPs.} The first variance-reduced variants of \algname{SGDA} (\algname{SVRGDA} and \algname{SAGA-SGDA} -- analogs of \algname{SVRG} \citep{johnson2013accelerating} and \algname{SAGA} \citep{defazio2014saga}) for solving \eqref{eq:VI} with strongly monotone operator $F$ having a finite-sum form with Lipschitz summands were proposed in \citet{palaniappan2016stochastic}. For two-sided PL min-max problems without regularization \citet{yang2020global} proposed a variance-reduced version of \algname{SGDA} with alternating updates. Since the considered class of problems includes non-strongly-convex-non-strongly-concave min-max problems, the rates from \citet{yang2020global} are inferior to \citet{palaniappan2016stochastic}. 
There are also several works studying variance-reduced methods based on different methods rather than \algname{SGDA}. \citet{chavdarova2019reducing} proposed a combination of \algname{SVRG} and Extragradient (\algname{EG}) \citep{korpelevich1976extragradient} called \algname{SVRE} and analyzed the method for strongly monotone VIPs without regularization and with cocoercive summands $F_i$. The cocoercivity assumption was relaxed to averaged Lipschitzness in \citet{alacaoglu2021stochastic}, where the authors proposed another variance-reduced version of \algname{EG} (\algname{EG-VR}) based on Loopless variant of \algname{SVRG} \citep{hofmann2015variance, kovalev2019don}. \citet{loizou2020stochastic} studied stochastic
Hamiltonian gradient descent (SHGD), and propose the first
stochastic variance reduced Hamiltonian method, named L-SVRHG, for solving stochastic bilinear games and and stochastic games satisfying a “sufficiently bilinear” condition. Moreover, \citet{loizou2020stochastic} provided the first set of global non-asymptotic last-iterate convergence guarantees for a stochastic game over a non-compact domain, in the absence of strong
monotonicity assumptions.

We should highlight that the rates from \citet{alacaoglu2021stochastic} match the lower bounds from \citet{han2021lower}. Under additional assumptions similar results were achieved in \citet{carmon2019variance}. \citet{alacaoglu2021forward} developed variance-reduced method (\algname{FoRB-VR}) based on Forward-Reflected-Backward algorithm \citep{malitsky2020forward}, but the derived rates are inferior to those from \citet{alacaoglu2021stochastic}. 

Using Catalyst acceleration framework of \citet{lin2018catalyst}, \citet{palaniappan2016stochastic, tominin2021accelerated} achieve (neglecting extra logarithmic factors) similar rates as in  \citet{alacaoglu2021stochastic} and \citet{luo2021near} derive even tighter rates for min-max problems. However, as all Catalyst-based approaches, these methods require solving an auxiliary problem at each iteration, which reduces their practical efficiency.

\paragraph{Communication compression for VIPs.} While distributed methods with compression were extensively studied for solving minimization problems both for unbiased compression operators \citep{alistarh2017qsgd, wen2017terngrad, mishchenko2019distributed, horvath2019stochastic, li2020acceleration,khaled2020unified, pmlr-v139-gorbunov21a} 
and biased compression operators \citep{seide20141, stich2018sparsified, karimireddy2019error, beznosikov2020biased, gorbunov2020linearly, qian2020error, EF21}, 
much less is known for min-max problems and VIPs. To the best of our knowledge, the first work on distributed methods with compression for min-max problems is \citet{yuan2014dual}, where the authors proposed a distributed version of Dual Averaging \citep{nesterov2009primal} with rounding and showed a convergence to the neighborhood of the solution that cannot be reduced via standard tricks like increasing the batchsize or decreasing the stepsize. More recently, \citet{beznosikov2021distributed} proposed new distributed variants of \algname{EG} with unbiased/biased compression for solving \eqref{eq:VI} with (strongly) monotone and Lipschitz operator $F$. \citet{beznosikov2021distributed} obtained the first linear convergence guarantees on distributed VIPs with compressed communication.

\paragraph{On quasi-strong monotonicity and star-cocoercivity.} 

In this work we focus on quasi-strongly monotone VI problems, a class of structured non-monotone
operators for which we are able to provide tight convergence guarantees and avoid the standard issues
(cycling and divergence of the methods) appearing in the more general non-monotone regime.  

Since in general non-monotone problems, finding approximate first-order locally optimal solutions is intractable \citep{daskalakis2021complexity, diakonikolas2021efficient}, it is reasonable to consider class of problems
that satisfy special structural
assumptions on the objective function for which these intractability barriers can be bypassed. Examples of problems belong in this category are the ones of our work which satisfy \eqref{eq:QSM} or, for example, the two-sided PL condition \citep{yang2020global} or the error-bound condition \citep{hsieh2020explore}. It is worth highlighting that quasi-strong monotone problems were considered in \citet{mertikopoulos2019learning, song2020optimistic, loizou2021stochastic, gorbunov2021stochastic} as well.

Cocoercivity is a classical assumption in the literature on VIPs \citep{zhu1996co} and operator splittings \citep{davis2017three, vu2013splitting}. It can be interpreted as an intermediate notion between monotonicity and strong monotonicity. In general, it is stronger than monotonicity and Lipschitzness of the operator, e.g., simple bilinear games are non-cocoercive. From Cauchy-Swartz's inequality, one can show that a $\ell$-co-coercive operator is $\ell$-Lipschitz.
In single-objective minization, one can prove the converse statement by using convex duality. Thus, a gradient of a function is $L$--co-coercive if and only if the function is convex and $L$-smooth (i.e. $L$-Lipschitz gradients)~\citep{bauschke2011convex}. However, in general, a $L$-Lipchitz operator is \emph{not} $L$--co-coercive.
 Star-cocoercivity is a new notion recently introduced in  \cite{loizou2021stochastic} and is weaker than classical cocoercivity and can be achieved via a proper transformation of quasi-monotone Lipschitz operator \citep{gorbunov2021extragradient}. Moreover, any $\mu$-quasi strongly monotone $L$-Lipschitz operator $F$ is $\ell$-star-cocoercive with $\ell \in [L, \nicefrac{L^2}{\mu}]$ and there exist examples of operators that are quasi-strongly monotone and star-cocoercive but neither monotone nor Lipschitz \citep{loizou2021stochastic}.

\paragraph{Coordinate and zeroth-order methods for solving min-max problems and VIPs.} Coordinate methods for solving VIPs are rarely considered in the literature. The most relevant results are given in the literature on zeroth-order methods for solving min-max problems. Although some of them can be easily extended to the coordinate versions of methods for solving VIPs, these methods are usually considered and analyzed for min-max problems. The closest work to our paper is \citet{sadiev2020zeroth}: they propose and analyze several zeroth-order variants of \algname{SGDA} and Stochastic \algname{EG} with two-point feedback oracle for solving strongly-convex-strongly-concave and convex-concave smooth min-max problems with bounded domain. Moreover, \citet{sadiev2020zeroth} consider firmly smooth convex-concave min-max problems which is an analog of cocoercivity for min-max problems. There are also papers focusing on different problems like non-sonvex-strongly-concave smooth min-max problems \citep{liu2020min, wang2020zeroth}, non-smooth strongly-convex-strongly-concave and convex-concave min-max problems \citep{beznosikov2020gradient} and on different methods like ones that use one-point feedback oracle \citep{beznosikov2021one}. These works are less relevant to our paper than \citet{sadiev2020zeroth}. Moreover, the results derived in these papers are inferior to the ones from \citet{sadiev2020zeroth}.

\newpage

\section{MISSING DETAILS ON NUMERICAL EXPERIMENTS}
\label{app:experiment_details}

The code for the experiments is available here: \url{https://github.com/hugobb/sgda}.

\subsection{Setup}

We consider the special case of \eqref{eq:VI} with $F$ and $R$ defined as follows:
\begin{gather}
    F(x) = \frac{1}{n}\sum\limits_{i=1}^n F_i(x),\quad F_i(x) = \mA_i x + b_i, \label{ExpAffOperator}\\
    R(x) = \lambda\|x\|_1 + \delta_{B_r(0)}(x) = \lambda\|x\|_1 + \begin{cases}0,&\text{if } \|x\|_\infty \leq r,\\ +\infty,&\text{if } \|x\|_\infty > r, \end{cases} \label{eq:quadratic_game_appendix}
\end{gather}
where each matrix $\mA_i \in \R^{d\times d}$ is non-symmetric with all eigenvalues with strictly positive real part, $b_i \in \R^d$, $r > 0$ is the radius of $\ell_\infty$-ball, and $\lambda \geq 0$ is regularization parameter. One can show (see Example~6.22 from \citet{beck2017first}) that for the given $R(x)$ prox operator has an explicit formula:
\begin{equation}
    \prox_{\gamma R}(x) = \operatorname{sign}\left(x\right) \min \left\{\max \left\{|x|-\gamma \lambda, 0\right\}, r\right\}, \label{eq:prox_for_quadr_game}
\end{equation}
where $\operatorname{sign}(\cdot)$ and $|\cdot|$ are component-wise operators. The considered problem generalizes the following quadratic game:
\begin{equation*}
    \min\limits_{\|x_1\|_\infty \leq r}\max\limits_{\|x_2\|_\infty \leq r}\frac{1}{n}\sum\limits_{i=1}^n \frac{1}{2}x_1^\top \mA_{1,i} x_1 + x_1^\top \mA_{2,i} x_2 - \frac{1}{2} x_2^\top \mA_{3,i} x_2 + b_{1,i}^\top x_1 - b_{2,i}^\top x_2 + \lambda\|x_1\|_1 - \lambda \|x_2\|_1
\end{equation*}
with $\mu_i \mI \preccurlyeq \mA_{1,i} \preccurlyeq L_i \mI$ and $\mu_i \mI \preccurlyeq \mA_{3,i} \preccurlyeq L_i \mI$. Indeed, the above problem is a special case of \eqref{eq:VI}+\eqref{eq:quadratic_game_appendix} with
\begin{gather*}
    x = \begin{pmatrix}x_1 \\ x_2\end{pmatrix},\quad \mA_i = \begin{pmatrix}\mA_{1,i} & \mA_{2,i} \\ -\mA_{2,i} & \mA_{3,i} \end{pmatrix}, \quad b_i = \begin{pmatrix}b_{1,i} \\ b_{2,i}\end{pmatrix},\\
    R(x) = \lambda\|x_1\|_1 + \lambda\|x_2\|_1 + \delta_{B_r(0)}(x_1) + \delta_{B_r(0)}(x_2).
\end{gather*}

In our experiments, to generate the non-symmetric matrices $\mA_i \in \R^{d\times d}$ defined in \eqref{eq:quadratic_game_appendix}, we first sample real random matrices $\mB_i$ where the elements of the matrices are sampled from a normal distribution. We then compute the eigendecomposition of the matrices $\mB_i=\mQ_i\mD_i\mQ_i^{-1}$, where the $\mD_i$ are diagonal matrices with complex numbers on the diagonal. Next, we construct the matrices $\mA_i = \Re(\mQ_i\mD_i^+\mQ_i^{-1})$ where $\Re(\mM)_{i,j} = \Re(\mM_{i,j})$ and $\mD_i^+$ is obtained by transforming all the elements of $\mD_i$ to have positive real part. This process ensures that the eigenvalues of $\mA_i$ all have positive real part, and thus that $F(x)$ is strongly monotone and cocoercive. The $b_i \in \R^{d}$ are sampled from a normal distribution with variance $\nicefrac{100}{d}$.
For all the experiments we choose $n=1000$ and $d=100$. For the distributed experiments we simulate $m = 10$ nodes on a single machine with 2 CPUs. 


\subsection{Additional Numerical Experiments with Distributed Methods}

In the main part, we reported the numerical results on the comparison of \algname{QSGDA}, \algname{DIANA-SGDA}, and \algname{VR-DIANA-SGDA} applied to solve a distributed version of the quadratic game, in which we assume that $F(x) = \tfrac{1}{n}\sum_{i=1}^n F_i(x)$ with each $\{F_i\}_{i=1}^n$ having similar form to \eqref{ExpAffOperator}. Fig.~\ref{fig:distributed_exp} shows the results for the problem with $R(x) = 0$. In Fig.~\ref{fig:distributed_prox}, we present the results for the problem with $R(x)$ defined in \eqref{eq:quadratic_game_appendix}. The behavior of the methods in this case is very similar to the case without regularization $R(x)$.



\begin{figure}[H]
    \begin{subfigure}[b]{0.49\columnwidth}
         \centering
    \includegraphics[width=\textwidth]{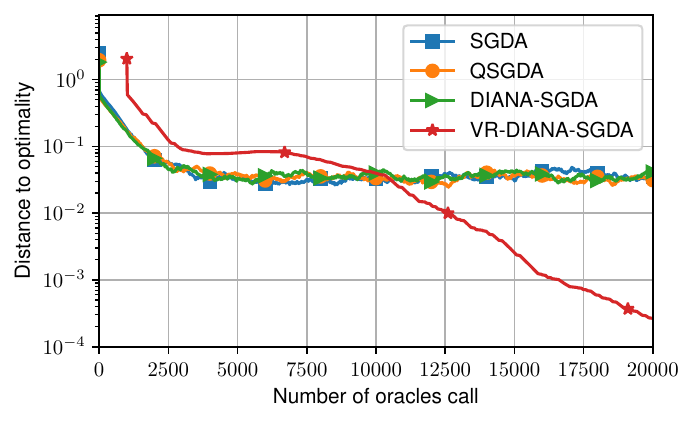}
     \end{subfigure}
     \begin{subfigure}[b]{0.49\columnwidth}
         \centering
    \includegraphics[width=\textwidth]{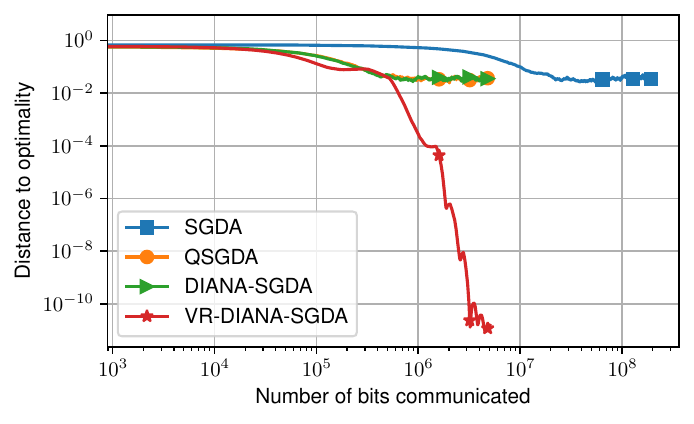}
     \end{subfigure}
    \caption{Results on distributed quadratic games with constraints. \textbf{Letf:} Number of oracle calls. \textbf{Right:} Number of bits communicated between nodes.}
    \label{fig:distributed_prox}
\end{figure}

However, in both Fig.~\ref{fig:distributed_exp} and \ref{fig:distributed_prox}, \algname{DIANA-SGDA} performs similarly to \algname{QSGDA} since the noise $\sigma^2$ is larger than the dissimilarity constant $\zeta_*^2$. To illustrate further the difference between \algname{DIANA-SGDA} and \algname{QSGDA}, we conduct an additional experiment with full-batched methods ($\sigma = 0$), see Fig.~\ref{fig:qsgda_vs_diana}. We consider the full-batch version of \algname{QSGDA} and \algname{DIANA-SGDA}. This enables us to separate the noise coming from the quantization from the noise coming from the stochasticity. We observe that when using full-batch \algname{DIANA-SGDA} converges linearly to the solution while \algname{QSGDA} only converges to a neighborhood of the solution. An interesting observation is that although the convergence is linear, the distance to optimality is not monotonically decreasing, this does not contradicts the theory.

\begin{figure}[H]
    \centering
    \includegraphics[width=0.45\textwidth]{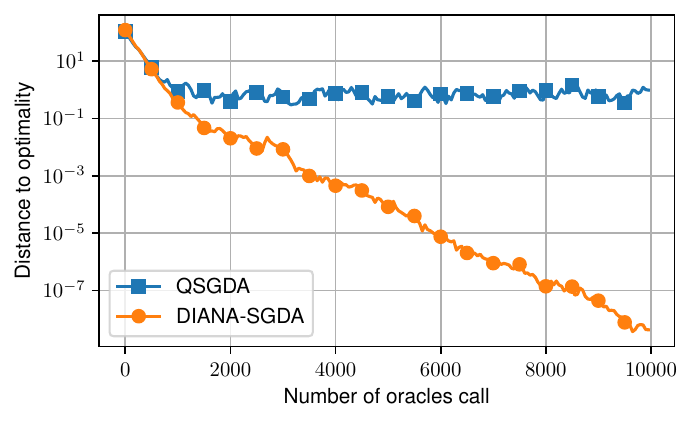}
    \caption{\algname{QSGDA} vs \algname{DIANA-SGDA}: \algname{DIANA-SGDA} converges linearly to the solution while \algname{QSGDA} only converges to a neighborhood of the solution.}
    \label{fig:qsgda_vs_diana}
\end{figure}

\newpage

\section{AUXILIARY RESULTS AND TECHNICAL LEMMAS}\label{app:aux_results}

\paragraph{Useful inequalities.} In our proofs, we often apply the following inequalities that hold for any $a,b \in \R^d$ and $\alpha > 0$:
\begin{eqnarray}
    \|a+b\|^2 &\leq& 2\|a\|^2 + 2\|b\|^2, \label{eq:a_plus_b_squared}\\
    \langle a, b \rangle &\leq& \frac{1}{2\alpha}\|a\|^2 + \frac{\alpha}{2}\|b\|^2. \label{eq:young_ineq}
\end{eqnarray}

\paragraph{Useful lemmas.} The following lemma from \citet{stich2019unified} allows us to derive the rates of convergence to the exact solution.
\begin{lemma}[Simplified version of Lemma 3 from \cite{stich2019unified}]\label{lem:stich_lemma_for_str_cvx_conv}
	Let the non-negative sequence $\{r_k\}_{k\ge 0}$ satisfy the relation
	\begin{equation*}
		r_{k+1} \leq (1 - a\gamma_k)r_k + c\gamma_k^2
	\end{equation*}
	for all $k \geq 0$, parameters $a>0,$ $c\ge 0$, and any non-negative sequence $\{\gamma_k\}_{k\ge 0}$ such that $\gamma_k \leq \nicefrac{1}{h}$ for some $h \ge a$, $h> 0$. Then, for any $K \ge 0$ one can choose $\{\gamma_k\}_{k \ge 0}$ as follows:
	\begin{eqnarray*}
		\text{if } K \le \frac{h}{a}, && \gamma_k = \frac{1}{h},\\
		\text{if } K > \frac{h}{a} \text{ and } k < k_0, && \gamma_k = \frac{1}{h},\\
		\text{if } K > \frac{h}{a} \text{ and } k \ge k_0, && \gamma_k = \frac{2}{a(\kappa + k - k_0)},
	\end{eqnarray*}
	where $\kappa = \nicefrac{2h}{a}$ and $k_0 = \left\lceil \nicefrac{K}{2} \right\rceil$. For this choice of $\gamma_k$ the following inequality holds:
	\begin{eqnarray*}
		r_{K} \le \frac{32hr_0}{a}\exp\left(-\frac{a K}{2h}\right) + \frac{36c}{a^2 K}.
	\end{eqnarray*}
\end{lemma}

In the analysis of monotone case, we rely on the classical result from proximal operators theory.
\begin{lemma}[Theorem 6.39 (iii) from \citet{beck2017first}] \label{lemma:main_result_monotone_appendix}
Let $R$ be a proper lower semicontinuous convex function and $x^+ = \text{prox}_{\gamma R} (x)$. Then for all $z \in \R^d$ the following inequality holds:
\begin{equation*}
    \langle x^+ - x, z - x^+ \rangle \geq \gamma \left( R(x^+) - R(z) \right).
\end{equation*}
\end{lemma}

Finally, we rely on the following technical lemma for handling the sums arising in the proofs for the monotone case.

\begin{lemma}\label{lem:variance_lemma}
    Let $K > 0$ be a positive integer and $\eta_1, \eta_2, \ldots, \eta_K$ be random vectors such that $\Exp_k[\eta_k] \eqdef \Exp[\eta_k \mid \eta_1,\ldots, \eta_{k-1}] = 0$ for $k = 2, \ldots, K$. Then
    \begin{equation}
        \Exp\left[\left\|\sum\limits_{k=1}^K \eta_k\right\|^2\right] = \sum\limits_{k=1}^K \Exp[\|\eta_k\|^2]. \label{eq:variance_lemma}
    \end{equation}
\end{lemma}
\begin{proof}
    We start with the following derivation:
    \begin{eqnarray*}
        \Exp\left[\left\|\sum\limits_{k=1}^K \eta_k\right\|^2\right] &=& \Exp[\|\eta_K\|^2] + 2\Exp\left[\left\langle \eta_K,  \sum\limits_{k=1}^{K-1} \eta_k \right\rangle\right] + \Exp\left[\left\|\sum\limits_{k=1}^{K-1} \eta_k\right\|^2\right]\\
        &=& \Exp[\|\eta_K\|^2] + 2\Exp\left[\Exp_K\left[\left\langle \eta_K,  \sum\limits_{k=1}^{K-1} \eta_k \right\rangle\right]\right] + \Exp\left[\left\|\sum\limits_{k=1}^{K-1} \eta_k\right\|^2\right]\\
        &=& \Exp[\|\eta_K\|^2] + 2\Exp\left[\left\langle \Exp_K[\eta_K],  \sum\limits_{k=1}^{K-1} \eta_k \right\rangle\right] + \Exp\left[\left\|\sum\limits_{k=1}^{K-1} \eta_k\right\|^2\right]\\
        &=& \Exp[\|\eta_K\|^2] + \Exp\left[\left\|\sum\limits_{k=1}^{K-1} \eta_k\right\|^2\right].
    \end{eqnarray*}
    Applying similar steps to $\Exp\left[\left\|\sum_{k=1}^{K-1} \eta_k\right\|^2\right], \Exp\left[\left\|\sum_{k=1}^{K-2} \eta_k\right\|^2\right], \ldots, \Exp\left[\left\|\sum_{k=1}^{2} \eta_k\right\|^2\right]$, we get the result.
\end{proof}

\newpage

\section{PROOFS OF THE MAIN RESULTS}
In this section, we provide complete proofs of our main results. 

\subsection{Quasi-Strongly Monotone Case}
We start with the case when $F$ satisfies \eqref{eq:QSM} with $\mu > 0$. For readers convenience, we restate the theorems below.
\begin{theorem}[Theorem~\ref{thm:main_result}]\label{thm:main_result_appendix}
    Let $F$ be $\mu$-quasi-strongly monotone with $\mu > 0$ and Assumption~\ref{as:key_assumption} hold. Assume that
    \begin{equation}
        0 < \gamma \leq \min\left\{\frac{1}{\mu}, \frac{1}{2(A + CM)}\right\} \label{eq:gamma_condition_appendix}
    \end{equation}
    for some $M > \nicefrac{B}{\rho}$. Then for the Lyapunov function $V_k = \|x^k - x^{*,k}\|^2 + M\gamma^2 \sigma_k^2$, and for all $k\ge 0$ we have
    \begin{eqnarray}
        \Exp[V_k] &\leq& \left(1 - \min\left\{\gamma\mu, \rho - \frac{B}{M}\right\}\right)^k \Exp[V_0] + \frac{\gamma^2(D_1 + MD_2)}{\min\left\{\gamma\mu, \rho - \nicefrac{B}{M}\right\}}. \label{eq:main_result_appendix}
    \end{eqnarray}
\end{theorem}
\begin{proof}
First of all, we recall a well-known fact about proximal operators: for any solution $x^*$ of \eqref{eq:VI} we have
\begin{equation}
    x^* = \prox_{\gamma R}(x^* - \gamma F(x^*)). \label{eq:x^*_stationarity}
\end{equation}
Using this and non-expansiveness of proximal operator, we derive
\begin{eqnarray}
    \|x^{k+1}-x^{*,k+1}\|^2 &\leq& \|x^{k+1}-x^{*,k}\|^2\notag\\
    &=& \left\|\prox_{\gamma R}(x^k-\gamma g^k) - \prox_{\gamma R}(x^{*,k} - \gamma F(x^{*,k}))\right\|^2\notag\\
    &\leq& \left\|x^k-\gamma g^k - x^{*,k} - \gamma F(x^{*,k})\right\|^2\notag\\
    &=&\|x^k-x^{*,k}\|^2 - 2\gamma\left\langle x^k-x^{*,k}, g^k - F(x^{*,k}) \right\rangle  + \gamma^2\|g^k - F(x^{*,k})\|^2.\notag
\end{eqnarray}
Next, we take an expectation $\Exp_k[\cdot]$ w.r.t.\ the randomness at iteration $k$ and get
\begin{eqnarray}
\Exp_{k} \left[ \|x^{k+1}-x^{*,k+1}\|^2 \right]&=& \|x^k-x^{*,k}\|^2- 2\gamma  \left\langle x^k-x^{*,k}, F(x^k) - F(x^{*,k})\right\rangle\notag\\
&&\quad + \gamma^2  \Exp_{k}  \left[ \left\|g^k - F(x^{*,k})\right\|^2 \right] \notag\\
&\overset{\eqref{eq:second_moment_bound}}{\leq}& \|x^k-x^{*,k}\|^2- 2\gamma  \left\langle x^k-x^*, F(x^k) - F(x^{*,k})\right\rangle \notag\\
&&\quad + \gamma^2 \left(2A \left\langle x^k-x^{*,k}, F(x^k) - F(x^{*,k})\right\rangle +B\sigma_k^2 + D_1\right). \notag
\end{eqnarray}
Summing up this inequality with \eqref{eq:sigma_k_bound} multiplied by $M\gamma^2$, we obtain
\begin{align}
\label{temp0_str_monotone_appendix}
\Exp_{k} \left[ \|x^{k+1}-x^{*,k+1}\|^2 \right] &+M\gamma^2 \Exp_{k} [\sigma_{k+1}^2]\notag\\
&\leq  \|x^k-x^{*,k}\|^2- 2\gamma  \left\langle x^k-x^{*,k}, F(x^k) - F(x^{*,k})\right\rangle \notag\\ 
&\quad + \gamma^2 \left(2A \left\langle x^k-x^{*,k}, F(x^k) - F(x^{*,k})\right\rangle +B\sigma_k^2 + D_1\right) \notag\\
&\quad + M\gamma^2 \left(2C \left\langle x^k-x^{*,k}, F(x^k) - F(x^{*,k})\right\rangle + (1-\rho) \sigma_k^2 + D_2\right) \notag\\
&= \|x^k-x^{*,k}\|^2 + M\gamma^2\left(1 - \rho + \frac{B}{M}\right)\sigma_k^2 + \gamma^2(D_1 + MD_2)\notag\\
&\quad -2\gamma\left(1 - \gamma(A + CM)\right)\left\langle x^k-x^{*,k}, F(x^k) - F(x^{*,k})\right\rangle.
\end{align}
Since $\gamma \leq \frac{1}{2(A+CM)}$ the factor $-2\gamma\left(1 - \gamma(A + CM)\right)$ is non-positive. Therefore, applying strong quasi-monotonicity of $F$, we derive
\begin{eqnarray}
\Exp_k \left[ \|x^{k+1}-x^{*,k+1}\|^2 +M\gamma^2 \sigma_{k+1}^2\right] &\leq & \left(1 - 2 \gamma\mu\left(1 - \gamma(A + CM)\right)\right) \|x^k-x^{*,k}\|^2 \notag\\ 
&&\quad + M\gamma^2\left(1 - \rho + \frac{B}{M}\right)\sigma_k^2 + \gamma^2(D_1 + MD_2).\notag
\end{eqnarray}
Using $\gamma \leq \frac{1}{2(A+CM)}$ and the definition $V_k = \|x^{k}-x^{*,k}\|^2 +M\gamma^2 \sigma_{k}^2$, we get
\begin{eqnarray}
\Exp_k \left[ V_{k+1}\right] &\leq & \left(1-\gamma \mu \right) \|x^k-x^{*,k}\|^2 + M\gamma^2\left(1 - \rho + \frac{B}{M}\right)\sigma_k^2 + \gamma^2(D_1 + MD_2)\notag\\
&\leq& \left(1 - \min\left\{\gamma\mu, \rho - \frac{B}{M}\right\}\right)V_k + \gamma^2(D_1 + MD_2).\notag
\end{eqnarray}
Next, we take the full expectation from the above inequality and establish the following recurrence:
\begin{eqnarray}
\Exp \left[ V_{k+1}\right] &\leq & \left(1 - \min\left\{\gamma\mu, \rho - \frac{B}{M}\right\}\right)\Exp[V_k] + \gamma^2(D_1 + MD_2).\label{eq:V_k_recurrence}
\end{eqnarray}
Unrolling the recurrence, we derive
\begin{eqnarray}
\Exp \left[ V_{k}\right] &\leq & \left(1 - \min\left\{\gamma\mu, \rho - \frac{B}{M}\right\}\right)^k\Exp[V_0] + \gamma^2(D_1 + MD_2)\sum\limits_{t=0}^{k-1}\left(1 - \min\left\{\gamma\mu, \rho - \frac{B}{M}\right\}\right)^t\notag\\
&\leq& \left(1 - \min\left\{\gamma\mu, \rho - \frac{B}{M}\right\}\right)^k\Exp[V_0] + \gamma^2(D_1 + MD_2)\sum\limits_{t=0}^{\infty}\left(1 - \min\left\{\gamma\mu, \rho - \frac{B}{M}\right\}\right)^t\notag\\
&=& \left(1 - \min\left\{\gamma\mu, \rho - \frac{B}{M}\right\}\right)^k\Exp[V_0] + \frac{\gamma^2(D_1 + MD_2)}{\min\left\{\gamma\mu, \rho - \nicefrac{B}{M}\right\}},\notag
\end{eqnarray}
which finishes the proof.
\end{proof}

Using this and Lemma~\ref{lem:stich_lemma_for_str_cvx_conv}, we derive the following result about the convergence to the exact solution.
\begin{corollary}[Corollary~\ref{cor:main_result}]\label{cor:main_result_appendix}
    Let the assumptions of Theorem~\ref{thm:main_result} hold. Consider two possible cases.
    \begin{enumerate}
        \item Let $D_1 = D_2 = 0$. Then, for any $K \ge 0$, $M = \nicefrac{2B}{\rho}$, and
        \begin{equation}
            \gamma = \min\left\{\frac{1}{\mu}, \frac{1}{2(A + \nicefrac{2BC}{\rho})}\right\} \label{eq:stepsize_choice_1_QSM}
        \end{equation}
        we have
        \begin{equation}
            \Exp[V_K] \leq \Exp[V_0] \exp\left(-\min\left\{\frac{\mu}{2(A + \nicefrac{2BC}{\rho})}, \frac{\rho}{2}\right\}K\right). \label{eq:linear_conv_QSM}
        \end{equation}
        \item Let $D_1 + MD_2 > 0$. Then, for any $K \ge 0$ and $M = \nicefrac{2B}{\rho}$ one can choose $\{\gamma_k\}_{k \ge 0}$ as follows:
	\begin{eqnarray}
		\text{if } K \le \frac{h}{\mu}, && \gamma_k = \frac{1}{h},\notag\\
		\text{if } K > \frac{h}{\mu} \text{ and } k < k_0, && \gamma_k = \frac{1}{h},\label{eq:stepsize_choice_2_QSM_appendix}\\
		\text{if } K > \frac{h}{\mu} \text{ and } k \ge k_0, && \gamma_k = \frac{2}{\mu(\kappa + k - k_0)},\notag
	\end{eqnarray}
	where $h = \max\left\{2(A + \nicefrac{2BC}{\rho}), \nicefrac{2\mu}{\rho}\right\}$, $\kappa = \nicefrac{2h}{\mu}$ and $k_0 = \left\lceil \nicefrac{K}{2} \right\rceil$. For this choice of $\gamma_k$ the following inequality holds:
	\begin{eqnarray}
		\Exp[V_K] &\le& 32\max\left\{\frac{2(A + \nicefrac{2BC}{\rho})}{\mu}, \frac{2}{\rho}\right\}\Exp[V_0]\exp\left(-\min\left\{\frac{\mu}{2(A + \nicefrac{2BC}{\rho})}, \frac{\rho}{4}\right\}K\right)\notag\\
		&&\quad + \frac{36(D_1 + \nicefrac{2BD_2}{\rho})}{\mu^2 K}.\label{eq:sublinear_conv_QSM}
	\end{eqnarray}
    \end{enumerate}
\end{corollary}
\begin{proof}
    The first part of the corollary follows from Theorem~\ref{thm:main_result} due to
    \begin{equation*}
        \left(1 - \min\left\{\gamma\mu, \rho - \frac{B}{M}\right\}\right)^K = \left(1 - \min\left\{\gamma\mu, \frac{\rho}{2}\right\}\right)^K \leq \exp\left(- \min\left\{\gamma\mu, \frac{\rho}{2}\right\}K\right).
    \end{equation*}
    Plugging \eqref{eq:stepsize_choice_1_QSM} in the above inequality, we derive \eqref{eq:linear_conv_QSM}. Next, we consider the case when $D_1 + MD_2 > 0$. First, we notice that \eqref{eq:V_k_recurrence} holds for non-constant stepsizes $\gamma_k$ such that \begin{equation*}
        0 < \gamma_k \leq \min\left\{\frac{1}{\mu}, \frac{1}{2(A + CM)}\right\}.
    \end{equation*}
    Therefore, for any $k \ge 0$ we have
    \begin{eqnarray*}
    \Exp \left[ V_{k+1}\right] &\leq & \left(1 - \min\left\{\gamma_k\mu, \rho - \frac{B}{M}\right\}\right)\Exp[V_k] + \gamma_k^2(D_1 + MD_2)\\
    &\overset{M = \nicefrac{2B}{\rho}}{=}& \left(1 - \min\left\{\gamma_k\mu, \nicefrac{\rho}{2}\right\}\right)\Exp[V_k] + \gamma_k^2(D_1 + \nicefrac{2BD_2}{\rho}).
    \end{eqnarray*}
    Secondly, we assume that for all $k \ge 0$
    \begin{equation*}
        0 < \gamma_k \leq \min\left\{\frac{\rho}{2\mu}, \frac{1}{2(A + CM)}\right\}.
    \end{equation*}
    Applying this to the recurrence for $\Exp[V_k]$, we obtain
    \begin{eqnarray*}
        \Exp \left[V_{k+1}\right] &\leq & \left(1 - \gamma_k\mu\right)\Exp[V_k] + \gamma_k^2(D_1 + \nicefrac{2BD_2}{\rho}).
    \end{eqnarray*}
    It remains to apply Lemma~\ref{lem:stich_lemma_for_str_cvx_conv} with $r_k = \Exp[V_k]$, $a = \mu$, $c = D_1 + \nicefrac{2BD_2}{\rho}$, and\newline $h = \max\left\{2(A + \nicefrac{2BC}{\rho}), \nicefrac{2\mu}{\rho}\right\}$ to the above recurrence.
\end{proof}

\subsection{Monotone Case}\label{sec:monotone_appendix}
Next, we consider the case when $\mu = 0$. Before deriving the proof, we provide additional discussion of the setup.

We emphasize that the maximum in \eqref{eq:gap} is taken over the compact set $\cC$ containing the solution set $X^*$. Therefore, the quantity $\text{Gap}_{\cC} (z)$ is a valid measure of convergence \citep{nesterov2007dual}. We point out that the iterates $x^k$ do not have to lie in $\cC$. Our analysis works for the problems with unbounded and bounded domains (see \citet{nesterov2007dual, alacaoglu2021stochastic} for similar setups).

Another popular convergence measure for the case when $R(x) \equiv 0$ in \eqref{eq:VI} is $\|F(x^k)\|^2$. Although the squared norm of the operator is a weaker guarantee, it is easier to compute in practice and better suited for non-monotone problems \citep{yoon2021accelerated}. Nevertheless, $\|F(x^k)\|^2$ is not a valid measure of convergence for \eqref{eq:VI} with $R(x) \not\equiv 0$. Therefore, we focus on $\text{Gap}_{\cC} (z)$ in the monotone case.\footnote{When $R(x) \equiv 0$, our analysis can be modified to get the guarantees on the squared norm of the operator.}  

\begin{theorem}[Theorem~\ref{thm:main_result_monotone}]\label{thm:main_result_monotone_appendix}
    Let $F$ be monotone, $\ell$-star-cocoercive and Assumptions~\ref{as:key_assumption}, \ref{as:boundness} hold. Assume that
    \begin{equation}
        0 < \gamma \leq \frac{1}{2(A + \nicefrac{BC}{\rho})}. \label{eq:gamma_condition_monotone_appendix}
    \end{equation}
    Then for the function $\text{Gap}_{\cC} (z)$ from \eqref{eq:gap} and for all $K\ge 0$ we have
\begin{eqnarray}
    \label{eq:main_result_monotone_appendix}
    \Exp\left[\text{Gap}_{\cC} \left(\frac{1}{K}\sum\limits_{k=1}^{K}  x^{k}\right)\right] 
    &\leq& \frac{3\left[\max_{u \in \mathcal{C}}\|x^{0} - u\|^2\right]}{2\gamma K} \notag\\
    &&\quad + \frac{8\gamma \ell^2 \Omega_{\mathcal{C}}^2}{K} + \left( 4A + \ell + \nicefrac{8BC}{\rho}\right) \cdot \frac{\|x^0-x^{*,0}\|^2}{K} \notag\\
    &&\quad + \left(4 +  \left( 4A + \ell + \nicefrac{8BC}{\rho}\right) \gamma \right) \frac{\gamma B \sigma_0^2}{\rho K} \notag\\
    &&\quad +\gamma (2 + \gamma\left( 4A + \ell + \nicefrac{8BC}{\rho}\right))(D_1 + \nicefrac{2BD_2}{\rho})\notag\\
    &&\quad + 9\gamma\max\limits_{x^* \in X^*}\|F(x^*)\|^2.
\end{eqnarray}
\end{theorem}
\begin{proof}
First, we apply the classical result about proximal operators (Lemma \ref{lemma:main_result_monotone_appendix}) with $x^+ = x^{k+1}$, $x = x^k - \gamma g^k$, and $z = u$ for arbitrary point $u\in \R^d$:
\begin{eqnarray*}
    \langle x^{k+1} - x^k + \gamma g^k, u - x^{k+1} \rangle &\geq& \gamma \left( R(x^{k+1}) - R(u) \right).
\end{eqnarray*}
Multiplying by the factor of $2$ and making small rearrangement, we get
\begin{eqnarray*}
    2\gamma\langle g^k, u - x^{k} \rangle + 2\langle x^{k+1} - x^k , u - x^{k} \rangle + 2\langle x^{k+1} - x^k + \gamma g^k, x^{k} - x^{k+1} \rangle &\geq& 2\gamma \left( R(x^{k+1}) - R(u) \right)
\end{eqnarray*}
implying
\begin{eqnarray*}
    2\gamma \left( \langle F(x^k), x^k - u \rangle + R(x^{k+1}) - R(u) \right) &\leq& 
    2\langle x^{k+1} - x^k , u - x^{k} \rangle + 2\gamma \langle F(x^k) - g^k, x^k - u \rangle \\
    &&\quad+2\langle x^{k+1} - x^k, x^{k} - x^{k+1} \rangle + 2\gamma\langle  g^k, x^{k} - x^{k+1} \rangle.
\end{eqnarray*}
Next, we use a squared norm decomposition $\|a+b\|^2 = \|a\|^2 + \|b\|^2 + 2\langle a,b\rangle$, and obtain
\begin{eqnarray}
    \label{eq:temp1_mon}
    2\gamma \left( \langle F(x^k), x^k - u \rangle + R(x^{k+1}) - R(u) \right) &\leq& 
    \|x^{k+1} - x^k\|^2 + \|x^{k} - u\|^2 - \| x^{k+1} - u\|^2 \notag\\
    &&\quad + 2\gamma \langle F(x^k) - g^k, x^k - u \rangle \notag\\
    &&\quad-2\| x^{k+1} - x^k\|^2 + 2\gamma\langle g^k, x^{k} - x^{k+1} \rangle.
\end{eqnarray}
Then, due to $2\langle  a, b \rangle \leq \|a\|^2 + \|b\|^2$ we have
\begin{eqnarray*}
    2\gamma \left( \langle F(x^k), x^k - u \rangle + R(x^{k+1}) - R(u) \right) &\leq& 
    \|x^{k+1} - x^k\|^2 + \|x^{k} - u\|^2 - \| x^{k+1} - u\|^2\\
    &&\quad + 2\gamma \langle F(x^k) - g^k, x^k - u \rangle \\
    &&\quad-2\| x^{k+1} - x^k\|^2 + \gamma^2\| g^k\|^2 + \| x^{k} - x^{k+1}\|^2 \\
    &=& \|x^{k} - u\|^2 - \| x^{k+1} - u\|^2 \\
    &&\quad + 2\gamma \langle F(x^k) - g^k, x^k - u \rangle + \gamma^2\| g^k\|^2.
\end{eqnarray*}
Monotonicity of $F$ implies $\langle F(u), x^k - u \rangle \leq \langle F(x^k), x^k - u \rangle$, allowing us to continue our derivation as follows:
\begin{eqnarray*}
    2\gamma \left( \langle F(u), x^k - u \rangle + R(x^{k+1}) - R(u) \right) &\leq& 
     \|x^{k} - u\|^2 - \| x^{k+1} - u\|^2\\
     &&\quad + 2\gamma \langle F(x^k) - g^k, x^k - u \rangle + \gamma^2\| g^k\|^2 \\
     &=& \|x^{k} - u\|^2 - \| x^{k+1} - u\|^2 \\
     &&\quad+ 2\gamma \langle F(x^k) - g^k, x^k - u \rangle\\
     &&\quad+ \gamma^2\| g^k - g^{*,k} + g^{*,k}\|^2 \\
     &\overset{\eqref{eq:a_plus_b_squared}}{\leq}& \|x^{k} - u\|^2 - \| x^{k+1} - u\|^2\\
     &&\quad + 2\gamma \langle F(x^k) - g^k, x^k - u \rangle \\
     &&\quad+ 2\gamma^2\| g^k - g^{*,k}\|^2 + 2\gamma^2\|g^{*,k}\|^2.
\end{eqnarray*}
Summing up the above inequality for $k = 0,1,\ldots, K-1$, we get
\begin{eqnarray*}
    2\gamma \sum\limits_{k=0}^{K-1} \left( \langle F(u), x^k - u \rangle + R(x^{k+1}) - R(u) \right) &\leq& 
    \sum\limits_{k=0}^{K-1}\|x^{k} - u\|^2 - \sum\limits_{k=0}^{K-1}\| x^{k+1} - u\|^2\\
    &&\quad + 2\gamma^2\sum\limits_{k=0}^{K-1}\|g^{*,k}\|^2\\
    &&\quad + 2\gamma \sum\limits_{k=0}^{K-1} \langle F(x^k) - g^k, x^k - u \rangle\\
    &&\quad + 2\gamma^2 \sum\limits_{k=0}^{K-1} \| g^k - g^{*,k}\|^2 \\
    &=& \|x^{0} - u\|^2 - \| x^{K} - u\|^2 + 2\gamma^2\sum\limits_{k=0}^{K-1}\|g^{*,k}\|^2\\
    &&\quad+ 2\gamma \sum\limits_{k=0}^{K-1} \langle F(x^k) - g^k, x^k - u \rangle\\
    &&\quad+ 2\gamma^2 \sum\limits_{k=0}^{K-1} \| g^k - g^{*,k}\|^2.
\end{eqnarray*}
Next, we divide both sides by $2\gamma K$ 
\begin{eqnarray*}
    \frac{1}{K}\sum\limits_{k=0}^{K-1} \left( \langle F(u),  x^k - u \rangle + R(x^{k+1}) - R(u) \right) &\leq& 
     \frac{\|x^{0} - u\|^2 - \| x^{K} - u\|^2 }{2\gamma K} + \frac{\gamma}{K}\sum\limits_{k=0}^{K-1}\|g^{*,k}\|^2\\
    &&\quad+ \frac{1}{K} \sum\limits_{k=0}^{K-1} \langle F(x^k) - g^k, x^k - u \rangle\\
    &&\quad + \frac{\gamma}{K} \sum\limits_{k=0}^{K-1} \| g^k - g^{*,k}\|^2
\end{eqnarray*}
and, after small rearrangement, we obtain
\begin{eqnarray*}
    \frac{1}{K}\sum\limits_{k=0}^{K-1} \left( \langle F(u),  x^{k+1} - u \rangle + R(x^{k+1}) - R(u) \right) &\leq& 
     \frac{\|x^{0} - u\|^2 - \| x^{K} - u\|^2}{2\gamma K} +   \frac{\langle F(u),  x^{K} - x^{0}  \rangle }{K} \\
     &&\quad +  \frac{\gamma}{K}\sum\limits_{k=0}^{K-1}\|g^{*,k}\|^2\\
    &&\quad+ \frac{1}{K} \sum\limits_{k=0}^{K-1} \langle F(x^k) - g^k, x^k - u \rangle\\
    &&\quad + \frac{\gamma}{K} \sum\limits_{k=0}^{K-1} \| g^k - g^{*,k}\|^2.
\end{eqnarray*}
Applying Jensen's inequality for convex function $R$, we get $R\left(\frac{1}{K}\sum_{k=0}^{K-1}  x^{k+1}\right) \le \frac{1}{K}\sum_{k=0}^{K-1}R(  x^{k+1} )$. Plugging this in the previous inequality, we derive for $u^*$ being a projection of $u$ on $X^*$
\begin{align*}
    \Bigg\langle F(u),  \left(\frac{1}{K}\sum\limits_{k=0}^{K-1}  x^{k+1}\right) &- u \Bigg\rangle + R\left(\frac{1}{K}\sum\limits_{k=0}^{K-1}  x^{k+1}\right) - R(u) \\
    &\leq
     \frac{\|x^{0} - u\|^2- \| x^{K} - u\|^2}{2\gamma K} +   \frac{\langle F(u),  x^{K} - x^{0}  \rangle }{K}    +  \frac{\gamma}{K}\sum\limits_{k=0}^{K-1}\|g^{*,k}\|^2\\
    &\quad+ \frac{1}{K} \sum\limits_{k=0}^{K-1} \langle F(x^k) - g^k, x^k - u \rangle + \frac{\gamma}{K} \sum\limits_{k=0}^{K-1} \| g^k - g^{*,k}\|^2 \\
    &\overset{\eqref{eq:young_ineq}}{\leq} \frac{\|x^{0} - u\|^2- \| x^{K} - u\|^2}{2\gamma K} +
    \frac{\| x^{K} - x^{0} \|^2}{4\gamma K} + \frac{4\gamma}{K} \| F(u) - F(u^*) + F(u^*)\|^2 \\
    &\quad+  \frac{\gamma}{K}\sum\limits_{k=0}^{K-1}\|g^{*,k}\|^2 + \frac{1}{K} \sum\limits_{k=0}^{K-1} \langle F(x^k) - g^k, x^k - u \rangle + \frac{\gamma}{K} \sum\limits_{k=0}^{K-1} \| g^k - g^{*,k}\|^2 \\
    &\overset{\eqref{eq:a_plus_b_squared}}{\leq} \frac{\|x^{0} - u\|^2 - \|x^K - u\|^2}{2\gamma K} + \frac{\|x^{0} - u\|^2 + \|x^K - u\|^2}{2\gamma K}
     + \frac{8\gamma}{K} \| F(u) - F(u^*)\|^2 \\
    &\quad+  \frac{\gamma}{K}\sum\limits_{k=0}^{K-1}\|g^{*,k}\|^2 + 8\gamma\|F(u^*)\|^2 + \frac{1}{K} \sum\limits_{k=0}^{K-1} \langle F(x^k) - g^k, x^k - u \rangle\\
    &\quad + \frac{\gamma}{K} \sum\limits_{k=0}^{K-1} \| g^k - g^{*,k}\|^2 \\
    &\overset{\eqref{eq:cocoercivity}}{\leq} \frac{\|x^{0} - u\|^2}{\gamma K} + \frac{8\gamma \ell^2 \|u - u^*\|^2}{K} +  9\gamma \max\limits_{x^* \in X^*}\|F(x^*)\|^2\\
    &\quad+ \frac{1}{K} \sum\limits_{k=0}^{K-1} \langle F(x^k) - g^k, x^k - u \rangle + \frac{\gamma}{K} \sum\limits_{k=0}^{K-1} \| g^k - g^{*,k}\|^2.
\end{align*}
Next, we take maximum from the both sides in $u \in \mathcal{C}$, which gives $\text{Gap}_{\cC} \left(\frac{1}{K}\sum_{k=1}^{K}  x^{k}\right)$ in the left-hand side by definition \eqref{eq:gap}, and take the expectation of the result:
\begin{eqnarray}
\label{temp_0_monotone_appendix}
    \Exp\left[\text{Gap}_{\cC} \left(\frac{1}{K}\sum\limits_{k=1}^{K}  x^{k}\right)\right] &\leq&      \frac{\Exp\left[\max_{u \in \mathcal{C}}\|x^{0} - u\|^2\right]}{\gamma K}  + \frac{8\gamma \ell^2 \Exp\left[\max_{u \in \mathcal{C}}\| u - u^*\|^2\right]}{K}\notag\\
    &&\quad+  9\gamma \max\limits_{x^* \in X^*}\|F(x^*)\|^2 \notag\\
    &&\quad + \frac{1}{K} \Exp\left[\max_{u \in \mathcal{C}}\sum\limits_{k=0}^{K-1} \langle F(x^k) - g^k, x^k - u \rangle \right] + \frac{\gamma}{K} \sum\limits_{k=0}^{K-1} \Exp\left[\| g^k - g^{*,k}\|^2\right] \notag\\
    &\leq&      \frac{\Exp\left[\max_{u \in \mathcal{C}}\|x^{0} - u\|^2\right]}{\gamma K}  + \frac{8\gamma \ell^2 \Omega_{\mathcal{C}}^2}{K}  +  9\gamma \max\limits_{x^* \in X^*}\|F(x^*)\|^2 \notag\\
    &&\quad+ \frac{1}{K} \Exp\left[\max_{u \in \mathcal{C}}\sum\limits_{k=0}^{K-1} \langle F(x^k) - g^k, x^k - u \rangle \right]\notag\\
    &&\quad + \frac{\gamma}{K} \sum\limits_{k=0}^{K-1} \Exp\left[\| g^k - g^{*,k}\|^2\right].
\end{eqnarray}
In the last step, we also use that $X^* \subset \cC$ and $\Omega_{\mathcal{C}} \eqdef \max_{x,y \in \cC} \|x-y\|$ (Assumption \ref{as:boundness}).

It remains to upper bound the terms from the last two lines of \eqref{temp_0_monotone_appendix}. We start with the first one. Since
\begin{eqnarray*}
    \Exp\left[\sum_{k=0}^{K-1} \langle F(x^k) - g^k, x^k \rangle\right] &=& \Exp\left[\sum_{k=0}^{K-1} \langle \Exp[F(x^k) - g^k \mid x^k], x^k \rangle\right] = 0,\\
    \Exp\left[\sum_{k=0}^{K-1} \langle F(x^k) - g^k, x^0 \rangle\right] &=& \sum\limits_{k=0}^{K-1} \left\langle \Exp[F(x^k) - g^k], x^0\right\rangle = 0,
\end{eqnarray*}
we have
\begin{eqnarray*}
    \frac{1}{K} \Exp\left[\max_{u \in \mathcal{C}}\sum\limits_{k=0}^{K-1} \langle F(x^k) - g^k, x^k - u \rangle \right] &=& \frac{1}{K}\Exp\left[\sum_{k=0}^{K-1} \langle F(x^k) - g^k, x^k \rangle\right]\\
    &&\quad+ \frac{1}{K} \Exp\left[\max_{u \in \mathcal{C}}\sum\limits_{k=0}^{K-1} \langle F(x^k) - g^k, - u \rangle \right]\\
    &=& \frac{1}{K} \Exp\left[\max_{u \in \mathcal{C}}\sum\limits_{k=0}^{K-1} \langle F(x^k) - g^k, - u \rangle \right]\\
    &=& \frac{1}{K}\Exp\left[\sum_{k=0}^{K-1} \langle F(x^k) - g^k, x^0 \rangle\right]\\
    &&\quad + \frac{1}{K} \Exp\left[\max_{u \in \mathcal{C}}\sum\limits_{k=0}^{K-1} \langle F(x^k) - g^k, - u \rangle \right]\\
    &=& \Exp\left[\max\limits_{u \in \cC}\left\langle \frac{1}{K}\sum\limits_{k=0}^{K-1}(F(x^k) - g^k), x^0 - u \right\rangle\right]\\
    &\overset{\eqref{eq:young_ineq}}{\leq}& \Exp\left[\max\limits_{u \in \cC}\left\{\frac{\gamma K}{2}\left\| \frac{1}{K}\sum\limits_{k=0}^{K-1}(F(x^k) - g^k) \right\|^2 + \frac{1}{2\gamma K}\|x^0 - u\|^2\right\}\right]\\
    &=& \frac{\gamma}{2K}\Exp\left[\left\|\sum\limits_{k=0}^{K-1}(F(x^k) - g^k) \right\|^2\right] + \frac{1}{2\gamma K}\max\limits_{u\in \cC} \|x^0 - u\|^2.
\end{eqnarray*}
We notice that $\Exp[F(x^k) - g^k\mid F(x^0) - g^0, \ldots, F(x^{k-1}) - g^{k-1}] = 0$ for all $k \geq 1$, i.e., conditions of Lemma~\ref{lem:variance_lemma} are satisfied. Therefore, applying Lemma~\ref{lem:variance_lemma}, we get
\begin{eqnarray}
    \label{temp_2_monotone_appendix}
    \frac{1}{K} \Exp\left[\max_{u \in \mathcal{C}}\sum\limits_{k=0}^{K-1} \langle F(x^k) - g^k, x^k - u \rangle \right] &\leq& \frac{\gamma}{2K}\sum\limits_{k=0}^{K-1}\Exp[\|F(x^k) - g^k \|^2]\notag\\
    &&\quad + \frac{1}{2\gamma K}\max\limits_{u\in \cC} \|x^0 - u\|^2.
\end{eqnarray}
Combining \eqref{temp_0_monotone_appendix} and \eqref{temp_2_monotone_appendix}, we derive
\begin{eqnarray}
    \label{eq:mon_temp1}
    \Exp\left[\text{Gap}_{\cC} \left(\frac{1}{K}\sum\limits_{k=1}^{K}  x^{k}\right)\right] 
    &\leq&      \frac{3\left[\max_{u \in \mathcal{C}}\|x^{0} - u\|^2\right]}{2\gamma K}  + \frac{8\gamma \ell^2 \Omega_{\mathcal{C}}^2}{K} +  9\gamma \max\limits_{x^* \in X^*}\|F(x^*)\|^2\notag\\
    &&\quad+  \frac{\gamma}{2K}\sum\limits_{k=0}^{K-1} \Exp\left[\|g^k - F(x^k)\|^2\right] \notag\\
    &&\quad+ \frac{\gamma}{K} \sum\limits_{k=0}^{K-1} \Exp\left[\| g^k - g^{*,k}\|^2\right] \\
    &\overset{\eqref{eq:a_plus_b_squared}}{\leq}&  \frac{3\left[\max_{u \in \mathcal{C}}\|x^{0} - u\|^2\right]}{2\gamma K}  + \frac{8\gamma \ell^2 \Omega_{\mathcal{C}}^2}{K} +  9\gamma \max\limits_{x^* \in X^*}\|F(x^*)\|^2\notag\\
    &&\quad+  \frac{\gamma}{K}\sum\limits_{k=0}^{K-1}  \Exp\left[\|F(x^k) - g^{*,k}\|^2\right]  + \frac{2\gamma}{K} \sum\limits_{k=0}^{K-1} \Exp\left[\| g^k - g^{*,k}\|^2\right] \notag.
\end{eqnarray}
Using $\ell$-star-cocoercivity of $F$ together with the first part of Assumption~\ref{as:key_assumption}, we continue our derivation as follows:
\begin{eqnarray*}
    \Exp\left[\text{Gap}_{\cC} \left(\frac{1}{K}\sum\limits_{k=1}^{K}  x^{k}\right)\right]
    &\leq&  \frac{3\left[\max_{u \in \mathcal{C}}\|x^{0} - u\|^2\right]}{2\gamma K}  + \frac{8\gamma \ell^2 \Omega_{\mathcal{C}}^2}{K} + 2\gamma D_1 +  9\gamma \max\limits_{x^* \in X^*}\|F(x^*)\|^2\notag\\
    &&\quad+  \frac{\gamma (4A + \ell)}{K}\sum\limits_{k=0}^{K-1}  \Exp\left[\langle F(x^k) - g^{*,k}, x^k - x^{*,k}\rangle\right] + \frac{2\gamma B}{K} \sum\limits_{k=0}^{K-1} \Exp\left[\sigma_k^2 \right] \\
    &=& \frac{3\left[\max_{u \in \mathcal{C}}\|x^{0} - u\|^2\right]}{2\gamma K}  + \frac{8\gamma \ell^2 \Omega_{\mathcal{C}}^2}{K} + 2\gamma D_1 +  9\gamma \max\limits_{x^* \in X^*}\|F(x^*)\|^2\notag\\
    &&\quad+  \frac{\gamma (4A + \ell)}{K}\sum\limits_{k=0}^{K-1}  \Exp\left[\langle F(x^k) - g^{*,k}, x^k - x^{*,k}\rangle\right] \notag\\
    &&\quad + \frac{2\gamma B}{K} \left(1 + \frac{1}{\rho} \right) \sum\limits_{k=0}^{K-1} \Exp\left[\sigma_k^2 \right] - \frac{2\gamma B}{\rho K} \sum\limits_{k=0}^{K-1} \Exp\left[\sigma_k^2 \right].
\end{eqnarray*}
Next, we use the second part of Assumption \ref{as:key_assumption} and get
\begin{eqnarray*}
    \Exp\left[\text{Gap}_{\cC} \left(\frac{1}{K}\sum\limits_{k=1}^{K}  x^{k}\right)\right]
    &\leq& \frac{3\left[\max_{u \in \mathcal{C}}\|x^{0} - u\|^2\right]}{2\gamma K}  + \frac{8\gamma \ell^2 \Omega_{\mathcal{C}}^2}{K} + 2\gamma D_1 +  9\gamma \max\limits_{x^* \in X^*}\|F(x^*)\|^2 \notag\\
    &&\quad+  \frac{\gamma (4A + \ell)}{K}\sum\limits_{k=0}^{K-1}  \Exp\left[\langle F(x^k) - g^{*,k}, x^k - x^{*,k}\rangle\right] \notag\\
    &&\quad + \frac{2\gamma B}{K} \left(1 + \frac{1}{\rho} \right) \sum\limits_{k=1}^{K-1} \Exp\left[2C\langle F(x^{k-1}) - g^{*,k-1}, x^{k-1} - x^{*,k-1}\rangle\right] 
    \notag\\
    &&\quad + \frac{2\gamma B}{K} \left(1 + \frac{1}{\rho} \right) \sum\limits_{k=1}^{K-1} \Exp\left[(1-\rho)\sigma_{k-1}^2 + D_2 \right] 
    \notag\\
    &&\quad + \frac{2\gamma B}{K} \left(1 + \frac{1}{\rho} \right) \sigma^2_0 - \frac{2\gamma B}{\rho K} \sum\limits_{k=0}^{K-1} \Exp\left[\sigma_k^2 \right] \notag\\
    &\leq& \frac{3\left[\max_{u \in \mathcal{C}}\|x^{0} - u\|^2\right]}{2\gamma K}  + \frac{8\gamma \ell^2 \Omega_{\mathcal{C}}^2}{K} + \frac{2\gamma B(1 + \nicefrac{1}{\rho})}{K}\sigma^2_0 \notag\\
    &&\quad + 2\gamma \left(D_1 + B (1 + \nicefrac{1}{\rho})D_2 \right) \notag\\
    &&\quad + 9\gamma \max\limits_{x^* \in X^*}\|F(x^*)\|^2  +  \frac{\gamma (4A + \ell)}{K}\sum\limits_{k=0}^{K-1}  \Exp\left[\langle F(x^k) - g^{*,k}, x^k - x^{*,k}\rangle\right] \notag\\
    &&\quad + \frac{2\gamma B}{K} \left(1 + \frac{1}{\rho} \right) \sum\limits_{k=0}^{K-2} \Exp\left[2C\langle F(x^{k}) - g^{*,k}, x^{k} - x^{*,k}\rangle+ (1-\rho)\sigma_{k}^2\right]
    \notag\\
    &&\quad - \frac{2\gamma B}{\rho K} \sum\limits_{k=0}^{K-1} \Exp\left[\sigma_k^2 \right] \notag\\
    &\leq& \frac{3\left[\max_{u \in \mathcal{C}}\|x^{0} - u\|^2\right]}{2\gamma K}  + \frac{8\gamma \ell^2 \Omega_{\mathcal{C}}^2}{K} + \frac{2\gamma B(1 + \nicefrac{1}{\rho})}{K}\sigma^2_0 \notag\\
    &&\quad + 2\gamma \left(D_1 + B (1 + \nicefrac{1}{\rho})D_2 \right) + 9\gamma \max\limits_{x^* \in X^*}\|F(x^*)\|^2 \notag\\
    &&\quad +  \left( 4A + \ell + 4BC(1+\nicefrac{1}{\rho})\right) \frac{\gamma}{K} \sum\limits_{k=0}^{K-1}  \Exp\left[\langle F(x^k) - g^{*,k}, x^k - x^{*,k}\rangle\right] \notag\\
    &&\quad + \frac{2\gamma B}{K} (1-\rho)\left(1 + \frac{1}{\rho} \right) \sum\limits_{k=0}^{K-2} \Exp\left[\sigma_k^2 \right] - \frac{2\gamma B}{\rho K} \sum\limits_{k=0}^{K-1} \Exp\left[\sigma_k^2 \right].
\end{eqnarray*}
Since $(1-\rho)\left(1 + \nicefrac{1}{\rho} \right) = -\rho + \nicefrac{1}{\rho} \leq \nicefrac{1}{\rho}$, the last row is non-positive and we have
\begin{eqnarray}
    \label{temp3_monotone_appendix}
    \Exp\left[\text{Gap}_{\cC} \left(\frac{1}{K}\sum\limits_{k=1}^{K}  x^{k}\right)\right]
    &\leq& \frac{3\left[\max_{u \in \mathcal{C}}\|x^{0} - u\|^2\right]}{2\gamma K}  + \frac{8\gamma \ell^2 \Omega_{\mathcal{C}}^2}{K} + \frac{2\gamma B(1 + \nicefrac{1}{\rho})}{K}\sigma^2_0 \notag \\
    &&\quad + 2\gamma \left(D_1 + B (1 + \nicefrac{1}{\rho})D_2 \right) + 9\gamma \max\limits_{x^* \in X^*}\|F(x^*)\|^2\\
    &&\quad + \frac{\gamma\left( 4A + \ell + 4BC(1+\nicefrac{1}{\rho})\right)}{K} \sum\limits_{k=0}^{K-1}  \Exp\left[\langle F(x^k) - g^{*,k}, x^k - x^{*,k}\rangle\right].\notag
\end{eqnarray}

Note that inequality \eqref{temp0_str_monotone_appendix} from the proof of Theorem~\ref{thm:main_result} is derived using Assumption~\ref{as:key_assumption} only. With $M = \nicefrac{B}{\rho}$ it gives
\begin{eqnarray*}
\Exp \left[ \|x^{k+1}-x^{*,k+1}\|^2 \right] +\frac{\gamma^2 B}{\rho} \Exp [\sigma_{k+1}^2] &\leq & \Exp \left[\|x^k-x^{*,k}\|^2\right] + \frac{\gamma^2B}{\rho}\Exp \left[\sigma_k^2\right] + \gamma^2(D_1 + \nicefrac{BD_2}{\rho})\notag\\
&&\quad -2\gamma\left(1 - \gamma(A + \nicefrac{BC}{\rho})\right)\Exp\left[\left\langle x^k-x^{*,k}, F(x^k) - g^{*,k}\right\rangle\right]. 
\end{eqnarray*}
Since $\gamma \leq \nicefrac{1}{2(A + BC/\rho)}$ we obtain
\begin{eqnarray*}
\gamma\Exp\left[\left\langle x^k-x^{*,k}, F(x^k) - g^{*,k}\right\rangle\right] &\leq & \Exp \left[\|x^k-x^{*,k}\|^2\right] + \frac{\gamma^2B}{\rho}\Exp \left[\sigma_k^2\right] - \Exp \left[ \|x^{k+1}-x^{*,k+1}\|^2 \right]  \\
&&\quad - \frac{\gamma^2B}{\rho} \Exp [\sigma_{k+1}^2] + \gamma^2(D_1 + \nicefrac{BD_2}{\rho}). 
\end{eqnarray*}
Plugging this inequality in \eqref{temp3_monotone_appendix}, we derive 
\begin{eqnarray}
    \label{eq:mon_temp3}
    \Exp\left[\text{Gap}_{\cC} \left(\frac{1}{K}\sum\limits_{k=1}^{K}  x^{k}\right)\right] 
    &\leq& \frac{3\left[\max_{u \in \mathcal{C}}\|x^{0} - u\|^2\right]}{2\gamma K}  + \frac{8\gamma \ell^2 \Omega_{\mathcal{C}}^2}{K} + \frac{2\gamma B(1 + \nicefrac{1}{\rho})}{K}\sigma^2_0\notag \\
    &&\quad + 2\gamma \left(D_1 + B (1 + \nicefrac{1}{\rho})D_2 \right) + 9\gamma \max\limits_{x^* \in X^*}\|F(x^*)\|^2\notag\\
    &&\quad +  \left( 4A + \ell + 4BC(1+\nicefrac{1}{\rho})\right) \cdot \frac{1}{K} \sum\limits_{k=0}^{K-1} \Exp \left[\|x^k-x^{*,k}\|^2\right] \notag\\
    &&\quad -  \left( 4A + \ell + 4BC(1+\nicefrac{1}{\rho})\right) \cdot \frac{1}{K} \sum\limits_{k=0}^{K-1} \Exp \left[\|x^{k+1}-x^{*,k+1}\|^2 \right] \notag\\
    &&\quad +  \left( 4A + \ell + 4BC(1+\nicefrac{1}{\rho})\right) \cdot \frac{\gamma^2B}{\rho K} \sum\limits_{k=0}^{K-1} \Exp \left[\sigma_k^2 - \sigma_{k+1}^2\right] \notag\\
    &&\quad +  \gamma^2 \left( 4A + \ell + 4BC(1+\nicefrac{1}{\rho})\right) \cdot  (D_1 + \nicefrac{BD_2}{\rho}) \notag\\
    &\leq& \frac{3\left[\max_{u \in \mathcal{C}}\|x^{0} - u\|^2\right]}{2\gamma K}  + \frac{8\gamma \ell^2 \Omega_{\mathcal{C}}^2}{K} + \left( 4A + \ell + \nicefrac{8BC}{\rho}\right) \cdot \frac{\|x^0-x^{*,0}\|^2}{K} \notag\\
    &&\quad + \left(4 +  \left( 4A + \ell + \nicefrac{8BC}{\rho}\right) \gamma \right) \frac{\gamma B \sigma_0^2}{\rho K} \\
    &&\quad
    +\gamma \left((2 + \gamma\left( 4A + \ell + \nicefrac{8BC}{\rho}\right))(D_1 + \nicefrac{2BD_2}{\rho}) + 9\max\limits_{x^* \in X^*}\|F(x^*)\|^2 \right) \notag,
\end{eqnarray}
where in the last inequality we use $1 + \nicefrac{1}{\rho} \leq \nicefrac{2}{\rho}$.
\end{proof}

\begin{corollary}\label{cor:main_result_monotone}
    Let the assumptions of Theorem~\ref{thm:main_result_monotone} hold. Then, for all $K$ one can choose $\gamma$ as
    \begin{equation}
        \gamma = \min\left\{\frac{1}{4A + \ell + \nicefrac{8BC}{\rho}}, \frac{\Omega_{0,\cC}\sqrt{\rho}}{\widehat{\sigma}_0\sqrt{B}}, \frac{\Omega_{0,\cC}}{\sqrt{K(D_1 + \nicefrac{2BD_2}{\rho})}}, \frac{\Omega_{0,\cC}}{G_*\sqrt{K}}\right\}, \label{eq:stepsize_choice_cor_monotone}
    \end{equation}
    where $\Omega_{0} \eqdef \|x^0 - x^{*,0}\|^2$ and $\Omega_{0,\cC}$, $\widehat{\sigma}_0$, and $G_*$ are some upper bounds for $\max_{u \in \cC}\|x^0 - u\|$, $\sigma_0$, and $\max_{x^* \in X^*}\|F(x^*)\|$ respectively. This choice of $\gamma$ implies $\Exp\left[\text{Gap}_{\cC} \left(\tfrac{1}{K}\sum_{k=1}^{K}  x^{k}\right)\right]$ equals
    \begin{align}
        \cO\left(\frac{(A + \ell + \nicefrac{BC}{\rho})(\Omega_{0,\cC}^2 + \Omega_0^2) + \ell \Omega_{\cC}^2}{K} + \frac{\Omega_{0,\cC} \widehat{\sigma}_{0}\sqrt{B}}{\sqrt{\rho} K} + \frac{\Omega_{0,\cC}(\sqrt{D_1 + \nicefrac{BD_2}{\rho}} + G_*)}{\sqrt{K}}\right).\notag
    \end{align}
\end{corollary}
\begin{proof}
    First of all, the choice of $\gamma$ from \eqref{eq:stepsize_choice_cor_monotone} implies \eqref{eq:gamma_condition_monotone_appendix} since
    \begin{equation*}
        \frac{1}{4A + \ell + \nicefrac{8BC}{\rho}} \leq \frac{1}{2\left(A + \nicefrac{BC}{\rho}\right)}.
    \end{equation*}
    Using \eqref{eq:main_result_monotone}, the definitions of $\Omega_{0,\cC}$, $\widehat{\sigma}_0$, $G_*$, and $\gamma \leq \nicefrac{1}{(4A + \ell + \nicefrac{8BC}{\rho})}$, we get 
    \begin{eqnarray*}
    \Exp\left[\text{Gap}_{\cC} \left(\frac{1}{K}\sum\limits_{k=1}^{K}  x^{k}\right)\right] 
    &\leq& \frac{3\left[\max_{u \in \mathcal{C}}\|x^{0} - u\|^2\right]}{2\gamma K}  + \frac{8\gamma \ell^2 \Omega_{\mathcal{C}}^2}{K} + \left( 4A + \ell + \nicefrac{8BC}{\rho}\right) \cdot \frac{\|x^0-x^{*,0}\|^2}{K} \notag\\
    &&\quad + \left(4 +  \left( 4A + \ell + \nicefrac{8BC}{\rho}\right) \gamma \right) \frac{\gamma B \sigma_0^2}{\rho K} \notag\\
    &&\quad +\gamma \left((2 + \gamma\left( 4A + \ell + \nicefrac{8BC}{\rho}\right))(D_1 + \nicefrac{2BD_2}{\rho}) + 9\max\limits_{x^* \in X^*}\|F(x^*)\|^2 \right)\\
    &\leq& \frac{3\Omega_{0,\cC}^2}{2\gamma K}  + \frac{8\gamma \ell^2 \Omega_{\mathcal{C}}^2}{K} + \frac{(4A + \ell + \nicefrac{8BC}{\rho})\Omega_0^2}{K} \notag\\
    &&\quad + \left(4 +  \left( 4A + \ell + \nicefrac{8BC}{\rho}\right) \gamma \right) \frac{\gamma B \widehat{\sigma}_0^2}{\rho K} \notag\\
    &&\quad +\gamma \left((2 + \gamma\left( 4A + \ell + \nicefrac{8BC}{\rho}\right))(D_1 + \nicefrac{2BD_2}{\rho}) + 9G_*^2 \right)\\
    &\leq& \frac{3\Omega_{0,\cC}^2}{2\gamma K}  + \frac{8\gamma \ell^2 \Omega_{\mathcal{C}}^2}{K} + \frac{(4A + \ell + \nicefrac{8BC}{\rho})\Omega_0^2}{K}  + \frac{5\gamma B \widehat{\sigma}_0^2}{\rho K} \\
    &&\quad +3\gamma \left(D_1 + \frac{2BD_2}{\rho} + 3G_*^2 \right).
    \end{eqnarray*}
    Finally, we apply \eqref{eq:stepsize_choice_cor_monotone}:
    \begin{eqnarray*}
        \Exp\left[\text{Gap}_{\cC} \left(\frac{1}{K}\sum\limits_{k=1}^{K}  x^{k}\right)\right] &\leq&  \frac{3\Omega_{0,\cC}^2}{2\min\left\{\frac{1}{4A + \ell + \nicefrac{8BC}{\rho}}, \frac{\Omega_{0,\cC}\sqrt{\rho}}{\widehat{\sigma}_0\sqrt{B}}, \frac{\Omega_{0,\cC}}{\sqrt{K(D_1 + \nicefrac{2BD_2}{\rho})}}, \frac{\Omega_{0,\cC}}{G_*\sqrt{K}}\right\} K} + \frac{1}{\ell} \cdot \frac{8\ell^2\Omega_{\cC}^2}{K} \\
        &&\quad + \frac{(4A + \ell + \nicefrac{8BC}{\rho})\Omega_0^2}{K} + \frac{\Omega_{0,\cC}\sqrt{\rho}}{\widehat{\sigma}_0\sqrt{B}}\cdot \frac{\gamma B \widehat{\sigma}_0^2}{\rho K}\\
        &&\quad + \frac{\Omega_{0,\cC}}{\sqrt{K(D_1 + \nicefrac{2BD_2}{\rho})}}\cdot 3\left(D_1 + \frac{2BD_2}{\rho}\right) + \frac{\Omega_{0,\cC}}{G_*\sqrt{K}} \cdot 9G_*^2\\
        &=& \cO\Bigg(\frac{(A + \ell + \nicefrac{BC}{\rho})(\Omega_{0,\cC}^2 + \Omega_0^2) + \ell \Omega_{\cC}^2}{K} + \frac{\Omega_{0,\cC} \widehat{\sigma}_{0}\sqrt{B}}{\sqrt{\rho} K}\\
        &&\quad\quad\quad\quad\quad\quad\quad\quad\quad\quad\quad\quad\quad\quad\quad+ \frac{\Omega_{0,\cC}(\sqrt{D_1 + \nicefrac{BD_2}{\rho}} + G_*)}{\sqrt{K}}\Bigg).
    \end{eqnarray*}
\end{proof}

\newpage

\subsection{Cocoercive Case}\label{sec:monotone_coco_appendix}

The upper bound from Theorem~\ref{thm:main_result_monotone} contains the term proportional to $\max_{x^* \in X^*}\|F(x^*)\|^2$, which is non-zero in general. Therefore, even when there is no noise the method with constant stepsize converges only to some error proportional to $\max_{x^* \in X^*}\|F(x^*)\|^2$. To resolve this issue we assume $\ell$-cocoercivity of $F$, i.e., we assume that
\begin{equation*}
    \|F(x) - F(y)\|^2 \le  \ell \langle F(x) - F(y), x - y \rangle \quad \forall x,y \in \R^d.
\end{equation*}

\begin{theorem}[Theorem \ref{thm:main_result_monotone_coco}]\label{thm:main_result_monotone_coco_appendix}
    Let $F$ be $\ell$-cocoercive and Assumptions~\ref{as:key_assumption}, \ref{as:boundness} hold. Assume that
    \begin{equation}
        \label{eq:gamma_condition_monotone_appendix_coco}
        0 < \gamma \leq \min\left\{\frac{1}{\ell}, \frac{1}{2(A + \nicefrac{BC}{\rho})}\right\}.
    \end{equation}
    Then for the function $\text{Gap}_{\cC} (z)$ from \eqref{eq:gap} and for all $K\ge 0$ we have
\begin{eqnarray}
    \label{eq:main_result_monotone_coco_appendix}
    \Exp\left[\text{Gap}_{\cC} \left(\frac{1}{K}\sum\limits_{k=1}^{K}  x^{k}\right)\right] 
    &\leq& \frac{3\max_{u \in \mathcal{C}}\|x^{0} - u\|^2}{2\gamma K}  + \left( 6A + 3\ell + \nicefrac{12BC}{\rho}\right) \cdot \frac{\|x^0-x^{*,0}\|^2}{K} \notag\\
    &&\quad + \left(6 +  \left( 6A + 3\ell + \nicefrac{12BC}{\rho}\right) \gamma \right) \frac{\gamma B \sigma_0^2}{\rho K} \\
    &&\quad
    +\gamma (3 + \gamma\left( 6A + 3\ell + \nicefrac{12BC}{\rho}\right))(D_1 + \nicefrac{2BD_2}{\rho}) \notag.
\end{eqnarray}
\end{theorem}

\begin{proof}
We start the proof from \eqref{eq:temp1_mon}.
\begin{eqnarray*}
    2\gamma \left( \langle F(x^k), x^k - u \rangle + R(x^{k+1}) - R(u) \right) &\leq& 
    \|x^{k+1} - x^k\|^2 + \|x^{k} - u\|^2 - \| x^{k+1} - u\|^2 \notag\\
    &&\quad + 2\gamma \langle F(x^k) - g^k, x^k - u \rangle \notag\\
    &&\quad-2\| x^{k+1} - x^k\|^2 + 2\gamma\langle g^k, x^{k} - x^{k+1} \rangle \\
    &=& \|x^{k} - u\|^2 - \| x^{k+1} - u\|^2 \notag\\
    &&\quad + 2\gamma \langle F(x^k) - g^k, x^k - u \rangle \notag\\
    &&\quad-\| x^{k+1} - x^k\|^2  + 2\gamma\langle F(u), x^{k} - x^{k+1} \rangle\\
    &&\quad  + 2\gamma\langle g^k - F(u), x^{k} - x^{k+1} \rangle.
\end{eqnarray*}
Then, due to $2\langle  a, b \rangle \leq \|a\|^2 + \|b\|^2$ we have
\begin{eqnarray*}
    2\gamma \left( \langle F(x^k), x^k - u \rangle + R(x^{k+1}) - R(u) \right) &\leq& 
    \|x^{k} - u\|^2 - \| x^{k+1} - u\|^2 \notag\\
    &&\quad + 2\gamma \langle F(x^k) - g^k, x^k - u \rangle \notag\\
    &&\quad-\| x^{k+1} - x^k\|^2  + 2\gamma\langle F(u), x^{k} - x^{k+1} \rangle\\
    &&\quad  + \gamma^2 \| g^k - F(u) \|^2 + \| x^{k} - x^{k+1} \|^2 \\
    &=& \|x^{k} - u\|^2 - \| x^{k+1} - u\|^2 \notag\\
    &&\quad + 2\gamma \langle F(x^k) - g^k, x^k - u \rangle \notag\\
    &&\quad + 2\gamma\langle F(u), x^{k} - x^{k+1} \rangle  + \gamma^2 \| g^k - F(u) \|^2.
\end{eqnarray*}
Next, we add $2 \gamma \left( \langle F(u), x^{k+1} - u\rangle - \langle F(x^k), x^k - u \rangle \right)$ to both sides of the previous inequality.
\begin{eqnarray*}
    2\gamma \left( \langle F(u), x^{k+1} - u \rangle + R(x^{k+1}) - R(u) \right) &\leq& 
    \|x^{k} - u\|^2 - \| x^{k+1} - u\|^2 \notag\\
    &&\quad + 2\gamma \langle F(u) - g^k, x^k - u \rangle   + \gamma^2 \| g^k - F(u) \|^2 \\
    &\leq& \|x^{k} - u\|^2 - \| x^{k+1} - u\|^2 \notag\\
    &&\quad - 2\gamma \langle F(x^k) - F(u), x^k - u \rangle \\
    &&\quad - 2\gamma \langle g^k - F(x^k), x^k - u \rangle \\
    &&\quad + 2\gamma^2 \| g^k - F(x^k) \|^2 + 2\gamma^2 \| F(x^k) - F(u) \|^2.
\end{eqnarray*}
Using that $F$ is $\ell$-co-cocoercive, we get
\begin{eqnarray*}
    2\gamma \left( \langle F(u), x^{k+1} - u \rangle + R(x^{k+1}) - R(u) \right) 
    &\leq& \|x^{k} - u\|^2 - \| x^{k+1} - u\|^2 \notag\\
    &&\quad - \frac{2\gamma}{\ell} \| F(x^k) - F(u) \|^2 \\
    &&\quad - 2\gamma \langle g^k - F(x^k), x^k - u \rangle \\
    &&\quad + 2\gamma^2 \| g^k - F(x^k) \|^2 + 2\gamma^2 \| F(x^k) - F(u) \|^2 \\
    &=& \|x^{k} - u\|^2 - \| x^{k+1} - u\|^2 \notag\\
    &&\quad - \frac{2\gamma}{\ell} \left( 1 - \gamma \ell \right) \| F(x^k) - F(u) \|^2 \\
    &&\quad - 2\gamma \langle g^k - F(x^k), x^k - u \rangle \\
    &&\quad + 2\gamma^2 \| g^k - F(x^k) \|^2.
\end{eqnarray*}
With $\gamma \leq \tfrac{1}{\ell}$, we have
\begin{eqnarray*}
    2\gamma \left( \langle F(u), x^{k+1} - u \rangle + R(x^{k+1}) - R(u) \right) 
    &\leq& \|x^{k} - u\|^2 - \| x^{k+1} - u\|^2 \notag\\
    &&\quad - 2\gamma \langle g^k - F(x^k), x^k - u \rangle \\
    &&\quad + 2\gamma^2 \| g^k - F(x^k) \|^2.
\end{eqnarray*}
Summing up the above inequality for $k = 0,1,\ldots, K-1$, we get
\begin{eqnarray*}
    2\gamma \sum\limits_{k=0}^{K-1} \left( \langle F(u), x^{k+1} - u \rangle + R(x^{k+1}) - R(u) \right) &\leq& 
    \sum\limits_{k=0}^{K-1}\|x^{k} - u\|^2 - \sum\limits_{k=0}^{K-1}\| x^{k+1} - u\|^2\\
    &&\quad + 2\gamma^2\sum\limits_{k=0}^{K-1}\| g^k - F(x^k) \|^2\\
    &&\quad + 2\gamma \sum\limits_{k=0}^{K-1} \langle F(x^k) - g^k, x^k - u \rangle.
\end{eqnarray*}
Next, we divide both sides by $2\gamma K$ 
\begin{eqnarray*}
    \frac{1}{K} \sum\limits_{k=0}^{K-1} \left( \langle F(u), x^{k+1} - u \rangle + R(x^{k+1}) - R(u) \right) &\leq& 
    \frac{\|x^{0} - u\|^2 - \| x^{K} - u\|^2 }{2\gamma K}\\
    &&\quad + \frac{\gamma}{K}\sum\limits_{k=0}^{K-1}\| g^k - F(x^k) \|^2\\
    &&\quad + \frac{1}{K} \sum\limits_{k=0}^{K-1} \langle F(x^k) - g^k, x^k - u \rangle.
\end{eqnarray*}
Applying Jensen's inequality for convex function $R$, we get $R\left(\frac{1}{K}\sum_{k=0}^{K-1}  x^{k+1}\right) \le \frac{1}{K}\sum_{k=0}^{K-1}R(  x^{k+1} )$.
\begin{align*}
    \left\langle F(u), \left(\frac{1}{K}\sum\limits_{k=0}^{K-1}  x^{k+1}\right) - u \right\rangle + R\left(\frac{1}{K}\sum\limits_{k=0}^{K-1}  x^{k+1}\right) -& R(u) \\ 
    \leq& 
    \frac{\|x^{0} - u\|^2 - \| x^{K} - u\|^2 }{2\gamma K}\\
    &\quad + \frac{\gamma}{K}\sum\limits_{k=0}^{K-1}\| g^k - F(x^k) \|^2\\
    &\quad + \frac{1}{K} \sum\limits_{k=0}^{K-1} \langle F(x^k) - g^k, x^k - u \rangle.
\end{align*}
Next, we take maximum from the both sides in $u \in \mathcal{C}$, which gives $\text{Gap}_{\cC} \left(\frac{1}{K}\sum_{k=1}^{K}  x^{k}\right)$ in the left-hand side by definition \eqref{eq:gap}, and take the expectation of the result:
\begin{eqnarray*}
    \Exp\left[\text{Gap}_{\cC} \left(\frac{1}{K}\sum\limits_{k=1}^{K}  x^{k}\right)\right] &\leq&      \frac{\Exp\left[\max_{u \in \mathcal{C}}\|x^{0} - u\|^2\right]}{\gamma K}+ \frac{\gamma}{K} \sum\limits_{k=0}^{K-1} \Exp\left[\| g^k - F(x^k) \|^2\right]\notag\\
    &&\quad+ \frac{1}{K} \Exp\left[\max_{u \in \mathcal{C}}\sum\limits_{k=0}^{K-1} \langle F(x^k) - g^k, x^k - u \rangle \right].
\end{eqnarray*}
Using the estimate \eqref{temp_2_monotone_appendix}, we get
\begin{eqnarray*}
    \Exp\left[\text{Gap}_{\cC} \left(\frac{1}{K}\sum\limits_{k=1}^{K}  x^{k}\right)\right] &\leq&      \frac{\max_{u \in \mathcal{C}}\|x^{0} - u\|^2}{\gamma K}+ \frac{\gamma}{K} \sum\limits_{k=0}^{K-1} \Exp\left[\| g^k - F(x^k) \|^2\right]\notag\\
    &&\quad+ \frac{\gamma}{2K}\sum\limits_{k=0}^{K-1}\Exp[\|F(x^k) - g^k \|^2] + \frac{1}{2\gamma K}\max\limits_{u\in \cC} \|x^0 - u\|^2 \\
    &\leq&      \frac{3\max_{u \in \mathcal{C}}\|x^{0} - u\|^2}{2\gamma K}+ \frac{3\gamma}{2K} \sum\limits_{k=0}^{K-1} \Exp\left[\| g^k - F(x^k) \|^2\right].
\end{eqnarray*}
It remains to estimate $\frac{1}{K} \sum\limits_{k=0}^{K-1} \Exp\left[\| g^k - F(x^k) \|^2\right]$. This was done in the previous proof (see from \eqref{eq:mon_temp1} to \eqref{eq:mon_temp3}). Then, we finally have 
\begin{eqnarray*}
    \Exp\left[\text{Gap}_{\cC} \left(\frac{1}{K}\sum\limits_{k=1}^{K}  x^{k}\right)\right] 
    &\leq& \frac{3\max_{u \in \mathcal{C}}\|x^{0} - u\|^2}{2\gamma K}  + \left( 6A + 3\ell + \nicefrac{12BC}{\rho}\right) \cdot \frac{\|x^0-x^{*,0}\|^2}{K} \notag\\
    &&\quad + \left(6 +  \left( 6A + 3\ell + \nicefrac{12BC}{\rho}\right) \gamma \right) \frac{\gamma B \sigma_0^2}{\rho K} \\
    &&\quad
    +\gamma (3 + \gamma\left( 6A + 3\ell + \nicefrac{12BC}{\rho}\right))(D_1 + \nicefrac{2BD_2}{\rho}) .
\end{eqnarray*}
\end{proof}

\begin{corollary}\label{cor:main_result_monotone_coco}
    Let the assumptions of Theorem~\ref{thm:main_result_monotone_coco_appendix} hold. Then, for all $K$ one can choose $\gamma$ as
    \begin{equation}
        \gamma = \min\left\{\frac{1}{6A + 3\ell + \nicefrac{12BC}{\rho}}, \frac{\Omega_{0,\cC}\sqrt{\rho}}{\widehat{\sigma}_0\sqrt{B}}, \frac{\Omega_{0,\cC}}{\sqrt{K(D_1 + \nicefrac{2BD_2}{\rho})}}\right\}, \label{eq:stepsize_choice_cor_monotone_coco}
    \end{equation}
    where $\Omega_{0} \eqdef \|x^0 - x^{*,0}\|^2$ and $\Omega_{0,\cC}$, and $\widehat{\sigma}_0$ are some upper bounds for $\max_{u \in \cC}\|x^0 - u\|$, and $\sigma_0$ respectively. This choice of $\gamma$ implies $\Exp\left[\text{Gap}_{\cC} \left(\tfrac{1}{K}\sum_{k=1}^{K}  x^{k}\right)\right]$ equals
    \begin{align}
        \cO\left(\frac{(A + \ell + \nicefrac{BC}{\rho})(\Omega_{0,\cC}^2 + \Omega_0^2)}{K} + \frac{\Omega_{0,\cC} \widehat{\sigma}_{0}\sqrt{B}}{\sqrt{\rho} K} + \frac{\Omega_{0,\cC}\sqrt{D_1 + \nicefrac{BD_2}{\rho}}}{\sqrt{K}}\right).\notag
    \end{align}
\end{corollary}
\begin{proof}
    First of all, the choice of $\gamma$ from \eqref{eq:stepsize_choice_cor_monotone_coco} implies \eqref{eq:gamma_condition_monotone_appendix} since
    \begin{equation*}
        \frac{1}{6A + 3\ell + \nicefrac{12BC}{\rho}} \leq \frac{1}{2\left(A + \nicefrac{BC}{\rho}\right)}.
    \end{equation*}
    Using \eqref{eq:main_result_monotone_coco_appendix}, the definitions of $\Omega_{0,\cC}$, $\widehat{\sigma}_0$, and $\gamma \leq \nicefrac{1}{(6A + 3\ell + \nicefrac{12BC}{\rho})}$, we get 
    \begin{eqnarray*}
    \Exp\left[\text{Gap}_{\cC} \left(\frac{1}{K}\sum\limits_{k=1}^{K}  x^{k}\right)\right] 
    &\leq& \frac{3\max_{u \in \mathcal{C}}\|x^{0} - u\|^2}{2\gamma K}  + \left( 6A + 3\ell + \nicefrac{12BC}{\rho}\right) \cdot \frac{\|x^0-x^{*,0}\|^2}{K} \notag\\
    &&\quad + \left(6 +  \left( 6A + 3\ell + \nicefrac{12BC}{\rho}\right) \gamma \right) \frac{\gamma B \sigma_0^2}{\rho K} \\
    &&\quad
    +\gamma (3 + \gamma\left( 6A + 3\ell + \nicefrac{12BC}{\rho}\right))(D_1 + \nicefrac{2BD_2}{\rho}) \notag\\
    &\leq& \frac{3\Omega_{0,\cC}^2}{2\gamma K}  + \frac{(6A + 3\ell + \nicefrac{12BC}{\rho})\Omega_0^2}{K} \notag\\
    &&\quad + \left(6 +  \left( 6A + 3\ell + \nicefrac{12BC}{\rho}\right) \gamma \right) \frac{\gamma B \widehat{\sigma}_0^2}{\rho K} \notag\\
    &&\quad +\gamma (3 + \gamma\left( 6A + 3\ell + \nicefrac{12BC}{\rho}\right))(D_1 + \nicefrac{2BD_2}{\rho})\\
    &\leq& \frac{3\Omega_{0,\cC}^2}{2\gamma K}  + \frac{(6A + 3\ell + \nicefrac{12BC}{\rho})\Omega_0^2}{K}  + \frac{7\gamma B \widehat{\sigma}_0^2}{\rho K} +4\gamma \left(D_1 + \frac{2BD_2}{\rho}\right).
    \end{eqnarray*}
    Finally, we apply \eqref{eq:stepsize_choice_cor_monotone}:
    \begin{eqnarray*}
        \Exp\left[\text{Gap}_{\cC} \left(\frac{1}{K}\sum\limits_{k=1}^{K}  x^{k}\right)\right] &\leq&  \frac{3\Omega_{0,\cC}^2}{2\min\left\{\frac{1}{6A + 3\ell + \nicefrac{12BC}{\rho}}, \frac{\Omega_{0,\cC}\sqrt{\rho}}{\widehat{\sigma}_0\sqrt{B}}, \frac{\Omega_{0,\cC}}{\sqrt{K(D_1 + \nicefrac{2BD_2}{\rho})}}\right\} K} \\
        &&\quad + \frac{(6A + 3\ell + \nicefrac{12BC}{\rho})\Omega_0^2}{K} + \frac{\Omega_{0,\cC}\sqrt{\rho}}{\widehat{\sigma}_0\sqrt{B}}\cdot \frac{\gamma B \widehat{\sigma}_0^2}{\rho K}\\
        &&\quad + \frac{\Omega_{0,\cC}}{\sqrt{K(D_1 + \nicefrac{2BD_2}{\rho})}}\cdot 4\left(D_1 + \frac{2BD_2}{\rho}\right)\\
        &=& \cO\Bigg(\frac{(A + \ell + \nicefrac{BC}{\rho})(\Omega_{0,\cC}^2 + \Omega_0^2)}{K} + \frac{\Omega_{0,\cC} \widehat{\sigma}_{0}\sqrt{B}}{\sqrt{\rho} K}\\
        &&\quad\quad\quad\quad\quad\quad\quad\quad\quad\quad\quad\quad\quad\quad\quad+ \frac{\Omega_{0,\cC}\sqrt{D_1 + \nicefrac{BD_2}{\rho}}}{\sqrt{K}}\Bigg).
    \end{eqnarray*}
\end{proof}

\newpage

\section{\algname{SGDA} WITH ARBITRARY SAMPLING: MISSING PROOFS AND DETAILS}
\label{AppendixSGDA_AS}

\begin{algorithm}[h!]
   \caption{\algname{SGDA-AS}: Stochastic Gradient Descent-Ascent with Arvitrary Sampling}
   \label{alg:prox_SGDA}
\begin{algorithmic}[1]
   \State {\bfseries Input:} starting point $x^0 \in \R^d$, distribution $\cD$, stepsize $\gamma > 0$, number of steps $K$
   \For{$k=0$ {\bfseries to} $K-1$}
   \State Sample $\xi^k \sim \cD$ independently from previous iterations and compute $g^k = F_{\xi^k}(x^k)$
   \State $x^{k+1} = \prox_{\gamma R}(x^k - \gamma g^k)$
   \State 
   \EndFor
\end{algorithmic}
\end{algorithm}

\subsection{Proof of Proposition~\ref{thm:prox_SGDA_convergence}}

\begin{proposition}[Proposition~\ref{thm:prox_SGDA_convergence}]\label{thm:prox_SGDA_convergence_appendix_1}
    Let Assumption~\ref{as:expected_cocoercivity} hold. Then, \algname{SGDA} satisfies Assumption~\ref{as:key_assumption} with
    \begin{gather*}
        A = \ell_{\cD},\quad B = 0,\quad \sigma_k^2 \equiv 0,\quad D_1 = 2\sigma_*^2 \eqdef 2\max_{x^* \in X^*}\Exp_{\cD}\left[\|F_{\xi}(x^*) - F(x^*)\|^2\right],\\ C = 0,\quad \rho = 1,\quad D_2 = 0.    
    \end{gather*}
\end{proposition}
\begin{proof}
    To prove the result, it is sufficient to derive an upper bound for $\Exp_k\left[\|g^k - F(x^{*,k})\|^2\right]$:
    \begin{eqnarray*}
        \Exp_k\left[\|g^k - F(x^{*,k})\|^2\right] &=& \Exp_{\cD}\left[\|F_{\xi^k}(x^k) - F(x^{*,k})\|^2\right]\\
        &\leq& 2\Exp_{\cD}\left[\|F_{\xi^k}(x^k) - F_{\xi^k}(x^{*,k})\|^2\right] + 2\Exp_{\cD}\left[\|F_{\xi^k}(x^{*,k}) - F(x^{*,k})\|^2\right]\\
        &\overset{Ass. \eqref{as:expected_cocoercivity}}{\leq}& 2\ell_{\cD} \langle F(x^k) - F(x^{*,k}), x^k - x^{*,k} \rangle + 2\sigma_*^2,
    \end{eqnarray*}
    where $\sigma_*^2 \eqdef \max_{x^* \in X^*}\Exp_{\cD}\left[\|F_{\xi}(x^*) - F(x^*)\|^2\right]$. The above inequality implies that Assumption~\ref{as:key_assumption} holds with
    \begin{gather*}
        A = \ell_{\cD},\quad B = 0,\quad \sigma_k^2 \equiv 0,\quad D_1 = 2\sigma_*^2 \eqdef 2\max_{x^* \in X^*}\Exp_{\cD}\left[\|F_{\xi}(x^*) - F(x^*)\|^2\right],\\ C = 0,\quad \rho = 1,\quad D_2 = 0.    
    \end{gather*}
\end{proof}

\subsection{Analysis of \algname{SGDA-AS} in the Quasi-Strongly Monotone Case}

Plugging the parameters from the above proposition in Theorem~\ref{thm:main_result} and Corollary~\ref{cor:main_result} we get the following results.

\begin{theorem}\label{thm:prox_SGDA_convergence_appendix}
    Let $F$ be $\mu$-quasi strongly monotone, Assumption~\ref{as:expected_cocoercivity} hold, and $0 < \gamma \leq \nicefrac{1}{2\ell_{\cD}}$. Then, for all $k \ge 0$ the iterates produced by \algname{SGDA-AS} satisfy
    \begin{equation}
        \Exp\left[\|x^k - x^{*,k}\|^2\right] \leq (1 - \gamma\mu)^k \|x^0 - x^{0,*}\|^2 + \frac{2\gamma\sigma_*^2}{\mu}. \label{eq:prox_SGDA_convergence_appendix}
    \end{equation}
\end{theorem}

\begin{corollary}[Corollary~\ref{cor:prox_SGDA_convergence}]\label{cor:prox_SGDA_convergence_appendix}
    Let the assumptions of Theorem~\ref{thm:prox_SGDA_convergence_appendix} hold. Then, for any $K \ge 0$ one can choose $\{\gamma_k\}_{k \ge 0}$ as follows:
	\begin{eqnarray}
		\text{if } K \le \frac{2\ell_{\cD}}{\mu}, && \gamma_k = \frac{1}{2\ell_{\cD}},\notag\\
		\text{if } K > \frac{2\ell_{\cD}}{\mu} \text{ and } k < k_0, && \gamma_k = \frac{1}{2\ell_{\cD}},\label{eq:stepsize_choice_2_QSM_prox_SGDA}\\
		\text{if } K > \frac{2\ell_{\cD}}{\mu} \text{ and } k \ge k_0, && \gamma_k = \frac{2}{4\ell_{\cD} + \mu(k - k_0)},\notag
	\end{eqnarray}
	where $k_0 = \left\lceil \nicefrac{K}{2} \right\rceil$. For this choice of $\gamma_k$ the following inequality holds for \algname{SGDA-AS}:
	\begin{eqnarray}
		\Exp[\|x^{K} - x^{*,K}\|^2] \le \frac{64\ell_{\cD}}{\mu}\|x^{0} - x^{*,0}\|^2\exp\left(-\frac{\mu}{2\ell_{\cD}}K\right) + \frac{72\sigma_{*}^2}{\mu^2 K}.\notag
	\end{eqnarray}
\end{corollary}

\subsection{Analysis of \algname{SGDA-AS} in the Monotone Case}

In the monotone case, using Theorem~\ref{thm:main_result_monotone}, we establish the new result for \algname{SGDA-AS}.

\begin{theorem}\label{thm:prox_SGDA_convergence_monotone}
    Let $F$ be monotone $\ell$-star-cocoercive and Assumptions~\ref{as:key_assumption}, \ref{as:boundness}, \ref{as:expected_cocoercivity} hold. Assume that $\gamma \leq \nicefrac{1}{2\ell_{\cD}}$. Then for $\text{Gap}_{\cC} (z)$ from \eqref{eq:gap} and for all $K\ge 0$ the iterates produced by \algname{SGDA-AS} satisfy
    \begin{eqnarray}
    \label{eq:prox_SGDA_convergence_monotone}
    \Exp\left[\text{Gap}_{\cC} \left(\frac{1}{K}\sum\limits_{k=1}^{K}  x^{k}\right)\right] &\leq& \frac{3\max_{u \in \mathcal{C}}\|x^{0} - u\|^2}{2\gamma K}  + \frac{8\gamma \ell^2 \Omega_{\mathcal{C}}^2}{K} + \frac{\left( 4\ell_{\cD} + \ell\right)\|x^0-x^{*,0}\|^2}{K}\notag\\
    &&\quad + 2\gamma(2 + \gamma\left( 4\ell_{\cD} + \ell\right))\sigma_{*}^2 + 9\gamma\max_{x^*\in X^*}\|F(x^*)\|^2.\notag
\end{eqnarray}
\end{theorem}

Next, we apply Corollary~\ref{cor:main_result_monotone} and get the following rate of convergence to the exact solution.
\begin{corollary}\label{cor:prox_SGDA_convergence_monotone}
    Let the assumptions of Theorem~\ref{thm:prox_SGDA_convergence_monotone} hold. Then $\forall K > 0$ and
    \begin{equation}
    \gamma = \min\left\{\frac{1}{4\ell_{\cD} + \ell}, \frac{\Omega_{0,\cC}}{\sqrt{2K}\sigma_*}, \frac{\Omega_{0,\cC}}{G_*\sqrt{K}}\right\} \label{eq:stepsize_choice_cor_monotone_prox_SGDA}
    \end{equation}
    the iterates produced by \algname{SGDA-AS} satisfy
    \begin{align}
        \Exp\left[\text{Gap}_{\cC} \left(\frac{1}{K}\sum\limits_{k=1}^{K}  x^{k}\right)\right] = \cO\left(\frac{(\ell_{\cD} + \ell)(\Omega_{0,\cC}^2 + \Omega_0^2) + \ell \Omega_{\cC}^2}{K} + \frac{\Omega_{0,\cC}(\sigma_* + G_*)}{\sqrt{K}}\right).\notag
    \end{align}
\end{corollary}

As we already mentioned before, the above result is new for \algname{SGDA-AS}: the only known work on \algname{SGDA-AS} \citep{loizou2021stochastic} focuses on the $\mu$-quasi-strongly monotone case only with $\mu > 0$. Moreover, neglecting the dependence on problem/noise parameters, the derived convergence rate $\cO\left(\nicefrac{1}{K} + \nicefrac{1}{\sqrt{K}}\right)$ is standard for the analysis of stochastic methods for solving monotone VIPs \citep{juditsky2011solving}.

\subsection{Analysis of \algname{SGDA-AS} in the Cocoercive Case}

In the cocoercive case, using Theorem~\ref{thm:main_result_monotone_coco_appendix}, we establish the new result for \algname{SGDA-AS}.

\begin{theorem}\label{thm:prox_SGDA_convergence_monotone_coco}
    Let $F$ be $\ell$-cocoercive and Assumptions~\ref{as:key_assumption}, \ref{as:boundness}, \ref{as:expected_cocoercivity} hold. Assume that $\gamma \leq \nicefrac{1}{2\ell_{\cD}}$. Then for $\text{Gap}_{\cC} (z)$ from \eqref{eq:gap} and for all $K\ge 0$ the iterates produced by \algname{SGDA-AS} satisfy
    \begin{eqnarray}
    \label{eq:prox_SGDA_convergence_monotone_coco}
    \Exp\left[\text{Gap}_{\cC} \left(\frac{1}{K}\sum\limits_{k=1}^{K}  x^{k}\right)\right] &\leq& \frac{3\max_{u \in \mathcal{C}}\|x^{0} - u\|^2}{2\gamma K} + \frac{\left( 6\ell_{\cD} + 3\ell\right)\|x^0-x^{*,0}\|^2}{K}\notag\\
    &&\quad + 2\gamma(3 + \gamma\left( 6\ell_{\cD} + 3\ell\right))\sigma_{*}^2.\notag
\end{eqnarray}
\end{theorem}

Next, we apply Corollary~\ref{cor:main_result_monotone_coco} and get the following rate of convergence to the exact solution.
\begin{corollary}\label{cor:prox_SGDA_convergence_monotone_coco}
    Let the assumptions of Theorem~\ref{thm:prox_SGDA_convergence_monotone_coco} hold. Then $\forall K > 0$ and
    \begin{equation}
    \gamma = \min\left\{\frac{1}{6\ell_{\cD} + 3\ell}, \frac{\Omega_{0,\cC}}{\sqrt{2K}\sigma_*}\right\} \label{eq:stepsize_choice_cor_monotone_prox_SGDA_coco}
    \end{equation}
    the iterates produced by \algname{SGDA-AS} satisfy
    \begin{align}
        \Exp\left[\text{Gap}_{\cC} \left(\frac{1}{K}\sum\limits_{k=1}^{K}  x^{k}\right)\right] = \cO\left(\frac{(\ell_{\cD} + \ell)(\Omega_{0,\cC}^2 + \Omega_0^2)}{K} + \frac{\Omega_{0,\cC}\sigma_*}{\sqrt{K}}\right).\notag
    \end{align}
\end{corollary}

\subsection{Missing Details on Arbitrary Sampling}\label{sec:arb_sampl_sp_cases}
In the main part of the paper, we discuss the Arbitrary Sampling paradigm and, in particular, using our general theoretical framework, we obtain convergence guarantees for \algname{SGDA} under Expected Cocoercivity assumption (Assumption~\ref{as:expected_cocoercivity}). In this section, we give the particular examples of arbitrary sampling fitting this setup. In all the examples below, we focus on a special case of stochastic reformulation from \eqref{eq:stoch_reformulation} and assume that for all $i\in [n]$ operator $F_i$ is $(\ell_i,X^*)$-cocoercive, i.e., for all $i\in [n]$ and $x\in \R^d$ we have
\begin{equation}
    \|F_i(x) - F_i(x^*)\|^2 \leq \ell_i \langle F_i(x) - F_i(x^*), x - x^* \rangle, \label{eq:ell_i_cocoercivity}
\end{equation}
where $x^*$ is the projection of $x$ on $X^*$. Note that \eqref{eq:ell_i_cocoercivity} holds whenever $F_i$ are cocoercive.

\paragraph{Uniform Sampling.} We start with the classical uniform sampling: let $\PP\left\{\xi = n e_i\right\}  = \nicefrac{1}{n}$ for all $i\in [n]$, where $e_i \in \R^n$ is the $i$-th coordinate vector from the standard basis in $\R^n$. Then, $\Exp[\xi_i] = 1$ for all $i\in [n]$ and Assumption~\ref{as:expected_cocoercivity} holds with $\ell_{\cD} =  \max_{i\in[n]}\ell_i$:
\begin{eqnarray*}
    \Exp_{\cD}\left[\|F_{\xi}(x) - F_{\xi}(x^*)\|^2\right] &=& \frac{1}{n}\sum\limits_{i=1}^n\|F_i(x) - F_i(x^*)\|^2\\
    &\overset{\eqref{eq:ell_i_cocoercivity}}{\leq}& \frac{1}{n}\sum\left(\ell_i \langle F_i(x) - F_i(x^*), x - x^* \rangle\right)\\
    &\leq& \max\limits_{i\in[n]}\ell_i \langle F(x) - F(x^*), x - x^* \rangle
\end{eqnarray*}
In this case, Corollaries~\ref{cor:prox_SGDA_convergence_appendix} and \ref{cor:prox_SGDA_convergence_monotone} imply the following rate for \algname{SGDA} in $\mu$-quasi strongly monotone, monotone and cocoercive cases respectively:
\begin{eqnarray}
	\Exp[\|x^{K} - x^{*,K}\|^2] &\le& \frac{64\max_{i\in[n]}\ell_i}{\mu}\|x^{0} - x^{*,0}\|^2\exp\left(-\frac{\mu}{2\max_{i\in[n]}\ell_i}K\right) + \frac{72\sigma_{*,\text{US}}^2}{\mu^2 K},\notag\\
	\Exp\left[\text{Gap}_{\cC} \left(\frac{1}{K}\sum\limits_{k=1}^{K}  x^{k}\right)\right] &=& \cO\left(\frac{(\max_{i\in[n]}\ell_i + \ell)(\Omega_{0,\cC}^2 + \Omega_0^2) + \ell \Omega_{\cC}^2}{K} + \frac{\Omega_{0,\cC}(\sigma_{*,\text{US}} + G_*)}{\sqrt{K}}\right), \notag\\
	\Exp\left[\text{Gap}_{\cC} \left(\frac{1}{K}\sum\limits_{k=1}^{K}  x^{k}\right)\right] &=& \cO\left(\frac{(\max_{i\in[n]}\ell_i + \ell)(\Omega_{0,\cC}^2 + \Omega_0^2)}{K} + \frac{\Omega_{0,\cC}\sigma_{*,\text{US}}}{\sqrt{K}}\right),\notag
\end{eqnarray}
where $\sigma_{*,\text{US}}^2 \eqdef \max_{x^* \in X^*}\tfrac{1}{n}\sum_{i=1}^n\|F_{i}(x^*) - F(x^*)\|^2$.

\paragraph{Importance Sampling.} Next, we consider a non-uniform sampling strategy -- importance sampling: let $\PP\left\{\xi = e_i\nicefrac{n\overline{\ell}}{\ell_i}\right\}  = \nicefrac{\ell_i}{n\overline{\ell}}$ for all $i\in [n]$, where $\overline{\ell} = \tfrac{1}{n}\sum_{i=1}^n \ell_i$. Then, $\Exp[\xi_i] = 1$ for all $i\in [n]$ and Assumption~\ref{as:expected_cocoercivity} holds with $\ell_{\cD} =  \overline{\ell}$:
\begin{eqnarray*}
    \Exp_{\cD}\left[\|F_{\xi}(x) - F_{\xi}(x^*)\|^2\right] &=& \sum\limits_{i=1}^n\frac{\ell_i}{n\overline{\ell}}\left\|\frac{\overline{\ell}}{\ell_i}\left(F_i(x) - F_i(x^*)\right)\right\|^2\\
     &=& \sum\limits_{i=1}^n\frac{\overline{\ell}}{n\ell_i}\|F_i(x) - F_i(x^*)\|^2\\
    &\overset{\eqref{eq:ell_i_cocoercivity}}{\leq}& \frac{\overline{\ell}}{n}\sum\langle F_i(x) - F_i(x^*), x - x^* \rangle\\
    &\leq& \overline{\ell} \langle F(x) - F(x^*), x - x^* \rangle
\end{eqnarray*}
In this case, Corollaries~\ref{cor:prox_SGDA_convergence_appendix} and \ref{cor:prox_SGDA_convergence_monotone} imply the following rate for \algname{SGDA} in $\mu$-quasi strongly monotone, monotone and cocoercive cases respectively:
\begin{eqnarray}
	\Exp[\|x^{K} - x^{*,K}\|^2] &\le& \frac{64\overline{\ell}}{\mu}\|x^{0} - x^{*,0}\|^2\exp\left(-\frac{\mu}{2\overline{\ell}}K\right) + \frac{72\sigma_{*,\text{IS}}^2}{\mu^2 K},\notag\\
	\Exp\left[\text{Gap}_{\cC} \left(\frac{1}{K}\sum\limits_{k=1}^{K}  x^{k}\right)\right] &=& \cO\left(\frac{(\overline{\ell} + \ell)(\Omega_{0,\cC}^2 + \Omega_0^2) + \ell \Omega_{\cC}^2}{K} + \frac{\Omega_{0,\cC}(\sigma_{*,\text{IS}} + G_*)}{\sqrt{K}}\right), \notag \\
	\Exp\left[\text{Gap}_{\cC} \left(\frac{1}{K}\sum\limits_{k=1}^{K}  x^{k}\right)\right] &=& \cO\left(\frac{(\overline{\ell} + \ell)(\Omega_{0,\cC}^2 + \Omega_0^2)}{K} + \frac{\Omega_{0,\cC}\sigma_{*,\text{IS}}}{\sqrt{K}}\right)\notag
\end{eqnarray}
where $\sigma_{*,\text{IS}}^2 \eqdef \max_{x^* \in X^*}\tfrac{1}{n}\sum_{i=1}^n\tfrac{\ell_i}{\overline\ell}\left\|\tfrac{\overline\ell}{\ell_i}F_{i}(x^*) - F(x^*)\right\|^2$. We emphasize that $\overline{\ell} \leq \max_{i\in [n]}\ell_i$ and, in fact, $\overline{\ell}$ might be much smaller than $\max_{i\in [n]}\ell_i$. Therefore, compared to \algname{SGDA} with uniform sampling, \algname{SGDA} with importance sampling has better exponentially decaying term in the quasi-strongly monotone case and converges faster to the neighborhood, if executed with constant stepsize. Moreover, $\sigma_{*,\text{IS}}^2 \leq \sigma_{*,\text{US}}^2$, when $\max_{x^* \in X^*}\|F_i(x^*)\| \sim \ell_i$. In this case, \algname{SGDA} with importance sampling has better $\cO(\nicefrac{1}{K})$ term than \algname{SGDA} with uniform sampling as well.

\paragraph{Minibatch Sampling With Replacement.} Let $\xi = \tfrac{1}{b}\sum_{i=1}^b \xi^i$, where $\xi^i$ are i.i.d.\ samples from some distribution $\cD$ satisfying \eqref{eq:stoch_reformulation} and Assumption~\ref{as:expected_cocoercivity}. Then, the distribution of $\xi$ satisfies \eqref{eq:stoch_reformulation} and Assumption~\ref{as:expected_cocoercivity} as well with the same constant $\ell_{\cD}$. Therefore, minibatched versions of uniform sampling and importance sampling fit the framework as well with $\ell_{\cD} = \max_{i\in [n]}\ell_i, \sigma_*^2 = \tfrac{\sigma_{*,\text{US}}^2}{b}$ and $\ell_{\cD} = \overline \ell, \sigma_*^2 = \tfrac{\sigma_{*,\text{IS}}^2}{b}$.

\paragraph{Minibatch Sampling Without Replacement.} For given batchsize $b \in [n]$ we consider the following sampling strategy: for each subset $S\subseteq [n]$ such that $|S| = b$ we have $\PP\left\{\xi = \tfrac{n}{b}\sum_{i\in S} e_i\right\} = \frac{b!(n-b)!}{n!}$, i.e., $S$ is chosen uniformly at random from all $b$-element subsets of $[n]$. In the special case, when $R(x) \equiv 0$, \citet{loizou2021stochastic} show that this sampling strategy satisfies \eqref{eq:stoch_reformulation} and Assumption~\ref{as:expected_cocoercivity} with
\begin{equation}
    \ell_{\cD} = \frac{n(b-1)}{b(n-1)}\ell + \frac{n-b}{b(n-1)}\max_{i\in [n]}\ell_i,\quad \sigma_*^2 = \frac{n-b}{b(n-1)}\sigma_{*,\text{US}}^2. \label{eq:b_nice_sampling_parameters}
\end{equation}
Clearly, both parameters are smaller than corresponding parameters for minibatched version of uniform sampling with replacement, which indicates the theoretical benefits of sampling without replacement. Plugging the parameters from \eqref{eq:b_nice_sampling_parameters} in Corollaries~\ref{cor:prox_SGDA_convergence_appendix} and \ref{cor:prox_SGDA_convergence_monotone}, we get the rate of convergence for this sampling strategy. Moreover, in the quasi-strongly monotone case, to guarantee $\Exp[\|x^{K} - x^{*,K}\|^2] \leq \varepsilon$ for some $\varepsilon > 0$, the method requires
\begin{eqnarray}
    Kb &=& \cO\left( \max\left\{\left(b \frac{\ell}{\mu} + \frac{(n-b)}{n}\frac{\max_{i\in [n]}\ell_i}{\mu}\right)\log\frac{\ell_{\cD}\|x^0 - x^{*,0}\|^2}{\mu \varepsilon}, \frac{(n-b)\sigma_{*,\text{US}}^2}{n\mu^2\varepsilon}\right\}\right)\notag\\
    &=& \widetilde\cO\left( \max\left\{ \frac{b\left(\ell - \tfrac{1}{n}\max_{i\in [n]}\ell_i\right) + \max_{i\in [n]}\ell_i}{\mu}, \frac{(n-b)\sigma_{*,\text{US}}^2}{n\mu^2\varepsilon}\right\}\right)\quad \text{oracle calls,}\label{eq:b_nice_sampling_complexity}
\end{eqnarray}
where $\widetilde\cO(\cdot)$ hides numerical and logarithmic factors. One can notice that the first term in the maximum linearly increases in $b$ (since $\ell$ cannot be smaller than $\tfrac{1}{n}\max_{i\in [n]}\ell_i$), while the second term linearly decreases in $b$. The first term in the maximum is lower bounded by $\tfrac{(n-b)}{n}\tfrac{\max_{i\in [n]}\ell_i}{\mu}$. Therefore, if $\max_{i\in [n]}\ell_i \geq \tfrac{\sigma_{*,\text{US}}^2}{\mu\varepsilon}$, the the first term in the maximum is always larger than the second one, meaning that the optimal batchsize, i.e., the batchsize that minimizes oracle complexity \eqref{eq:b_nice_sampling_complexity} neglecting the logarithmic terms, equals $b_* = 1$. Next, if $\max_{i\in [n]}\ell_i < \tfrac{\sigma_{*,\text{US}}^2}{\mu\varepsilon}$, then there exists a positive value of $b$ such that the first term in the maximum equals the second term. This value equals
\begin{equation*}
    \frac{n\left(\sigma_{*,\text{US}}^2 - \mu\varepsilon \max_{i\in [n]}\ell_i\right)}{\sigma_{*}^2 + \mu\varepsilon\left(n\ell - \max_{i\in[n]}\ell_i\right)}.
\end{equation*}
One can easily verify that it is always smaller than $n$, but it can be non integer and it can be smaller than $1$ as well. Therefore, the optimal batchsize is
\begin{equation*}
    b_* = \begin{cases} 1 ,& \text{if } \max_{i\in [n]}\ell_i \geq \tfrac{\sigma_{*,\text{US}}^2}{\mu\varepsilon},\\ \max\left\{1 , \left\lfloor\frac{n\left(\sigma_{*,\text{US}}^2 - \mu\varepsilon \max_{i\in [n]}\ell_i\right)}{\sigma_{*}^2 + \mu\varepsilon\left(n\ell - \max_{i\in[n]}\ell_i\right)}\right\rfloor \right\}, &\text{otherwise.}\end{cases}
\end{equation*}
We notice that \citet{loizou2021stochastic} derive the following formula for the optimal batchsize (ignoring numerical constants):
\begin{equation*}
    \widetilde b_* = \begin{cases} 1 ,& \text{if } \max_{i\in [n]}\ell_i - \ell \geq \tfrac{\sigma_{*,\text{US}}^2}{\mu\varepsilon},\\ \max\left\{1 , \left\lfloor\frac{n\left(\sigma_{*,\text{US}}^2 - \mu\varepsilon(\max_{i\in [n]}\ell_i -\ell)\right)}{\sigma_{*}^2 + \mu\varepsilon\left(n\ell - \max_{i\in[n]}\ell_i\right)}\right\rfloor \right\}, &\text{otherwise.}\end{cases}
\end{equation*}
However, in terms of $\widetilde\cO(\cdot)$ both formulas give the same complexity result.

\newpage

\section{\algname{SGDA} WITH VARIANCE REDUCTION: MISSING PROOFS AND DETAILS}
In this section, we provide missing proofs and details for Section~\ref{sec:vr_sgda}.

\subsection{\algname{L-SVRGDA}}\label{sec:l_svrgda}

\begin{algorithm}[h!]
   \caption{\algname{L-SVRGDA}: Loopless Stochastic Variance Reduced Gradient Descent-Ascent}
   \label{alg:prox_L_SVRGDA}
\begin{algorithmic}[1]
   \State {\bfseries Input:} starting point $x^0 \in \R^d$, probability $p \in (0,1]$, stepsize $\gamma > 0$, number of steps $K$
   \State Set $w^0 = x^0$ and compute $F(w^0)$
   \For{$k=0$ {\bfseries to} $K-1$}
   \State Draw a fresh sample $j_k$ from the uniform distribution on $[n]$ and compute $g^k = F_{j_k}(x^k) - F_{j_k}(w^k) + F(w^k)$
   \State $w^{k+1} = \begin{cases} x^k, & \text{with probability } p,\\ w^k,& \text{with probability } 1-p,\end{cases}$
   \State $x^{k+1} = \prox_{\gamma R}(x^k - \gamma g^k)$
   \EndFor
\end{algorithmic}
\end{algorithm}

\subsubsection{Proof of Proposition~\ref{thm:prox_SVRGDA_convergence}}
\begin{lemma}\label{lem:L_SVRGDA_second_moment_bound}
    Let Assumption~\ref{as:averaged_cocoercivity} hold. Then for all $k \ge 0$ \algname{L-SVRGDA} satisfies
    \begin{equation}
        \Exp_k\left[\|g^k - F(x^{*,k})\|^2\right] \le 2\widehat{\ell} \langle F(x^k) - F(x^{*,k}), x^k - x^{*,k} \rangle + 2\sigma_k^2, \label{eq:L_SVRGDA_second_moment_bound}
    \end{equation}
    where $\sigma_k^2 \eqdef \frac{1}{n}\sum_{i=1}^n \|F_i(w^k) - F_i(x^{*,k})\|^2$.
\end{lemma}
\begin{proof}
    Since $g^k = F_{j_k}(x^k) - F_{j_k}(w^k) + F(w^k)$, we have
    \begin{eqnarray*}
        \Exp_k\left[\|g^k - F(x^{*,k})\|^2\right] &=& \Exp_k\left[\|F_{j_k}(x^k) - F_{j_k}(w^k) + F(w^k) - F(x^{*,k})\|^2\right]\\
        &=& \frac{1}{n}\sum\limits_{i=1}^n \|F_{i}(x^k) - F_{i}(w^k) + F(w^k) - F(x^{*,k})\|^2\\
        &\leq& \frac{2}{n}\sum\limits_{i=1}^n \|F_{i}(x^k) - F_i(x^{*,k})\|^2\\
        &&\quad + \frac{2}{n}\sum\limits_{i=1}^n \|F_{i}(w^k) - F_i(x^{*,k}) - (F(w^k) - F(x^{*,k}))\|^2\\
        &\leq& \frac{2}{n}\sum\limits_{i=1}^n \|F_{i}(x^k) - F_i(x^{*,k})\|^2 + \frac{2}{n}\sum\limits_{i=1}^n \|F_{i}(w^k) - F_i(x^{*,k})\|^2\\
        &\overset{\eqref{eq:averaged_cocoercivity}}{\leq}& 2\widehat{\ell} \langle F(x^k) - F(x^{*,k}), x^k - x^{*,k} \rangle + 2\sigma_k^2.
    \end{eqnarray*}
\end{proof}

\begin{lemma}\label{lem:L_SVRGDA_sigma_k_bound}
    Let Assumptions~\ref{as:averaged_cocoercivity}~and~\ref{as:unique_solution} hold. Then for all $k \ge 0$ \algname{L-SVRGDA} satisfies
    \begin{equation}
        \Exp_k\left[\sigma_{k+1}^2\right] \le p\widehat{\ell} \langle F(x^k) - F(x^{*,k}), x^k - x^{*,k} \rangle + (1 - p)\sigma_k^2, \label{eq:L_SVRGDA_sigma_k_bound}
    \end{equation}
    where $\sigma_k^2 \eqdef \frac{1}{n}\sum_{i=1}^n \|F_i(w^k) - F_i(x^{*,k})\|^2$.
\end{lemma}
\begin{proof}
    Using the definitions of $\sigma_{k+1}^2$ and $w^{k+1}$ (see \eqref{eq:L_SVRGDA_w^k}), we derive
    \begin{eqnarray*}
        \Exp_k\left[\sigma_{k+1}^2\right] &=& \frac{1}{n}\sum_{i=1}^n \Exp_k\left[\|F_i(w^{k+1}) - F_i(x^{*,k+1})\|^2\right]\\
        &\overset{\text{As.}~\ref{as:unique_solution}}{=}& \frac{1}{n}\sum_{i=1}^n \Exp_k\left[\|F_i(w^{k+1}) - F_i(x^{*,k})\|^2\right]\\
        &=& \frac{p}{n}\sum_{i=1}^n \|F_i(x^{k}) - F_i(x^{*,k})\|^2 + \frac{1-p}{n}\sum_{i=1}^n \|F_i(w^{k}) - F_i(x^{*,k})\|^2\\
        &\overset{\eqref{eq:averaged_cocoercivity}}{\leq}& p\widehat{\ell} \langle F(x^k) - F(x^{*,k}), x^k - x^{*,k} \rangle + (1 - p)\sigma_k^2. 
    \end{eqnarray*}
\end{proof}

The above two lemmas imply that Assumption~\ref{as:key_assumption} is satisfied with certain parameters.

\begin{proposition}[Proposition~\ref{thm:prox_SVRGDA_convergence}]\label{thm:prox_SVRGDA_convergence_appendix_1}
    Let Assumptions~\ref{as:averaged_cocoercivity}~and~\ref{as:unique_solution} hold. Then, \algname{L-SVRGDA} satisfies Assumption~\ref{as:key_assumption} with
    \begin{gather*}
        A = \widehat{\ell},\quad B = 2,\quad \sigma_k^2 = \frac{1}{n}\sum_{i=1}^n\|F_i(w^k) - F_i(x^{*})\|^2,\quad C = \frac{p\widehat{\ell} }{2},\quad \rho = p,\quad D_1 = D_2 = 0.
    \end{gather*}
\end{proposition}

\subsubsection{Analysis of \algname{L-SVRGDA} in the Quasi-Strongly Monotone Case}

Plugging the parameters from the above proposition in Theorem~\ref{thm:main_result} and Corollary~\ref{cor:main_result} with $M = \frac{4}{p}$ we get the following results.

\begin{theorem}\label{thm:prox_SVRGDA_convergence_appendix}
    Let $F$ be $\mu$-quasi strongly monotone, Assumptions~\ref{as:averaged_cocoercivity}, \ref{as:unique_solution} hold,  and $0 < \gamma \leq \nicefrac{1}{6\widehat{\ell}}$. Then for all $k \ge 0$ the iterates produced by \algname{L-SVRGDA} satisfy
    \begin{equation}
        \Exp\left[\|x^k - x^{*}\|^2\right] \leq \left(1 - \min\left\{\gamma\mu, \nicefrac{p}{2}\right\}\right)^k V_0, \label{eq:prox_SVRGDA_convergence_appendix}
    \end{equation}
    where $V_0 = \|x^0 - x^{*}\|^2 + \nicefrac{4\gamma^2\sigma_0^2}{p}$.
\end{theorem}

\begin{corollary}\label{cor:prox_SVRGDA_convergence_appendix}
    Let the assumptions of Theorem~\ref{thm:prox_SVRGDA_convergence_appendix} hold. Then, for $p = n$, $\gamma = \nicefrac{1}{6\widehat{\ell}}$ and any $K \ge 0$ we have
	\begin{equation}
        \Exp[\|x^k - x^{*}\|^2] \leq V_0 \exp\left(-\min\left\{\frac{\mu}{6\widehat{\ell}}, \frac{1}{2n}\right\}K\right). \notag
    \end{equation}
\end{corollary}

\subsubsection{Analysis of \algname{L-SVRGDA} in the Monotone Case}

Next, using Theorem~\ref{thm:main_result_monotone}, we establish the convergence of \algname{L-SVRGDA} in the monotone case.
\begin{theorem}\label{thm:prox_SVRGDA_convergence_monotone}
    Let $F$ be monotone, $\ell$-star-cocoercive and Assumptions~\ref{as:key_assumption}, \ref{as:boundness}, \ref{as:averaged_cocoercivity}, \ref{as:unique_solution} hold. Assume that $\gamma \leq \nicefrac{1}{6\widehat{\ell}}$. Then for $\text{Gap}_{\cC} (z)$ from \eqref{eq:gap} and for all $K\ge 0$ the iterates of \algname{L-SVRGDA} satisfy
    \begin{eqnarray}
    \Exp\left[\text{Gap}_{\cC} \left(\frac{1}{K}\sum\limits_{k=1}^{K}  x^{k}\right)\right] &\leq& \frac{3\max_{u \in \mathcal{C}}\|x^{0} - u\|^2}{2\gamma K}  + \frac{8\gamma \ell^2 \Omega_{\mathcal{C}}^2}{K} + \frac{\left(12\widehat{\ell} + \ell\right)\|x^0-x^{*,0}\|^2}{K}\notag\\
    &&\quad + \left(4 +  \left( 12\widehat{\ell} + \ell\right) \gamma \right) \frac{2\gamma \sigma_0^2}{p K}  +9\gamma\max_{x^*\in X^*} \|F(x^*)\|^2.\notag
\end{eqnarray}
\end{theorem}

Applying Corollary~\ref{cor:main_result_monotone}, we get the rate of convergence to the exact solution.

\begin{corollary}\label{cor:prox_SVRGDA_convergence_monotone_appendix}
    Let the assumptions of Theorem~\ref{thm:prox_SVRGDA_convergence_monotone} hold and $p = \nicefrac{1}{n}$. Then $\forall K > 0$ one can choose $\gamma$ as
    \begin{equation}
        \gamma = \min\left\{\frac{1}{12\widehat{\ell} + \ell}, \frac{1}{\sqrt{2n\widehat{\ell}\ell}}, \frac{\Omega_{0,\cC}}{G_*\sqrt{K}}\right\}. \label{eq:stepsize_choice_cor_monotone_prox_SVRGDA}
    \end{equation}
    This choice of $\gamma$ implies
    \begin{align}
        \Exp\left[\text{Gap}_{\cC} \left(\frac{1}{K}\sum\limits_{k=1}^{K}  x^{k}\right)\right] = \cO\left(\frac{(\widehat{\ell} + \ell)(\Omega_{0,\cC}^2 + \Omega_0^2) + \sqrt{n\widehat{\ell}\ell}\Omega_{0,\cC}^2+ \ell \Omega_{\cC}^2}{K} + \frac{\Omega_{0,\cC}G_*}{\sqrt{K}}\right).\notag
    \end{align}
\end{corollary}
\begin{proof}
    First of all, \eqref{eq:averaged_cocoercivity}, \eqref{eq:cocoercivity}, and Cauchy-Schwarz inequality imply
    \begin{eqnarray*}
        \sigma_0^2 &=& \frac{1}{n}\sum_{i=1}^n\|F_i(x^0) - F_i(x^{*})\|^2\\
        &\overset{\eqref{eq:averaged_cocoercivity}}{\leq}& \widehat{\ell} \langle F(x^0) - F(x^*), x^0 - x^* \rangle\\
        &\leq& \widehat{\ell} \|F(x^0) - F(x^*)\|\cdot \|x^0 - x^*\|\\
        &\leq& \widehat{\ell}\ell \|x^0 - x^*\|^2 \leq \widehat{\ell}\ell\max_{u \in \mathcal{C}}\|x^{0} - u\|^2 \leq \widehat{\ell}\ell\Omega_{0,\cC}^2.
    \end{eqnarray*}
    Next, applying Corollary~\ref{cor:main_result_monotone} with $\widehat{\sigma}_0 \eqdef \sqrt{\widehat{\ell}\ell}\Omega_{0,\cC}$, we get the result.
\end{proof}

\subsubsection{Analysis of \algname{L-SVRGDA} in the Cocoercive Case}

Next, using Theorem~\ref{thm:main_result_monotone_coco}, we establish the convergence of \algname{L-SVRGDA} in the cocoercive case.
\begin{theorem}\label{thm:prox_SVRGDA_convergence_monotone_coco}
    Let $F$ be $\ell$-cocoercive and Assumptions~\ref{as:key_assumption}, \ref{as:boundness}, \ref{as:averaged_cocoercivity}, \ref{as:unique_solution} hold. Assume that $\gamma \leq \nicefrac{1}{6\widehat{\ell}}$. Then for $\text{Gap}_{\cC} (z)$ from \eqref{eq:gap} and for all $K\ge 0$ the iterates of \algname{L-SVRGDA} satisfy
    \begin{eqnarray}
    \Exp\left[\text{Gap}_{\cC} \left(\frac{1}{K}\sum\limits_{k=1}^{K}  x^{k}\right)\right] &\leq& \frac{3\max_{u \in \mathcal{C}}\|x^{0} - u\|^2}{2\gamma K}  + \frac{\left(18\widehat{\ell} + 3\ell\right)\|x^0-x^{*,0}\|^2}{K}\notag\\
    &&\quad + \left(6 +  \left( 18\widehat{\ell} + 3\ell\right) \gamma \right) \frac{2\gamma \sigma_0^2}{p K}.\notag
\end{eqnarray}
\end{theorem}

Applying Corollary~\ref{cor:main_result_monotone_coco}, we get the rate of convergence to the exact solution.

\begin{corollary}\label{cor:prox_SVRGDA_convergence_monotone_coco_appendix}
    Let the assumptions of Theorem~\ref{thm:prox_SVRGDA_convergence_monotone_coco} hold and $p = \nicefrac{1}{n}$. Then $\forall K > 0$ one can choose $\gamma$ as
    \begin{equation}
        \gamma = \min\left\{\frac{1}{18\widehat{\ell} + 3\ell}, \frac{1}{\sqrt{2n\widehat{\ell}\ell}}\right\}. \label{eq:stepsize_choice_cor_monotone_prox_SVRGDA_coco}
    \end{equation}
    This choice of $\gamma$ implies
    \begin{align}
        \Exp\left[\text{Gap}_{\cC} \left(\frac{1}{K}\sum\limits_{k=1}^{K}  x^{k}\right)\right] = \cO\left(\frac{(\widehat{\ell} + \ell)(\Omega_{0,\cC}^2 + \Omega_0^2) + \sqrt{n\widehat{\ell}\ell}\Omega_{0,\cC}^2}{K}\right).\notag
    \end{align}
\end{corollary}

\subsection{\algname{SAGA-SGDA}}\label{sec:SAGA}
In this section, we show that \algname{SAGA-SGDA} \citep{palaniappan2016stochastic} fits our theoretical framework and derive new results for this method under averaged star-cocoercivity.

\begin{algorithm}[h!]
   \caption{\algname{SAGA-SGDA} \citep{palaniappan2016stochastic}}
   \label{alg:prox_SAGA_SGDA}
\begin{algorithmic}[1]
   \State {\bfseries Input:} starting point $x^0 \in \R^d$, stepsize $\gamma > 0$, number of steps $K$
   \State Set $w_i^0 = x^0$ and compute $F_i(w_i^0)$ for all $i\in [n]$
   \For{$k=0$ {\bfseries to} $K-1$}
   \State Draw a fresh sample $j_k$ from the uniform distribution on $[n]$ and compute $g^k = F_{j_k}(x^k) - F_{j_k}(w_{j_k}^k) + \tfrac{1}{n}\sum_{i=1}^n F_i(w_i^k)$
   \State Set $w_{j_k}^{k+1} = x^k$ and $w_i^{k+1} = w_i^k$ for $i \neq j_k$
   \State $x^{k+1} = \prox_{\gamma R}(x^k - \gamma g^k)$
   \EndFor
\end{algorithmic}
\end{algorithm}

\subsubsection{\algname{SAGA-SGDA} Fits Assumption~\ref{as:key_assumption}}

\begin{lemma}\label{lem:SAGA_second_moment_bound}
    Let Assumption~\ref{as:averaged_cocoercivity} hold. Then for all $k \ge 0$ \algname{SAGA-SGDA} satisfies
    \begin{equation}
        \Exp_k\left[\|g^k - F(x^{*,k})\|^2\right] \le 2\widehat{\ell} \langle F(x^k) - F(x^{*,k}), x^k - x^{*,k} \rangle + 2\sigma_k^2, \label{eq:SAGA_second_moment_bound}
    \end{equation}
    where $\sigma_k^2 \eqdef \frac{1}{n}\sum_{i=1}^n \|F_i(w_i^k) - F_i(x^{*,k})\|^2$.
\end{lemma}
\begin{proof}
    For brevity, we introduce a new notation: $S^k = \tfrac{1}{n}\sum_{i=1}^n F_i(w_i^k)$. Since $g^k = F_{j_k}(x^k) - F_{j_k}(w_{j_k}^k) + S^k$, we have
    \begin{eqnarray*}
        \Exp_k\left[\|g^k - F(x^{*,k})\|^2\right] &=& \Exp_k\left[\|F_{j_k}(x^k) - F_{j_k}(w_{j_k}^k) + S^k - F(x^{*,k})\|^2\right]\\
        &=& \frac{1}{n}\sum\limits_{i=1}^n \|F_{i}(x^k) - F_{i}(w_{i}^k) + S^k - F(x^{*,k})\|^2\\
        &\leq& \frac{2}{n}\sum\limits_{i=1}^n \|F_{i}(x^k) - F_i(x^{*,k})\|^2\\
        &&\quad + \frac{2}{n}\sum\limits_{i=1}^n \|F_{i}(w_i^k) - F_i(x^{*,k}) - (S^k - F(x^{*,k}))\|^2\\
        &\leq& \frac{2}{n}\sum\limits_{i=1}^n \|F_{i}(x^k) - F_i(x^{*,k})\|^2 + \frac{2}{n}\sum\limits_{i=1}^n \|F_{i}(w_i^k) - F_i(x^{*,k})\|^2\\
        &\overset{\eqref{eq:averaged_cocoercivity}}{\leq}& 2\widehat{\ell} \langle F(x^k) - F(x^{*,k}), x^k - x^{*,k} \rangle + 2\sigma_k^2.
    \end{eqnarray*}
\end{proof}

\begin{lemma}\label{lem:SAGA_sigma_k_bound}
    Let Assumptions~\ref{as:averaged_cocoercivity}~and~\ref{as:unique_solution} hold. Then for all $k \ge 0$ \algname{SAGA-SGDA} satisfies
    \begin{equation}
        \Exp_k\left[\sigma_{k+1}^2\right] \le \frac{\widehat{\ell}}{n} \langle F(x^k) - F(x^{*,k}), x^k - x^{*,k} \rangle + (1 - \nicefrac{1}{n})\sigma_k^2, \label{eq:SAGA_sigma_k_bound}
    \end{equation}
    where $\sigma_k^2 \eqdef \frac{1}{n}\sum_{i=1}^n \|F_i(w_i^k) - F_i(x^{*,k})\|^2$.
\end{lemma}
\begin{proof}
    Using the definitions of $\sigma_{k+1}^2$ and $w_i^{k+1}$, we derive
    \begin{eqnarray*}
        \Exp_k\left[\sigma_{k+1}^2\right] &=& \frac{1}{n}\sum_{i=1}^n \Exp_k\left[\|F_i(w_i^{k+1}) - F_i(x^{*,k+1})\|^2\right]\\
        &\overset{\text{As.}~\ref{as:unique_solution}}{=}& \frac{1}{n}\sum_{i=1}^n \Exp_k\left[\|F_i(w_i^{k+1}) - F_i(x^{*,k})\|^2\right]\\
        &=& \frac{1}{n^2}\sum_{i=1}^n \|F_i(x^{k}) - F_i(x^{*,k})\|^2 + \frac{1-\nicefrac{1}{n}}{n}\sum_{i=1}^n \|F_i(w_i^{k}) - F_i(x^{*,k})\|^2\\
        &\overset{\eqref{eq:averaged_cocoercivity}}{\leq}& \frac{\widehat{\ell}}{n} \langle F(x^k) - F(x^{*,k}), x^k - x^{*,k} \rangle + (1 - \nicefrac{1}{n})\sigma_k^2. 
    \end{eqnarray*}
\end{proof}

The above two lemmas imply that Assumption~\ref{as:key_assumption} is satisfied with certain parameters.
\begin{proposition}\label{thm:prox_SAGA_convergence_appendix_1}
    Let Assumptions~\ref{as:averaged_cocoercivity}~and~\ref{as:unique_solution} hold. Then, \algname{SAGA-SGDA} satisfies Assumption~\ref{as:key_assumption} with
    \begin{gather*}
        A = \widehat{\ell},\quad B = 2,\quad \sigma_k^2 = \frac{1}{n}\sum_{i=1}^n\|F_i(w_i^k) - F_i(x^{*})\|^2,\quad C = \frac{\widehat{\ell} }{2n},\quad \rho = \frac{1}{n},\quad D_1 = D_2 = 0.
    \end{gather*}
\end{proposition}

\subsubsection{Analysis of \algname{SAGA-SGDA} in the Quasi-Strongly Monotone Case}

Applying Theorem~\ref{thm:main_result} and Corollary~\ref{cor:main_result} with $M = 4n$, we get the following results.

\begin{theorem}\label{thm:prox_SAGA_convergence_appendix}
    Let $F$ be $\mu$-quasi strongly monotone, Assumptions~\ref{as:averaged_cocoercivity},~\ref{as:unique_solution} hold,  and $0 < \gamma \leq \nicefrac{1}{6\widehat{\ell}}$. Then for all $k \ge 0$ the iterates produced by \algname{SAGA-SGDA} satisfy
    \begin{equation}
        \Exp\left[\|x^k - x^{*}\|^2\right] \leq \left(1 - \min\left\{\gamma\mu, \nicefrac{1}{2n}\right\}\right)^k V_0, \label{eq:prox_SAGA_convergence_appendix}
    \end{equation}
    where $V_0 = \|x^0 - x^{*}\|^2 + 4n\gamma^2\sigma_0^2$.
\end{theorem}

\begin{corollary}\label{cor:prox_SAGA_convergence_appendix}
    Let the assumptions of Theorem~\ref{thm:prox_SAGA_convergence_appendix} hold. Then, for $\gamma = \nicefrac{1}{6\widehat{\ell}}$ and any $K \ge 0$ we have
	\begin{equation}
        \Exp[\|x^K - x^{*}\|^2] \leq V_0 \exp\left(-\min\left\{\frac{\mu}{6\widehat{\ell}}, \frac{1}{2n}\right\}K\right). \notag
    \end{equation}
\end{corollary}

\subsubsection{Analysis of \algname{SAGA-SGDA} in the Monotone Case}
Next, using Theorem~\ref{thm:main_result_monotone}, we establish the convergence of \algname{SAGA-SGDA} in the monotone case.

\begin{theorem}\label{thm:prox_SAGA_convergence_monotone}
    Let $F$ be monotone, $\ell$-star-cocoercive and Assumptions~\ref{as:key_assumption}, \ref{as:boundness}, \ref{as:averaged_cocoercivity}, \ref{as:unique_solution} hold. Assume that $\gamma \leq \nicefrac{1}{6\widehat{\ell}}$. Then for $\text{Gap}_{\cC} (z)$ from \eqref{eq:gap} and for all $K\ge 0$ the iterates produced by \algname{SAGA-SGDA} satisfy
    \begin{eqnarray}
    \Exp\left[\text{Gap}_{\cC} \left(\frac{1}{K}\sum\limits_{k=1}^{K}  x^{k}\right)\right] &\leq& \frac{3\max_{u \in \mathcal{C}}\|x^{0} - u\|^2}{2\gamma K}  + \frac{8\gamma \ell^2 \Omega_{\mathcal{C}}^2}{K} + \frac{\left(12\widehat{\ell} + \ell\right)\|x^0-x^{*,0}\|^2}{K} \notag\\
    &&\quad + \left(4 +  \left( 12\widehat{\ell} + \ell\right) \gamma \right) \frac{2\gamma \sigma_0^2}{p K}  +9\gamma\max_{x^*\in X^*} \|F(x^*)\|^2.\notag
\end{eqnarray}
\end{theorem}

Applying Corollary~\ref{cor:main_result_monotone}, we get the rate of convergence to the exact solution.

\begin{corollary}\label{cor:prox_SAGA_convergence_monotone_appendix}
    Let the assumptions of Theorem~\ref{thm:prox_SAGA_convergence_monotone} hold. Then $\forall K > 0$ one can choose $\gamma$ as
    \begin{equation}
        \gamma = \min\left\{\frac{1}{12\widehat{\ell} + \ell}, \frac{1}{\sqrt{2n\widehat{\ell}\ell}}, \frac{\Omega_{0,\cC}}{G_*\sqrt{K}}\right\}, \label{eq:stepsize_choice_cor_monotone_prox_SAGA}
    \end{equation}
    This choice of $\gamma$ implies
    \begin{align}
        \Exp\left[\text{Gap}_{\cC} \left(\frac{1}{K}\sum\limits_{k=1}^{K}  x^{k}\right)\right] = \cO\left(\frac{(\widehat{\ell} + \ell)(\Omega_{0,\cC}^2 + \Omega_0^2) + \sqrt{n\widehat{\ell}\ell}\Omega_{0,\cC}^2+ \ell \Omega_{\cC}^2}{K} + \frac{\Omega_{0,\cC}G_*}{\sqrt{K}}\right).\notag
    \end{align}
\end{corollary}
\begin{proof}
    Since $\sigma_0$ for \algname{SAGA-SGDA} and \algname{L-SVRGDA} are the same, the proof of this corollary is identical to the one for Corollary~\ref{cor:prox_SVRGDA_convergence_monotone_appendix}.
\end{proof}

\subsubsection{Analysis of \algname{SAGA-SGDA} in the Cocoercive Case}
Next, using Theorem~\ref{thm:main_result_monotone_coco}, we establish the convergence of \algname{SAGA-SGDA} in the cocoercive case.

\begin{theorem}\label{thm:prox_SAGA_convergence_monotone_coco}
    Let $F$ be $\ell$-cocoercive and Assumptions~\ref{as:key_assumption}, \ref{as:boundness}, \ref{as:averaged_cocoercivity}, \ref{as:unique_solution} hold. Assume that $\gamma \leq \nicefrac{1}{6\widehat{\ell}}$. Then for $\text{Gap}_{\cC} (z)$ from \eqref{eq:gap} and for all $K\ge 0$ the iterates produced by \algname{SAGA-SGDA} satisfy
    \begin{eqnarray}
    \Exp\left[\text{Gap}_{\cC} \left(\frac{1}{K}\sum\limits_{k=1}^{K}  x^{k}\right)\right] &\leq& \frac{3\max_{u \in \mathcal{C}}\|x^{0} - u\|^2}{2\gamma K} + \frac{\left(18\widehat{\ell} + 3\ell\right)\|x^0-x^{*,0}\|^2}{K} \notag\\
    &&\quad + \left(6 +  \left( 18\widehat{\ell} + 3\ell\right) \gamma \right) \frac{2\gamma \sigma_0^2}{p K}.\notag
\end{eqnarray}
\end{theorem}

Applying Corollary~\ref{cor:main_result_monotone_coco}, we get the rate of convergence to the exact solution.

\begin{corollary}\label{cor:prox_SAGA_convergence_monotone_coco_appendix}
    Let the assumptions of Theorem~\ref{thm:prox_SAGA_convergence_monotone_coco} hold. Then $\forall K > 0$ one can choose $\gamma$ as
    \begin{equation}
        \gamma = \min\left\{\frac{1}{18\widehat{\ell} + 3\ell}, \frac{1}{\sqrt{2n\widehat{\ell}\ell}}\right\}, \label{eq:stepsize_choice_cor_monotone_coco_prox_SAGA}
    \end{equation}
    This choice of $\gamma$ implies
    \begin{align}
        \Exp\left[\text{Gap}_{\cC} \left(\frac{1}{K}\sum\limits_{k=1}^{K}  x^{k}\right)\right] = \cO\left(\frac{(\widehat{\ell} + \ell)(\Omega_{0,\cC}^2 + \Omega_0^2) + \sqrt{n\widehat{\ell}\ell}\Omega_{0,\cC}^2}{K}\right).\notag
    \end{align}
\end{corollary}

\subsection{Discussion of the Results in the Monotone and Cocoercive Cases}
Among the papers mentioned in the related work on variance-reduced methods (see Section~\ref{AppendixRelatedWork}), only \citet{alacaoglu2021stochastic, carmon2019variance, alacaoglu2021forward, tominin2021accelerated, luo2021near} consider monotone (convex-concave) and Lipschitz (smooth) VIPs (min-max problems) without assuming strong monotonicity (strong-convexity-strong-concavity) of the problem. In this case, \citet{alacaoglu2021stochastic} derive $\cO\left(n + \tfrac{\sqrt{n}L}{K}\right)$ convergence rate (neglecting the dependence on the quantities like $\Omega_{0,\cC}^2 = \max_{u\in \cC}\|x^0 - u\|^2$), which is optimal for the considered setting \citep{han2021lower}. Under additional assumptions a similar rate is derived in \citet{carmon2019variance}. \citet{tominin2021accelerated, luo2021near} also achieve this rate but using Catalyst. Finally, \citet{alacaoglu2021forward} derive $\cO\left(n + \tfrac{nL}{K}\right)$, which is worse than the one from \citet{alacaoglu2021stochastic}. Our results for monotone and star-cocoercive regularized VIPs give $\cO\left(\tfrac{\sqrt{n\ell\widehat{\ell}} + \widehat{\ell}}{K} + \tfrac{G_*}{\sqrt{K}}\right)$ rate, which is typically worse than $\cO\left(n + \tfrac{\sqrt{n}L}{K}\right)$ rate from \citet{alacaoglu2021stochastic} due to the relation between cocoercivity constants and Lipschitz constants (even when $R(x) \equiv 0$, i.e., $G_* = 0$). However, in general, it is possible that star-cocoercivity holds, while Lipschitzness does not \citep{loizou2021stochastic}. As for cocoercive case, we obtain $\cO\left(\tfrac{\sqrt{n\ell\widehat{\ell}} + \widehat{\ell}}{K}\right)$, which matches the rate from \citet{alacaoglu2021stochastic} up to the difference between cocoercivity and Lipschitz constants. Moreover, we emphasize here that \citet{alacaoglu2021stochastic} and other works do not consider \algname{SGDA} as the basis for their methods. To the best of our knowledge, our results are the first ones for variance-reduced \algname{SGDA}-type methods derived in the monotone case without assuming (quasi-)strong monotonicity.

\newpage

\section{DISTRIBUTED \algname{SGDA} WITH COMPRESSION: MISSING PROOFS AND DETAILS}

In this section, we provide missing proofs and details for Section~\ref{sec:distr_sgda}.

\subsection{\algname{QSGDA}}\label{sec:qsgda}

In this section (and in the one about \algname{DIANA-SGDA}), we assume that each $F_i$ has an expectation form: $F_i(x) = \Exp_{\xi_i \sim \cD_i}[F_{\xi_i}(x)]$. 

\begin{algorithm}[h!]
   \caption{\algname{QSGDA}: Quantized Stochastic Gradient Descent-Ascent}
   \label{alg:prox_QSGDA}
\begin{algorithmic}[1]
   \State {\bfseries Input:} starting point $x^0 \in \R^d$, stepsize $\gamma > 0$, number of steps $K$
   \For{$k=0$ {\bfseries to} $K-1$}
   \State Broadcast $x^k$ to all workers
      \For{$i=1,\ldots,n$ in parallel}
      \State Compute $g^k_i$ and send $\cQ(g^k_i)$ to the server
      \EndFor
   \State $g^k = \frac{1}{n} \sum_{i=1}^n \cQ(g^k_i)$
   \State $x^{k+1} = \prox_{\gamma R}\left(x^k - \gamma  g^k \right)$
   \EndFor
\end{algorithmic}
\end{algorithm}

\subsubsection{Proof of Proposition~\ref{thm:prox_QSGDA_convergence}}

\begin{proposition}[Proposition~\ref{thm:prox_QSGDA_convergence}]\label{lem:QSGDA_second_moment_bound}
     Let $F$  be $\ell$-star-cocoercive and Assumptions~ \ref{as:averaged_cocoercivity},~\ref{as:bounded_variance} hold. Then, \algname{QSGDA} with quantization \eqref{eq:quant} satisfies Assumption~\ref{as:key_assumption} with
    \begin{gather*}
        A = \left(\frac{3 \ell}{2} + \frac{9\omega \widehat\ell}{2n} \right),\quad D_1 = \frac{3(1 + 3\omega)\sigma^2 + 9\omega\zeta_*^2}{n},\quad \sigma_k^2 =0 , \quad B = 0,\\
        C = 0,\quad \rho =1,\quad D_2 = 0,
    \end{gather*}
    where $\sigma^2 = \tfrac{1}{n}\sum_{i=1}^n\sigma_i^2$ and $\zeta_*^2 = \frac{1}{n} \max_{x* \in X^*}\left[\sum_{i=1}^n \left\| F_i(x^{*})\right\|^2\right]$.
\end{proposition}
\begin{proof}
    Since $g^k = \frac{1}{n} \sum\limits_{i=1}^n \cQ\left(g^k_i\right)$, $\cQ\left(g^k_1\right),\ldots, \cQ\left(g^k_n\right)$ are independent for fixed $g_1^k,\ldots, g_n^k$, and $g_1^k,\ldots, g_n^k$ are independent for fixed $x^k$, we have
    \begin{eqnarray*}
        \Exp_k\left[\|g^k - F(x^{*,k})\|^2\right] &=& \Exp_k\left[\left\|\frac{1}{n} \sum\limits_{i=1}^n \cQ\left(g^k_i\right) - F(x^{*,k})\right\|^2\right]\\
        &=& \Exp_k\left[\left\|\frac{1}{n} \sum\limits_{i=1}^n \left[\cQ\left(g^k_i\right) - g^k_i + g^k_i - F_i(x^k)\right] + F(x^k) - F(x^{*,k})\right\|^2\right] \\
        &\leq& 3\Exp_k\left[\left\|\frac{1}{n} \sum\limits_{i=1}^n [\cQ\left(g^k_i\right) - g^k_i] \right\|^2\right] + 3\Exp_k\left[\left\|\frac{1}{n} \sum\limits_{i=1}^n [ g^k_i - F_i(x^k)] \right\|^2\right] \\
        &&\quad+ 3\left\| F(x^k) - F(x^{*,k})\right\|^2
        \\
        &=& \frac{3}{n^2} \sum\limits_{i=1}^n \Exp_k\left[\left\|\cQ\left(g^k_i\right) - g^k_i \right\|^2\right] + \frac{3}{n^2} \sum\limits_{i=1}^n \Exp_k\left[\left\|g^k_i - F_i(x^k) \right\|^2\right]\\
        &&\quad+ 3\left\| F(x^k) - F(x^{*,k})\right\|^2.
    \end{eqnarray*}
Next, we use Assumption \ref{as:bounded_variance}, $\sigma^2 = \tfrac{1}{n}\sum_{i=1}^n\sigma_i^2$, and the definition of quantization \eqref{eq:quant} and get
\begin{eqnarray*}
        \Exp_k\left[\|g^k - F(x^{*,k})\|^2\right] 
        &\leq&  \frac{3\omega}{n^2} \sum\limits_{i=1}^n \Exp_k\left[\left\|g^k_i \right\|^2\right]  + \frac{3\sigma^2}{n} + 3\left\| F(x^k) - F(x^{*,k})\right\|^2 \\ 
        &\leq& \frac{3\omega}{n^2} \sum\limits_{i=1}^n \Exp_k\left[\left\|g^k_i - F_i (x^k) + F_i (x^k)  - F_i(x^{*,k}) + F_i(x^{*,k})\right\|^2\right]  \\
        &&\quad+ \frac{3\sigma^2}{n} + 3\left\| F(x^k) - F(x^{*,k})\right\|^2 \\
        &\leq& \frac{9\omega}{n^2} \sum\limits_{i=1}^n \Exp_k\left[\left\|g^k_i - F_i (x^k)\right\|^2\right] + \frac{9\omega}{n^2} \sum\limits_{i=1}^n \Exp_k\left[\left\|F_i (x^k)  - F_i(x^{*,k})\right\|^2\right]  \\
        &&\quad + \frac{9\omega}{n^2} \sum\limits_{i=1}^n \Exp_k\left[\left\| F_i(x^{*,k})\right\|^2\right]  + \frac{3\sigma^2}{n} + 3\left\| F(x^k) - F(x^{*,k})\right\|^2 \\
        &\overset{\eqref{eq:variance}}{\leq}& \frac{9\omega}{n^2} \sum\limits_{i=1}^n \Exp_k\left[\left\|F_i (x^k)  - F_i(x^{*,k})\right\|^2\right] + 3\left\| F(x^k) - F(x^{*,k})\right\|^2 \\
        &&\quad + \frac{9\omega}{n^2} \sum\limits_{i=1}^n \Exp_k\left[\left\| F_i(x^{*,k})\right\|^2\right]  + \frac{3(1 + 3\omega)\sigma^2}{n}.
\end{eqnarray*}
Star-cocoercivity of $F$ and Assumption \ref{as:averaged_cocoercivity} give
\begin{eqnarray*}
        \Exp_k\left[\|g^k - F(x^{*,k})\|^2\right] 
        &\leq&  \left(3 \ell + \frac{9\omega}{n} \widehat\ell\right) \langle F(x^k) - F(x^{*,k}), x^k - x^{*,k} \rangle \\
        &&\quad + \frac{9\omega}{n^2} \sum\limits_{i=1}^n \Exp_k\left[\left\| F_i(x^{*,k})\right\|^2\right]  + \frac{3(1 + 3\omega)\sigma^2}{n}  \\
        &\leq&  \left(3 \ell + \frac{9\omega}{n} \widehat\ell\right)\langle F(x^k) - F(x^{*,k}), x^k - x^{*,k} \rangle \\
        &&\quad + \frac{9\omega}{n^2} \max_{x* \in X^*}\left[\sum\limits_{i=1}^n \left\| F_i(x^{*})\right\|^2\right]  + \frac{3(1 + 3\omega)\sigma^2}{n} .
\end{eqnarray*}
\end{proof}

\subsubsection{Analysis of \algname{QSGDA} in the Quasi-Strongly Monotone Case}

Applying Theorem~\ref{thm:main_result} and Corollary~\ref{cor:main_result}, we get the following results.

\begin{theorem}\label{thm:prox_QSGDA_convergence_appendix}
     Let $F$  be $\mu$-quasi strongly monotone, $\ell$-star-cocoercive, Assumptions~ \ref{as:averaged_cocoercivity},~\ref{as:bounded_variance} hold, and 
    $$0 < \gamma \leq \frac{1}{3 \ell + \frac{9\omega \widehat\ell}{n}}.$$
    Then, for all $k \ge 0$ the iterates produced by \algname{QSGDA} satisfy
    \begin{equation*}
        \Exp\left[\|x^k - x^{*}\|^2\right] \leq \left(1 - \gamma\mu\right)^k \|x^0 - x^{*}\|^2 + \gamma \frac{3(1 + 3\omega)\sigma^2 + 9\omega\zeta_*^2}{n\mu}. \label{eq:prox_QSGDA_convergence_appendix}
    \end{equation*}
\end{theorem}

\begin{corollary}\label{cor:prox_QSGDA_convergence_appendix}
    Let the assumptions of Theorem~\ref{thm:prox_QSGDA_convergence_appendix} hold. Then, for any $K \ge 0$ one can choose $\{\gamma_k\}_{k \ge 0}$ as follows:
	\begin{eqnarray*}
		\text{if } K \le \frac{1}{\mu} \cdot \left(3 \ell + \frac{9\omega \widehat\ell}{n} \right), && \gamma_k = \left(3 \ell + \frac{9\omega \widehat\ell}{n} \right)^{-1},\notag\\
		\text{if } K > \frac{1}{\mu} \cdot \left(3 \ell + \frac{9\omega \widehat\ell}{n} \right) \text{ and } k < k_0, && \gamma_k = \left(3 \ell + \frac{9\omega \widehat\ell}{n} \right)^{-1},\\
		\text{if } K > \frac{1}{\mu} \cdot \left(3 \ell + \frac{9\omega \widehat\ell}{n} \right) \text{ and } k \ge k_0, && \gamma_k = \frac{2}{(6\ell + \nicefrac{18\omega \widehat\ell}{n} + \mu(k - k_0))},\notag
	\end{eqnarray*}
	where $k_0 = \left\lceil \nicefrac{K}{2} \right\rceil$. For this choice of $\gamma_k$ the following inequality holds:
	\begin{eqnarray*}
		\Exp[\|x^{K} - x^{*,K}\|^2] &\le& \frac{32(3 \ell + \nicefrac{9\omega \widehat\ell}{n})}{\mu} \|x^{0} - x^{*,0}\|^2 \exp\left(-\frac{\mu}{(3 \ell + \nicefrac{9\omega \widehat\ell}{n})}K\right)\\
		&&\quad + \frac{36 }{\mu^2 K} \cdot \frac{3(1 + 3\omega)\sigma^2 + 9\omega\zeta_*^2}{n}.\notag
	\end{eqnarray*}
\end{corollary}

\subsubsection{Analysis of \algname{QSGDA} in the Monotone Case}

Next, using Theorem~\ref{thm:main_result_monotone}, we establish the convergence of \algname{QSGDA} in the monotone case.
\begin{theorem}\label{thm:prox_QSGDA_convergence_monotone}
    Let $F$ be monotone, $\ell$-star-cocoercive and Assumptions~\ref{as:key_assumption}, \ref{as:boundness}, \ref{as:averaged_cocoercivity},~\ref{as:bounded_variance} hold. Assume that $\gamma \leq \left(3 \ell + \frac{9\omega \widehat\ell}{n}\right)^{-1}$. Then for $\text{Gap}_{\cC} (z)$ from \eqref{eq:gap} and for all $K\ge 0$ the iterates produced by \algname{QSGDA} satisfy
    \begin{eqnarray*}
    \Exp\left[\text{Gap}_{\cC} \left(\frac{1}{K}\sum\limits_{k=1}^{K}  x^{k}\right)\right] &\leq& \frac{3\left[\max_{u \in \mathcal{C}}\|x^{0} - u\|^2\right]}{2\gamma K}  + \frac{8\gamma \ell^2 \Omega_{\mathcal{C}}^2}{K} + \left( 7 \ell + \frac{18\omega \widehat\ell}{n}\right) \cdot \frac{\|x^0-x^{*,0}\|^2}{K}\\
    &&\quad +\gamma \left(2 + \gamma\left( 7 \ell + \frac{18\omega \widehat\ell}{n}\right)\right)\cdot \frac{3(1 + 3\omega)\sigma^2 + 9\omega\zeta_*^2}{n} \\
    &&\quad +9\gamma \max_{x^*\in X^*}\left[ \|F(x^*)\|^2\right]
\end{eqnarray*}
\end{theorem}

Applying Corollary~\ref{cor:main_result_monotone}, we get the rate of convergence to the exact solution.
\begin{corollary}\label{cor:prox_QSGDA_convergence_monotone}
    Let the assumptions of Theorem~\ref{thm:prox_QSGDA_convergence_monotone} hold. Then $\forall K > 0$ one can choose $\gamma$ as
   \begin{equation*}
        \gamma = \min\left\{\frac{1}{7\ell + \frac{18\omega \widehat\ell}{n}}, \frac{\Omega_{0,\cC}\sqrt{n}}{\sqrt{3K(1 + 3\omega)\sigma^2 + 9K\omega\zeta_*^2}}, \frac{\Omega_{0,\cC}}{G_*\sqrt{K}}\right\}. 
    \end{equation*}
     This choice of $\gamma$ implies
    \begin{align*}
        \Exp\left[\text{Gap}_{\cC} \left(\frac{1}{K}\sum\limits_{k=1}^{K}  x^{k}\right)\right] = \cO\left(\frac{\left( \ell + \nicefrac{\omega \widehat\ell}{n}\right)(\Omega_{0,\cC}^2 + \Omega_0^2) + \ell \Omega_{\cC}^2}{K}  + \frac{\Omega_{0,\cC}(\sigma\sqrt{1 + \omega} + G_*\sqrt{n} + \zeta_*\sqrt{\omega})}{\sqrt{nK}}\right).\notag
    \end{align*}
\end{corollary}

\subsubsection{Analysis of \algname{QSGDA} in the Cocoercive Case}

Next, using Theorem~\ref{thm:main_result_monotone_coco}, we establish the convergence of \algname{QSGDA} in the cocoercive case.
\begin{theorem}\label{thm:prox_QSGDA_convergence_monotone_coco}
    Let $F$ be $\ell$-cocoercive and Assumptions~\ref{as:key_assumption}, \ref{as:boundness}, \ref{as:averaged_cocoercivity},~\ref{as:bounded_variance} hold. Assume that $\gamma \leq \left(3 \ell + \frac{9\omega \widehat\ell}{n}\right)^{-1}$. Then for $\text{Gap}_{\cC} (z)$ from \eqref{eq:gap} and for all $K\ge 0$ the iterates produced by \algname{QSGDA} satisfy
    \begin{eqnarray*}
    \Exp\left[\text{Gap}_{\cC} \left(\frac{1}{K}\sum\limits_{k=1}^{K}  x^{k}\right)\right] &\leq& \frac{3\left[\max_{u \in \mathcal{C}}\|x^{0} - u\|^2\right]}{2\gamma K} + \left( 10 \ell + \frac{27\omega \widehat\ell}{n}\right) \cdot \frac{\|x^0-x^{*,0}\|^2}{K}\\
    &&\quad +\gamma \left(3 + \gamma\left( 10 \ell + \frac{27\omega \widehat\ell}{n}\right)\right)\cdot \frac{3(1 + 3\omega)\sigma^2 + 9\omega\zeta_*^2}{n}.
\end{eqnarray*}
\end{theorem}

Applying Corollary~\ref{cor:main_result_monotone_coco}, we get the rate of convergence to the exact solution.
\begin{corollary}\label{cor:prox_QSGDA_convergence_monotone_coco}
    Let the assumptions of Theorem~\ref{thm:prox_QSGDA_convergence_monotone_coco} hold. Then $\forall K > 0$ one can choose $\gamma$ as
   \begin{equation*}
        \gamma = \min\left\{\frac{1}{10\ell + \frac{27\omega \widehat\ell}{n}}, \frac{\Omega_{0,\cC}\sqrt{n}}{\sqrt{3K(1 + 3\omega)\sigma^2 + 9K\omega\zeta_*^2}}\right\}. 
    \end{equation*}
     This choice of $\gamma$ implies
    \begin{align*}
        \Exp\left[\text{Gap}_{\cC} \left(\frac{1}{K}\sum\limits_{k=1}^{K}  x^{k}\right)\right] = \cO\left(\frac{\left( \ell + \nicefrac{\omega \widehat\ell}{n}\right)(\Omega_{0,\cC}^2 + \Omega_0^2)}{K}  + \frac{\Omega_{0,\cC}(\sigma\sqrt{1 + \omega} + \zeta_*\sqrt{\omega})}{\sqrt{nK}}\right).\notag
    \end{align*}
\end{corollary}

\subsection{\algname{DIANA-SGDA}}\label{sec:diana}

\begin{algorithm}[h!]
   \caption{\algname{DIANA-SGDA}: \algname{DIANA} Stochastic Gradient Descent-Ascent \cite{mishchenko2019distributed,horvath2019stochastic}}
   \label{alg:prox_DIANA_SGDA}
\begin{algorithmic}[1]
    \State {\bfseries Input:} starting points $x^0, h_1^0, \ldots, h_n^0 \in \R^d$, $h^0 = \frac{1}{n}\sum_{i=1}^n h^0_i$ ,  stepsizes $\gamma, \alpha > 0$, number of steps $K$
   \For{$k=0$ {\bfseries to} $K-1$}
   \State Broadcast $x^k$ to all workers
      \For{$i=1,\ldots,n$ in parallel}
      \State Compute $g^k_i$ and $\Delta^k_i = g^k_i - h_i^k$
      \State Send $\cQ(\Delta^k_i)$ to the server
      \State $h^{k+1}_i = h_i^k + \alpha \cQ(\Delta^k_i)$
      \EndFor
   \State $g^k = h^k + \frac{1}{n} \sum\limits_{i=1}^n \cQ(\Delta^k_i) = \frac{1}{n} \sum\limits_{i=1}^n (h_i^k + \cQ(\Delta^k_i))$
   \State $x^{k+1} = \prox_{\gamma R}\left(x^k - \gamma  g^k \right)$
   \State $h^{k+1} = h^k + \alpha \frac{1}{n} \sum\limits_{i=1}^n \cQ(\Delta^k_i) = \frac{1}{n} \sum\limits_{i=1}^n h^k_i$
   \EndFor
\end{algorithmic}
\end{algorithm}

\subsubsection{Proof of Proposition~\ref{thm:prox_DIANA_convergence}}
The following result follows from Lemmas 1 and 2 from \cite{horvath2019stochastic}. It holds in our settings as well, since it does not rely on the exact form of $F_i(x^k)$.

\begin{lemma}[Lemmas 1 and 2 from \cite{horvath2019stochastic}]\label{lem:diana_horv}
    Let Assumptions \ref{as:unique_solution}, \ref{as:bounded_variance} hold. Suppose that $\alpha \leq \nicefrac{1}{(1+\omega)}$. Then, for all $k \ge 0$ \algname{DIANA-SGDA} satisfies
    \begin{eqnarray*}
        \Exp_k \left[ g^k\right] &=& F(x^k), \\
        \Exp_k \left[ \|g^k - F(x^{*}) \|^2\right] &\leq& \left( 1 + \frac{2\omega}{n}\right)\frac{1}{n}\sum\limits_{i=1}^n \|F_i (x^k) - F_i (x^*) \|^2 + \frac{2\omega\sigma_k^2}{n} + \frac{(1 + \omega)\sigma^2}{n}, \\
        \Exp_k \left[ \sigma^2_{k+1}\right] &\leq& (1 - \alpha) \sigma^2_k + \frac{\alpha}{n} \sum\limits_{i=1}^n \|F_i (x^k) - F_i (x^*) \|^2 + \alpha \sigma^2,
    \end{eqnarray*}
    where $\sigma_k^2 = \frac{1}{n}\sum\limits_{i=1}^n \|h^k_i - F_i(x^*) \|^2$ and $\sigma^2 = \tfrac{1}{n}\sum_{i=1}^n\sigma_i^2$.
\end{lemma}

The lemma above implies that Assumption~\ref{as:key_assumption} is satisfied with certain parameters.

\begin{proposition}[Proposition \ref{thm:prox_DIANA_convergence}]\label{lem:DIANA_SGDA_second_moment_bound}
    Let Assumptions~\ref{as:averaged_cocoercivity}, \ref{as:unique_solution},  \ref{as:bounded_variance} hold. Suppose that $\alpha \leq \frac{1}{1+\omega}$. Then, \algname{DIANA-SGDA} with quantization \eqref{eq:quant} satisfies Assumption~\ref{as:key_assumption} with $\sigma_k^2 = \frac{1}{n}\sum_{i=1}^n \|h^k_i - F_i(x^*) \|^2$ and 
    \begin{gather*}
        A = \left(\frac{1}{2} + \frac{\omega}{n}\right)\widehat \ell,\quad B = \frac{2\omega}{n},\quad D_1 = \frac{(1 + \omega)\sigma^2}{n},\quad C = \frac{\alpha \widehat \ell}{2}, \quad \rho = \alpha, \quad D_2 = \alpha \sigma^2.
    \end{gather*}
\end{proposition}
\begin{proof}
To get the result, one needs to apply Assumption~\ref{as:averaged_cocoercivity} to estimate $\tfrac{1}{n}\sum_{i=1}^n \|F_i (x^k) - F_i (x^*) \|^2$ from Lemma \ref{lem:diana_horv}.
\end{proof}

\subsubsection{Analysis of \algname{DIANA-SGDA} in the Quasi-Strongly Monotone Case}

Applying Theorem~\ref{thm:main_result} and Corollary~\ref{cor:main_result} with $M = \frac{4\omega}{\alpha n}$, we get the following results.

\begin{theorem}\label{thm:prox_DIANA_convergence_appendix}
    Let $F$ be $\mu$-quasi strongly monotone, Assumptions~\ref{as:averaged_cocoercivity}, \ref{as:unique_solution},  \ref{as:bounded_variance} hold, $\alpha \leq \nicefrac{1}{(1+\omega)}$, and 
    $$
    0 < \gamma \leq \frac{1}{\left(1 + \frac{6\omega}{n}\right)\widehat \ell}.
    $$ 
    Then, for all $k \ge 0$ the iterates produced by \algname{DIANA-SGDA} satisfy
    \begin{equation*}
        \Exp\left[\|x^k - x^{*}\|^2\right] \leq \left(1 - \min\left\{\gamma\mu, \frac{\alpha }{2}\right\}\right)^k \Exp[V_0] + \frac{\gamma^2 \sigma^2 (1 + 5\omega)}{n \cdot \min\left\{\gamma\mu, \nicefrac{\alpha}{2}\right\}}, \label{eq:prox_DIANA_SGDA_convergence_appendix}
    \end{equation*}
    where $V_0 = \|x^0 - x^*\|^2 + \nicefrac{4\omega\gamma^2 \sigma_0^2}{\alpha n}$.
\end{theorem}

\begin{corollary}\label{cor:prox_DIANA_convergence_appendix}
    Let the assumptions of Theorem~\ref{thm:prox_DIANA_convergence} hold. Then, for any $K \ge 0$ one can choose $\alpha = \nicefrac{1}{(1+\omega)}$ and $\{\gamma_k\}_{k \ge 0}$ as follows:
	\begin{eqnarray*}
		\text{if } K \le \frac{h}{\mu}, && \gamma_k = \frac{1}{h},\notag\\
		\text{if } K > \frac{h}{\mu} \text{ and } k < k_0, && \gamma_k = \frac{1}{h},\\
		\text{if } K > \frac{h}{\mu} \text{ and } k \ge k_0, && \gamma_k = \frac{2}{2h + \mu(k - k_0)},\notag
	\end{eqnarray*}
	where $h = \max\left\{\left(1 + \tfrac{6\omega}{n}\right)\widehat{\ell}, 2\mu(1+ \omega)\right\}$, $k_0 = \left\lceil \nicefrac{K}{2} \right\rceil$. For this choice of $\gamma_k$ the following inequality holds:
	\begin{eqnarray*}
		\Exp[\|x^{K} - x^{*,K}\|^2] &\le& 32\max\left\{\frac{\left(1 + \tfrac{6\omega}{n}\right)\widehat{\ell}}{\mu}, 2(1+ \omega)\right\} V_0 \exp\left(-\min\left\{\frac{\mu}{\widehat\ell(1 + \tfrac{6\omega }{n})}, \frac{1}{1+\omega}\right\}K\right) \\
		&&\quad+ \frac{36 (1 + 5\omega)\sigma^2}{\mu^2 n K}.\notag
	\end{eqnarray*}
	
\end{corollary}

\subsubsection{Analysis of \algname{DIANA-SGDA} in the Monotone Case}
Next, using Theorem~\ref{thm:main_result_monotone}, we establish the convergence of \algname{DIANA-SGDA} in the monotone case.

\begin{theorem}\label{thm:prox_DIANA_convergence_monotone}
   Let $F$  be monotone, $\ell$-star-cocoercive and Assumptions~\ref{as:key_assumption}, \ref{as:boundness}, ~\ref{as:averaged_cocoercivity}, \ref{as:unique_solution},  \ref{as:bounded_variance}  hold. Assume that 
   \begin{equation*}
        0 < \gamma \leq \frac{1}{ \left(1 + \frac{4 \omega}{ n}\right) \widehat \ell}.
    \end{equation*}
   Then for $\text{Gap}_{\cC} (z)$ from \eqref{eq:gap} and for all $K\ge 0$ the iterates produced by \algname{DIANA-SGDA} satisfy
    \begin{eqnarray*}
    \Exp\left[\text{Gap}_{\cC} \left(\frac{1}{K}\sum\limits_{k=1}^{K}  x^{k}\right)\right] 
    &\leq& \frac{3\left[\max_{u \in \mathcal{C}}\|x^{0} - u\|^2\right]}{2\gamma K}  + \frac{8\gamma \ell^2 \Omega_{\mathcal{C}}^2}{K} + \left( 2 \widehat \ell + \frac{12\omega \widehat \ell}{n} + \ell\right) \frac{\|x^0-x^{*,0}\|^2}{K} \notag\\
    &&\quad + \left(4 +  \gamma \left( 2 \widehat \ell + \frac{12\omega \widehat \ell}{n} + \ell\right)  \right) \frac{\gamma B \sigma_0^2}{\rho K} \notag\\
    &&\quad +\gamma \left(\left(2 + \gamma\left( 2 \widehat \ell + \frac{12\omega \widehat \ell}{n} + \ell\right)\right)\left(\frac{(1 + 5\omega) \sigma^2}{n} \right)\right)\\
    &&\quad + 9\gamma\max\limits_{x^* \in X^*}\|F(x^*)\|^2.
\end{eqnarray*}
\end{theorem}

Applying Corollary~\ref{cor:main_result_monotone}, we get the rate of convergence to the exact solution.

\begin{corollary}\label{cor:prox_DIANA_convergence_monotone}
    Let the assumptions of Theorem~\ref{thm:prox_DIANA_convergence_monotone} hold. Then $\forall K > 0$ one can choose $\gamma$ as
    \begin{equation*}
        \gamma = \min\left\{\left( \ell + 2 \widehat \ell  + \frac{12\omega \widehat \ell}{n} \right)^{-1}, \frac{\sqrt{\alpha n}}{\sqrt{2 \omega \widehat{\ell}\ell}}, \frac{\Omega_{0,\cC}}{\sigma\sqrt{K \nicefrac{(1 + 3\omega)}{n}}}, \frac{\Omega_{0,\cC}}{G_*\sqrt{K}}\right\}, 
    \end{equation*}
This choice of $\gamma$ implies that $\Exp\left[\text{Gap}_{\cC} \left(\tfrac{1}{K}\sum_{k=1}^{K}  x^{k}\right)\right]$ equals
\begin{align*}
    \cO\left(\frac{(\ell + \widehat \ell  + \nicefrac{\omega \widehat \ell}{n})(\Omega_{0,\cC}^2 + \Omega_0^2) + \ell \Omega_{\cC}^2}{K} + \frac{\Omega^2_{0,\cC} \sqrt{\widehat{\ell}\ell}\sqrt{\omega}}{\sqrt{\alpha n} K} + \frac{\Omega_{0,\cC}(\sqrt{\nicefrac{(1 + \omega)\sigma^2}{n}} + G_*)}{\sqrt{K}}\right).\notag
    \end{align*}
\end{corollary}
\begin{proof}
The proof follows from the next upper bound $\widehat \sigma_0^2$ for $\sigma^2_0$ with initialization $h^0_i = F_i(x^0)$ 
    \begin{eqnarray*}
        \sigma_0^2 &=& \frac{1}{n}\sum_{i=1}^n \|F_i(x^0) - F_{i}(x^{*})\|^2\\
        &\overset{}{\leq}& \widehat{\ell}\langle F(x^0) - F(x^*), x^0 - x^* \rangle\\
        &\leq& \widehat{\ell} \|F(x^0) - F(x^*)\|\cdot \|x^0 - x^*\|\\
        &\leq& \widehat{\ell}\ell \|x^0 - x^*\|^2 \leq \widehat{\ell}\ell\max_{u \in \mathcal{C}}\|x^{0} - u\|^2 \leq \widehat{\ell}\ell\Omega_{0,\cC}^2.
    \end{eqnarray*}
    Next, applying Corollary~\ref{cor:main_result_monotone} with $\widehat{\sigma}_0 \eqdef \sqrt{\widehat{\ell}\ell}\Omega_{0,\cC}$, we get the result.
\end{proof}

\subsubsection{Analysis of \algname{DIANA-SGDA} in the Cocoercive Case}
Next, using Theorem~\ref{thm:main_result_monotone_coco}, we establish the convergence of \algname{DIANA-SGDA} in the cocoercive case.

\begin{theorem}\label{thm:prox_DIANA_convergence_monotone_coco}
   Let $F$  be $\ell$-cocoercive and Assumptions~\ref{as:key_assumption}, \ref{as:boundness}, ~\ref{as:averaged_cocoercivity}, \ref{as:unique_solution},  \ref{as:bounded_variance}  hold. Assume that 
   \begin{equation*}
        0 < \gamma \leq \frac{1}{ \left(1 + \frac{4 \omega}{ n}\right) \widehat \ell}.
    \end{equation*}
   Then for $\text{Gap}_{\cC} (z)$ from \eqref{eq:gap} and for all $K\ge 0$ the iterates produced by \algname{DIANA-SGDA} satisfy
    \begin{eqnarray*}
    \Exp\left[\text{Gap}_{\cC} \left(\frac{1}{K}\sum\limits_{k=1}^{K}  x^{k}\right)\right] 
    &\leq& \frac{3\left[\max_{u \in \mathcal{C}}\|x^{0} - u\|^2\right]}{2\gamma K} + \left( 3 \widehat \ell + \frac{18\omega \widehat \ell}{n} + 3\ell\right) \frac{\|x^0-x^{*,0}\|^2}{K} \notag\\
    &&\quad + \left(6 +  \gamma \left( 4 \widehat \ell + \frac{18\omega \widehat \ell}{n} + 3\ell\right)  \right) \frac{\gamma B \sigma_0^2}{\rho K} \notag\\
    &&\quad +\gamma \left(\left(3 + \gamma\left( 3 \widehat \ell + \frac{18\omega \widehat \ell}{n} + 3\ell\right)\right)\left(\frac{(1 + 5\omega) \sigma^2}{n} \right)\right).
\end{eqnarray*}
\end{theorem}

Applying Corollary~\ref{cor:main_result_monotone_coco}, we get the rate of convergence to the exact solution.

\begin{corollary}\label{cor:prox_DIANA_convergence_monotone_coco}
    Let the assumptions of Theorem~\ref{thm:prox_DIANA_convergence_monotone_coco} hold. Then $\forall K > 0$ one can choose $\gamma$ as
    \begin{equation*}
        \gamma = \min\left\{\left( 3\ell + 3 \widehat \ell  + \frac{18\omega \widehat \ell}{n} \right)^{-1}, \frac{\sqrt{\alpha n}}{\sqrt{2 \omega \widehat{\ell}\ell}}, \frac{\Omega_{0,\cC}}{\sigma\sqrt{K \nicefrac{(1 + 3\omega)}{n}}}\right\}, 
    \end{equation*}
This choice of $\gamma$ implies that $\Exp\left[\text{Gap}_{\cC} \left(\tfrac{1}{K}\sum_{k=1}^{K}  x^{k}\right)\right]$ equals
\begin{align*}
    \cO\left(\frac{(\ell + \widehat \ell  + \nicefrac{\omega \widehat \ell}{n})(\Omega_{0,\cC}^2 + \Omega_0^2)}{K} + \frac{\Omega^2_{0,\cC} \sqrt{\widehat{\ell}\ell}\sqrt{\omega}}{\sqrt{\alpha n} K} + \frac{\Omega_{0,\cC}\sqrt{\nicefrac{(1 + \omega)\sigma^2}{n}}}{\sqrt{K}}\right).\notag
    \end{align*}
\end{corollary}

\subsection{\algname{VR-DIANA-SGDA}}\label{sec:vr_diana_sgda}
In this section, we assume that each $F_i$ has a finite-sum form: $F_i(x) = \tfrac{1}{m}\sum_{j=1}^m F_{ij}(x)$.

\begin{algorithm}[h!]
   \caption{\algname{VR-DIANA-SGDA}: \algname{VR-DIANA} Stochastic Gradient Descent-Ascent \cite{horvath2019stochastic}}
   \label{alg:prox_VR_DIANA_SGDA}
\begin{algorithmic}[1]
    \State {\bfseries Input:} starting points $x^0, h_1^0, \ldots, h_n^0 \in \R^d$, $h^0 = \frac{1}{n}\sum\limits_{i=1}^n h^0_i$ , probability $p \in (0,1]$ stepsizes $\gamma, \alpha > 0$, number of steps $K$, 
   \For{$k=0$ {\bfseries to} $K-1$}
   \State Broadcast $x^k$ to all workers
      \For{$i=1,\ldots,n$ in parallel}
      \State Draw a fresh sample $j^k_i$ from the uniform distribution on $[m]$ and compute $g_i^k = F_{ij^k_i}(x^k) - F_{ij^k_i}(w_i^k) + F_i(w_i^k)$
   \State $w_i^{k+1} = \begin{cases} x^k, & \text{with probability } p,\\ w_i^k,& \text{with probability } 1-p,\end{cases}$
      \State $\Delta^k_i = g^k_i - h_i^k$
      \State Send $\cQ(\Delta^k_i)$ to the server
      \State $h^{k+1}_i = h_i^k + \alpha \cQ(\Delta^k_i)$
      \EndFor
   \State $g^k = h^k + \frac{1}{n} \sum\limits_{i=1}^n \cQ(\Delta^k_i) = \frac{1}{n} \sum\limits_{i=1}^n (h_i^k + \cQ(\Delta^k_i))$
   \State $x^{k+1} = \prox_{\gamma R}\left(x^k - \gamma  g^k \right)$
   \State $h^{k+1} = h^k + \alpha \frac{1}{n} \sum\limits_{i=1}^n \cQ(\Delta^k_i) = \frac{1}{n} \sum\limits_{i=1}^n h^k_i$
   \EndFor
\end{algorithmic}
\end{algorithm}

\subsubsection{Proof of Proposition~\ref{thm:prox_VR_DIANA_convergence}}

\begin{lemma}[Modification of Lemmas 3 and 7 from \cite{horvath2019stochastic}]\label{lem:vr_diana_horv1}
    Let $F$  be $\ell$-star-cocoercive and Assumptions~\ref{as:averaged_cocoercivity}, \ref{as:unique_solution}, \ref{as:sum_averaged_cocoercivity} hold. Then for all $k \ge 0$ \algname{VR-DIANA-SGDA} satisfies
    \begin{eqnarray*}
        \Exp_k \left[ g^k\right] &=& F(x^k), \\
        \Exp_k \left[ \|g^k - F(x^{*}) \|\right] &\leq& \left(
			\ell + \frac{2 \widetilde \ell}{n} + \frac{2 \omega (\widehat \ell + \widetilde \ell)}{n}
		\right)  \langle F(x^k) - F(x^{*}), x^k - x^{*} \rangle + \frac{2(\omega+1)}{n} \sigma_k^2,
    \end{eqnarray*}
    where $\sigma^2_k = \frac{H^k}{n} + \frac{D^k}{nm}$ with $H^k = \sum\limits_{i=1}^n \left\| h^k_i - F_i (x^*) \right\|^2$ and $D^k = \sum\limits_{i=1}^n \sum\limits_{j=1}^m \left\| F_{ij} (w^k_{i}) - F_{ij} (x^*) \right\|^2$.
\end{lemma}
\begin{proof} First of all, we derive unbiasedness:
\begin{align*}
		\E{g^k}
		=
		\frac{1}{n}
		\sum\limits_{i=1}^{n}
			\E{\cQ(g_i^k - h_i^k) + h_i^k}
		=
		\frac{1}{n}
		\sum\limits_{i=1}^{n}
			\E{g_i^k - h_i^k + h_i^k}
		=
		\frac{1}{n}
		\sum\limits_{i=1}^{n}
			F_i(x^k)
		=
		F(x^k).
	\end{align*}
By definition of the variance we get
	\begin{align*}
		\EE{\cQ}{\norm{g^k - F(x^*)}^2} &{=}
		\underbrace{
			\norm{\EE{\cQ}{g^k} - F(x^*)}^2
		}_{T_1}
		+
		\underbrace{
			\EE{Q}{\norm{g^k - \EE{Q}{g^k}}^2}
		}_{T_2}.
	\end{align*}
Next, we derive the upper bounds for terms $T_1$ and $T_2$ separately.
For $T_2$ we use unbiasedness of quantization and independence of workers:
	\begin{align*}
		T_2 &=
		\EE{Q}{
			\norm{
				\frac{1}{n}
				\sum\limits_{i=1}^{n}
					\cQ(g_i^k - h_i^k) - (g_i^k - h_i^k)
			}^2
		}\\
		&{=}
		\frac{1}{n^2}
		\sum\limits_{i=1}^{n}
			\EE{\cQ}{
				\norm{\cQ(g_i^k - h_i^k) - (g_i^k - h_i^k)}^2
			} \overset{\eqref{eq:quant}}{\leq}
		\frac{\omega}{n^2}
		\sum\limits_{i=1}^{n}
			\norm{g_i^k - h_i^k}^2.
	\end{align*}
Taking $\Exp_k[\cdot]$ from the both sides of the above inequality, we derive
	\begin{align*}
		\EE{k}{T_2}
		&\leq
		\frac{\omega}{n^2}
		\sum\limits_{i=1}^{n}
			\EE{k}{\norm{g_i^k - h_i^k}^2}
		=
		\frac{\omega}{n^2}
		\sum\limits_{i=1}^{n}
		    \left(
			\norm{\EE{k}{g_i^k - h_i^k}}^2
			+
			\EE{k}{\norm{g_i^k - h_i^k - \EE{k}{g_i^k - h_i^k}}^2}
			\right)\\
		&=
		\frac{\omega}{n^2}
		\sum\limits_{i=1}^{n}
		    \left(
			\norm{F_i(x^k) - h_i^k}^2
			+
			\EE{k}{\norm{g_i^k - F_i(x^k)}^2}
			\right)\\
		&=
		\frac{\omega}{n^2}
		\sum\limits_{i=1}^{n}
		    \left(
			\norm{F_i(x^k) - h_i^k}^2
			+
			\EE{k}{\norm{F_{ij_i^k}(x^k) - F_{ij_i^k}(w_{i}^k) - \EE{k}{F_{ij_i^k}(x^k) - F_{ij_i^k}(w_{i}^k)}}^2}
			\right)\\
		&{\leq}
		\frac{\omega}{n^2}
		\sum\limits_{i=1}^{n}
		     \left(
			 \norm{F_i(x^k) - h_i^k}^2
			 +
			 \EE{k}{\norm{F_{ij_i^k}(x^k) - F_{ij_i^k}(w_{i}^k)}^2}
			 \right)\\
		&{\leq}
		\frac{2\omega}{n^2}
		\sum\limits_{i=1}^{n}
		    \left(
			\norm{h_i^k - F_i(x^\star)}^2
			+
			\norm{F_i(x^k) - F_i(x^\star)}^2
			\right)\\
		&\quad+
		\frac{2\omega}{n^2}
		\sum\limits_{i=1}^{n}
		    \left(
				\EE{k}{\norm{F_{ij_i^k}(w_{i}^k) - F_{ij_i^k}(x^\star)}^2}
				+
				\EE{k}{\norm{F_{ij_i^k}(x^k) - F_{ij_i^k}(x^\star)}^2}
				\right).
	\end{align*}				
Since $j_i^k$ is sampled uniformly at random from $[m]$, we have		
	\begin{eqnarray*}
		\EE{k}{T_2}
		&{\leq}&
		\frac{2\omega}{n^2}
		\sum\limits_{i=1}^{n}
		    \left(
			\norm{h_i^k - F_i(x^\star)}^2
			+
			\norm{F_i(x^k) - F_i(x^\star)}^2
			\right)\\
		&&\quad+
		\frac{2\omega}{mn^2}
		\sum\limits_{i=1}^{n}
			\sum\limits_{j=1}^{m}
			    \left(
				\EE{k}{\norm{F_{ij}(w_{i}^k) - F_{ij}(x^\star)}^2}
				+
				\EE{k}{\norm{F_{ij}(x^k) - F_{ij}(x^\star)}^2}
				\right) \\
		&\overset{\eqref{eq:averaged_cocoercivity},\eqref{eq:sum_averaged_cocoercivity}}{\leq}&
		\frac{2\omega}{n^2} H^k
		+
		\frac{2\omega}{mn^2} D^k
		+
		\frac{2\omega (\widehat \ell + \widetilde \ell)}{n} \langle F(x^k) - F(x^{*}), x^k - x^{*} \rangle.
	\end{eqnarray*}
In last line, we also use the definitions of $H^k$, $D^k$.  For $T_1$ we use definition of $g^k$:
	\begin{align*}
		T_1
		&= 
		\norm{
			\frac{1}{n}
			\sum\limits_{i=1}^{n}
				\EE{\cQ}{\cQ(g_i^k - h_i^k) + h_i^k} - F(x^*)
		}^2
		=
		\norm{
			\frac{1}{n}
			\sum\limits_{i=1}^{n}
				g_i^k - F(x^*).
		}^2
	\end{align*}
Next, we estimate $\Exp_k[T_1]$ similarly to $\Exp_k[T_2]$:
	\begin{align*}
		\EE{k}{T_1}
		&=
		\EE{k}{\norm{
			\frac{1}{n}
			\sum\limits_{i=1}^{n}
				g_i^k - F(x^*)
		}^2}
		=
		\norm{
			\frac{1}{n}
			\sum\limits_{i=1}^{n}
			\E{g_i^k} - F(x^*)
		}^2
		+
		\EE{k}{\norm{
				\frac{1}{n}
				\sum\limits_{i=1}^{n}
				\left( g_i^k - \E{g_i^k}\right)
			}_2^2}\\
		&=
		\norm{F(x^k) - F(x^*) }^2
		+
		\frac{1}{n^2}
		\sum\limits_{i=1}^{n}
			\EE{k}{\norm{g_i^k - F_i(x^k)}^2}\\
		&\overset{\eqref{eq:cocoercivity}}{\leq}
		\ell \langle F(x^k) - F(x^{*}), x^k - x^{*} \rangle
		\\
		&\quad +
		\frac{1}{n^2}
		\sum\limits_{i=1}^{n}
			\E{\norm{F_{ij_i^k}(x^k) - F_{ij_i^k}(w_{i}^k) - \EE{k}{F_{ij_i^k}(x^k) -F_{ij_i^k}(w_{i}^k)}}^2}\\
		&{\leq}
		\ell \langle F(x^k) - F(x^{*}), x^k - x^{*} \rangle
		+
		\frac{1}{n^2}
		\sum\limits_{i=1}^{n}
			\EE{k}{\norm{F_{ij_i^k}(x^k) - F_{ij_i^k}(w_{i}^k)}^2}\\
		&{=}
		\ell \langle F(x^k) - F(x^{*}), x^k - x^{*} \rangle
		+
		\frac{1}{mn^2}
		\sum\limits_{i=1}^{n}
			\sum\limits_{j=1}^{m}
				\norm{F_{ij}(x^k) - F_{ij}(w_{i}^k)}^2\\
		&{\leq}
		\ell \langle F(x^k) - F(x^{*}), x^k - x^{*} \rangle
		+
		\frac{2}{mn^2}
		\sum\limits_{i=1}^{n}
			\sum\limits_{j=1}^{m}
			    \left(
				\norm{F_{ij}(w_{i}^k) - F_{ij}(x^\star)}^2
				+
				\norm{F_{ij}(x^k) - F_{ij}(x^\star)}^2
				\right)\\
		&\overset{\eqref{eq:sum_averaged_cocoercivity}}{\leq}
		\left(\ell + \frac{2 \widetilde \ell}{n}\right)\langle F(x^k) - F(x^{*}), x^k - x^{*} \rangle
		+
		\frac{2}{mn^2} D^k.
	\end{align*}
	Finally, summing $\E{T_1}$ and $\E{T_2}$ we get
	\begin{align*}
		\E{\norm{g^k - F(x^*)}^2}
		&=
		\E{T_1 + T_2}
		\\
		&\leq
		\left(\ell + \frac{2 \widetilde \ell}{n}\right)\langle F(x^k) - F(x^{*}), x^k - x^{*} \rangle
		+
		\frac{2}{mn^2} D^k\\
		&\quad +
		\frac{2\omega}{n^2} H^k
		+
		\frac{2\omega}{mn^2} D^k
		+
		\frac{2\omega (\widehat \ell + \widetilde \ell)}{n} \langle F(x^k) - F(x^{*}), x^k - x^{*} \rangle\\
		&\leq
		\left(
			\ell + \frac{2 \widetilde \ell}{n} + \frac{2 \omega (\widehat \ell + \widetilde \ell)}{n}
		\right) \langle F(x^k) - F(x^{*}), x^k - x^{*} \rangle
		+
		\frac{2\omega}{n^2} H^k
		+
		\frac{2(\omega+1)}{mn^2} D^k,
	\end{align*}
which concludes the proof since $\sigma^2_k = \frac{H^k}{n} + \frac{D^k}{nm}$.
\end{proof}

\begin{lemma}[Modification of Lemmas 5 and 6 from \cite{horvath2019stochastic}]\label{lem:vr_diana_horv2}
    Let $F$  be $\ell$-star-cocoercive and Assumptions~\ref{as:averaged_cocoercivity}, \ref{as:unique_solution}, \ref{as:sum_averaged_cocoercivity} hold. Suppose that $\alpha \leq \min\left\{\frac{p}{3};\frac{1}{1+\omega}\right\}$. Then for all $k \ge 0$ \algname{VR-DIANA-SGDA} satisfies
    \begin{eqnarray*}
        \Exp_k \left[ \sigma^2_{k+1}\right] &\leq& (1 - \alpha) \sigma^2_k + \left( p \widetilde \ell + 2\alpha(\widetilde \ell + \widehat \ell)\right) \langle F(x^k) - F(x^{*}), x^k - x^{*} \rangle,
    \end{eqnarray*}
    where $\sigma^2_k = \frac{H^k}{n} + \frac{D^k}{nm}$ with $H^k = \sum\limits_{i=1}^n \left\| h^k_i - F_i (x^*) \right\|^2$ and $D^k = \sum\limits_{i=1}^n \sum\limits_{j=1}^m \left\| F_{ij} (w^k_{i}) - F_{ij} (x^*) \right\|^2$.
\end{lemma}
\begin{proof} We start with considering $H^{k+1}$:
\begin{eqnarray*}
		\EE{k}{H^{k+1}}
		&=&
		\EE{k}{\sum\limits_{i=1}^{n}
			\norm{h_i^{k+1} - F_i(x^\star)}^2} \\
		&=&
		\sum\limits_{i=1}^{n}
			\norm{h_i^k - F_i(x^\star)}^2 +
		\sum\limits_{i=1}^{n}
			\EE{k}{
				2\langle\alpha \cQ(g_i^k - h_i^k), h_i^k - F_i(x^\star) \rangle
				+
				\alpha^2\norm{\cQ(g_i^k - h_i^k)}^2
			}\\
		&\overset{\eqref{eq:quant}}{\leq}&
		H^k
		+
		\sum\limits_{i=1}^{n}\EE{k}{
			2\alpha \langle g_i^k - h_i^k, h_i^k - F_i(x^\star)\rangle
		+
		\alpha^2 (\omega+1) \norm{g_i^k - h_i^k}^2}.\\
	\end{eqnarray*}
Since $\alpha \leq \nicefrac{1}{(\omega+1)}$, we have
	\begin{eqnarray*}
	\EE{k}{H^{k+1}}
		&\leq&
		H^k
		+
		\EE{k}{\sum\limits_{i=1}^{n}
			\alpha \langle g_i^k - h_i^k, g_i^k + h_i^k - 2F_i(x^\star)\rangle}\\
		&=&
		H^k
		+
		\EE{k}{\sum\limits_{i=1}^{n}
			\alpha \langle g_i^k - F_i(x^\star) + F_i(x^\star) - h_i^k, g_i^k - F_i(x^\star) + h_i^k - F_i(x^\star)\rangle}\\
		&=&
		H^k
		+
		\EE{k}{\sum\limits_{i=1}^{n}
			\alpha 
			\left(
				\norm{g_i^k - F_i(x^\star)}^2
				-
				\norm{h_i^k - F_i(x^\star)}^2
			\right)}\\
		&=&
		H^k(1-\alpha)
		+
		\EE{k}{\sum\limits_{i=1}^{n}
			\alpha 
			\left(
				\norm{g_i^k - F_i(x^\star)}^2
			\right)}\\
		&{\leq}&
		H^k(1-\alpha)
		+
		\sum\limits_{i=1}^{n}
		\left(
			2\alpha \EE{k}{		
				\norm{g_i^k - F_i(x^k)}^2
			} 
			+ 2\alpha\norm{F_i(x^k) - F_i(x^\star)}^2
			\right)\\
		&{=}&
		H^k(1-\alpha)
		+
		\sum\limits_{i=1}^{n}\EE{k}{
			2\alpha 
				\norm{F_{ij_i^k}(x^k) - F_{ij_i^k}(w_{i}^k) - \EE{k}{F_{ij_i^k}			(x^k) - F_{ij_i^k}(w_{i}^k)}}^2
			} \\
			&&\quad +
			2\alpha
			\sum\limits_{i=1}^{n}
			\norm{F_i(x^k) - F_i(x^\star)}^2
			\\
		&{\leq}&
		H^k(1-\alpha)
		+
		\sum\limits_{i=1}^{n}\left(
		\EE{k}{
			2\alpha 
				\norm{F_{ij_i^k}(x^k) - F_{ij_i^k}(w_{i}^k) }^2
			}
			+
			2\alpha\norm{F_i(x^k) - F_i(x^\star)}^2
			\right)\\
		&{\leq}&
		H^k(1-\alpha)
		+
		\frac{2\alpha}{m}
		\sum\limits_{i=1}^{n}
			\sum\limits_{j=1}^{m}
			\left( \norm{F_{ij}(x^k) - F_{ij}(x^\star)}^2 +
				\norm{F_{ij}(w_{i}^k) - F_{ij}(x^\star)}^2\right)\\
		&&\quad +
			2\alpha
			\sum\limits_{i=1}^{n}
			\norm{F_i(x^k) - F_i(x^\star)}_2^2
			\\
		&\overset{\eqref{eq:averaged_cocoercivity}, \eqref{eq:sum_averaged_cocoercivity}}{\leq}&
		H^k(1-\alpha)
		+
		\frac{2\alpha}{m}
		\sum\limits_{i=1}^{n}
			\sum\limits_{j=1}^{m}
				\norm{F_{ij}(w_{ij}^k) - F_{ij}(x^\star)}_2^2
		\\
		&&\quad+
		2\alpha n(\widetilde \ell + \widehat \ell) \langle F(x^k) - F(x^{*}), x^k - x^{*} \rangle\\
		&=&
		H^k(1-\alpha) + \frac{2\alpha}{m} D^k + 2\alpha n(\widetilde \ell + \widehat \ell) \langle F(x^k) - F(x^{*}), x^k - x^{*} \rangle.
	\end{eqnarray*}
Next, we consider $D^{k+1}$
\begin{align*}
		\EE{k}{D^{k+1}}
		&=
		\sum\limits_{i=1}^{n}
			\sum\limits_{j=1}^{m}
				\EE{k}{\norm{F_{ij}(w_{i}^{k+1}) - F_{ij}(x^\star)}^2}\\
		&=
		\sum\limits_{i=1}^{n}
			\sum\limits_{j=1}^{m}
				\left[
					\left(
						1 - p
					\right)
					\norm{F_{ij}(w_{ij}^k) - F_{ij}(x^\star)}_2^2
					+
					p
					\norm{F_{ij}(x^k) - F_{ij}(x^\star)}_2^2			
				\right]\\
		&\overset{\eqref{eq:sum_averaged_cocoercivity}}{\leq}
		D^k
		\left(
			1 - p
		\right)
		+
		nmp\widetilde \ell \langle F(x^k) - F(x^{*}), x^k - x^{*} \rangle.
	\end{align*}	
It remains put the upper bounds on $D^{k+1}$, $H^{k+1}$ together and use the definition	of $\sigma^2_{k+1}$:
\begin{align*}
    \EE{k}{\sigma^2_{k+1}} &= \frac{\EE{k}{H^{k+1}}}{n} + \frac{\EE{k}{D^{k+1}}}{nm} \\
    &\leq 
    (1 - \alpha)\frac{{H^{k}}}{n} + (1 + 2\alpha -p)\frac{{D^{k}}}{nm} + \left( p \widetilde \ell + 2\alpha(\widetilde \ell + \widehat \ell)\right)\langle F(x^k) - F(x^{*}), x^k - x^{*} \rangle
    \end{align*}
With $\alpha \leq \tfrac{p}{3}$ we get $-p \leq -3\alpha$, implying 
\begin{align*}
    \EE{k}{\sigma^2_{k+1}} 
    &\leq 
    (1 - \alpha)\frac{{H^{k}}}{n} + (1 - \alpha)\frac{{D^{k}}}{nm} + \left( p \widetilde \ell + 2\alpha(\widetilde \ell + \widehat \ell)\right)\langle F(x^k) - F(x^{*}), x^k - x^{*} \rangle \\
    & = (1 - \alpha) \sigma^2_k + \left( p \widetilde \ell + 2\alpha(\widetilde \ell + \widehat \ell)\right)\langle F(x^k) - F(x^{*}), x^k - x^{*} \rangle.
    \end{align*}
\end{proof}

The above two lemmas imply that Assumption~\ref{as:key_assumption} is satisfied with certain parameters.

\begin{proposition}[Proposition~\ref{thm:prox_VR_DIANA_convergence}]\label{thm:prox_VR_DIANA_convergence_appendix_1}
    Let $F$  be $\ell$-star-cocoercive and Assumptions~\ref{as:averaged_cocoercivity}, \ref{as:unique_solution}, \ref{as:sum_averaged_cocoercivity} hold. Suppose that $\alpha \leq \min\left\{\frac{p}{3};\frac{1}{1+\omega}\right\}$.  Then, \algname{VR-DIANA-SGDA} satisfies Assumption~\ref{as:key_assumption} with
    \begin{gather*}
        A = \left(
			\frac{\ell}{2} + \frac{\widetilde \ell}{n} + \frac{\omega (\widehat \ell + \widetilde \ell)}{n}
		\right),\quad B = \frac{2(\omega+1)}{n},\\
		\sigma_k^2 = \frac{1}{n} \sum\limits_{i=1}^n \left\| h^k_i - F_i (x^*) \right\|^2 + \frac{1}{nm}\sum\limits_{i=1}^n \sum\limits_{j=1}^m \left\| F_{ij} (w^k_{i}) - F_{ij} (x^*) \right\|^2 , \\
		C = \left( \frac{p \widetilde l}{2} + \alpha(\widetilde \ell + \widehat \ell)\right),\quad \rho = \alpha \leq \min\left\{\frac{p}{3};\frac{1}{1+\omega}\right\},\quad D_1 = D_2 = 0.
    \end{gather*}
\end{proposition}

\subsubsection{Analysis of \algname{VR-DIANA-SGDA} in the Quasi-Strongly Monotone Case}

Applying Theorem~\ref{thm:main_result} and Corollary~\ref{cor:main_result} with $M = \tfrac{4(\omega+1)}{n \alpha}$, we get the following results.

\begin{theorem}\label{thm:prox_VR_DIANA_convergence_appendix}
    Let $F$ be $\mu$-quasi strongly monotone, $\ell$-star-cocoercive and Assumptions~\ref{as:averaged_cocoercivity}, \ref{as:unique_solution}, \ref{as:sum_averaged_cocoercivity} hold. Suppose that $\alpha \leq \min\left\{\frac{p}{3};\frac{1}{1+\omega}\right\}$ and 
	$$
    0 < \gamma \leq \left(\ell + \frac{10(\omega + 1) (\widehat \ell + \widetilde \ell)}{n} +  \frac{4(\omega+1) p \widetilde l}{\alpha n}  \right)^{-1}.
	$$
	Then for all $k \ge 0$ the iterates of \algname{VR-DIANA-SGDA} satisfy
    \begin{equation*}
        \Exp\left[\|x^k - x^{*}\|^2\right] \leq \left(1 - \min\left\{\gamma\mu, \nicefrac{1}{\alpha n}\right\}\right)^k V_0, \label{eq:prox_VR_DIANA_convergence_appendix}
    \end{equation*}
    where $V_0 = \|x^0 - x^{*}\|^2 + \tfrac{4(\omega+1)\gamma^2}{n \alpha}\sigma_0^2$.
\end{theorem}

\begin{corollary}\label{cor:prox_VR_DIANA_convergence_appendix}
    Let the assumptions of Theorem~\ref{thm:prox_VR_DIANA_convergence_appendix} hold. Then, for $p = \tfrac{1}{m}$, $\alpha = \min\left\{\tfrac{1}{3m}, \tfrac{1}{1+\omega}\right\}$,
    \begin{equation*}
            \gamma = \left(\ell + \frac{10(\omega + 1) (\widehat \ell + \widetilde \ell)}{n} +  \frac{4(\omega+1)\max\{3m, 1+\omega\} \widetilde \ell}{nm}  \right)^{-1}
        \end{equation*}
    and any $K \ge 0$ we have
	\begin{equation*}
            \Exp[\|x^k - x^{*}\|^2] \leq V_0 \exp\left(-\min\left\{\frac{\mu}{\ell + \frac{10(\omega + 1) (\widehat \ell + \widetilde \ell)}{n} +  \frac{4(\omega+1)\max\{3m, 1+\omega\} \widetilde \ell}{nm}}, \frac{1}{6m}, \frac{1}{2(1+\omega)}\right\}K\right). 
        \end{equation*}
\end{corollary}

\subsubsection{Analysis of \algname{VR-DIANA-SGDA} in the Monotone Case}
Next, using Theorem~\ref{thm:main_result_monotone}, we establish the convergence of \algname{VR-DIANA-SGDA} in the monotone case.

\begin{theorem}\label{thm:prox_VR_DIANA_convergence_monotone}
   Let $F$  be monotone, $\ell$-star-cocoercive and Assumptions~\ref{as:key_assumption}, \ref{as:boundness}, ~\ref{as:averaged_cocoercivity}, \ref{as:unique_solution}, \ref{as:sum_averaged_cocoercivity} hold. Assume that 
   \begin{equation*}
        0 < \gamma \leq \left(\ell + \frac{6(\omega + 1) (\widehat \ell + \widetilde \ell)}{n} + \frac{2(\omega+1) p \widetilde l}{\alpha n}\right)^{-1}
    \end{equation*}
    and $\alpha = \min\left\{\tfrac{p}{3}, \tfrac{1}{1+\omega}\right\}$. Then for $\text{Gap}_{\cC} (z)$ from \eqref{eq:gap} and for all $K\ge 0$ the iterates produced by \algname{VR-DIANA-SGDA} satisfy
    \begin{eqnarray*}
    \Exp\left[\text{Gap}_{\cC} \left(\frac{1}{K}\sum\limits_{k=1}^{K}  x^{k}\right)\right] 
    &\leq& \frac{3\left[\max_{u \in \mathcal{C}}\|x^{0} - u\|^2\right]}{2\gamma K} \notag\\
    &&\quad + \left(
			3\ell + \frac{12(\omega + 1) (\widehat \ell + \widetilde \ell)}{n}
		 + \frac{8(\omega+1) p \widetilde l}{\alpha n}\right)
\cdot \frac{\|x^0-x^{*,0}\|^2}{K} \notag\\
    &&\quad + \left(4 +  \gamma \left(
			3\ell + \frac{12(\omega + 1) (\widehat \ell + \widetilde \ell)}{n}
		 + \frac{8(\omega+1) p \widetilde l}{\alpha n}\right)\right) \frac{\gamma B \sigma_0^2}{\rho K} \notag.
\end{eqnarray*}
\end{theorem}
Applying Corollary~\ref{cor:main_result_monotone}, we get the rate of convergence to the exact solution.
\begin{corollary}\label{cor:prox_VR_DIANA_convergence_monotone}
    Let the assumptions of Theorem~\ref{thm:prox_VR_DIANA_convergence_monotone} hold. Then $\forall K > 0$ one can choose $p = \tfrac{1}{m}$, $\alpha = \min\left\{\tfrac{1}{3m}, \tfrac{1}{1+\omega}\right\}$ and $\gamma$ as
    \begin{eqnarray*}
        \gamma &=& \min\Bigg\{\frac{1}{3\ell + \frac{12(\omega + 1) (\widehat \ell + \widetilde \ell)}{n}
		 + \frac{8(\omega+1) \max\{3m, 1+\omega\} \widetilde \ell}{m n}},\\
		 &&\quad\quad\quad\quad\quad\quad\quad \frac{\Omega_{0,\cC}\sqrt{n}}{\Omega_{0,\cC} \sqrt{2\max\{3m, 1+\omega\} (\omega + 1)(\widetilde {\ell} + \widehat{\ell})\ell}}, \frac{\Omega_{0,\cC}}{G_*\sqrt{K}}\Bigg\}.
    \end{eqnarray*}
    This choice of $\alpha$ and $\gamma$ implies
   \begin{align*}
        \Exp\left[\text{Gap}_{\cC} \left(\frac{1}{K}\sum\limits_{k=1}^{K}  x^{k}\right)\right] &= \cO\Bigg(\frac{\left(
			\ell + \nicefrac{(\omega + 1) (\widehat \ell + \widetilde \ell)}{n}
		 + \nicefrac{(\omega+1) \max\{m, \omega\} \widetilde \ell}{m n}\right)(\Omega_{0,\cC}^2 + \Omega_0^2) + \ell \Omega_{\cC}^2}{K} \\\notag
		 &\quad\quad\quad\quad\quad\quad\quad+ \frac{\Omega^2_{0,\cC} \sqrt{\max\{m, \omega\}(\omega + 1)(\widetilde {\ell} + \widehat{\ell})\ell}}{\sqrt{n} K} + \frac{\Omega_{0,\cC} G_*}{\sqrt{K}}\Bigg).\notag
    \end{align*}
\end{corollary}
\begin{proof}
The proof follows from the next upper bound $\widehat \sigma_0^2$ for $\sigma^2_0$ with initialization $h^0_i = F_i(x^0)$ and $w_i = x_0$
    \begin{eqnarray*}
        \sigma_0^2 &=& \frac{1}{nm}\sum_{i=1}^n \sum_{j=1}^m\|F_{ij}(x^0) - F_{ij}(x^{*})\|^2 + \frac{1}{n}\sum_{i=1}^n \|F_i(x^0) - F_{i}(x^{*})\|^2\\
        &\overset{}{\leq}& (\widetilde {\ell} + \widehat{\ell})\langle F(x^0) - F(x^*), x^0 - x^* \rangle\\
        &\leq& (\widetilde {\ell} + \widehat{\ell}) \|F(x^0) - F(x^*)\|\cdot \|x^0 - x^*\|\\
        &\leq& (\widetilde {\ell} + \widehat{\ell})\ell \|x^0 - x^*\|^2 \leq (\widetilde {\ell} + \widehat{\ell})\ell\max_{u \in \mathcal{C}}\|x^{0} - u\|^2 \leq (\widetilde {\ell} + \widehat{\ell})\ell\Omega_{0,\cC}^2.
    \end{eqnarray*}
    Next, applying Corollary~\ref{cor:main_result_monotone} with $\widehat{\sigma}_0 \eqdef \sqrt{(\widetilde {\ell} + \widehat{\ell})\ell}\Omega_{0,\cC}$, we get the result.
\end{proof}

\subsubsection{Analysis of \algname{VR-DIANA-SGDA} in the Cocoercive Case}
Next, using Theorem~\ref{thm:main_result_monotone_coco}, we establish the convergence of \algname{VR-DIANA-SGDA} in the cocoercive case.

\begin{theorem}\label{thm:prox_VR_DIANA_convergence_monotone_coco}
   Let $F$  be $\ell$-cocoercive and Assumptions~\ref{as:key_assumption}, \ref{as:boundness}, ~\ref{as:averaged_cocoercivity}, \ref{as:unique_solution}, \ref{as:sum_averaged_cocoercivity} hold. Assume that 
   \begin{equation*}
        0 < \gamma \leq \left(\ell + \frac{6(\omega + 1) (\widehat \ell + \widetilde \ell)}{n} + \frac{2(\omega+1) p \widetilde l}{\alpha n}\right)^{-1}
    \end{equation*}
    and $\alpha = \min\left\{\tfrac{p}{3}, \tfrac{1}{1+\omega}\right\}$. Then for $\text{Gap}_{\cC} (z)$ from \eqref{eq:gap} and for all $K\ge 0$ the iterates produced by \algname{VR-DIANA-SGDA} satisfy
    \begin{eqnarray*}
    \Exp\left[\text{Gap}_{\cC} \left(\frac{1}{K}\sum\limits_{k=1}^{K}  x^{k}\right)\right] 
    &\leq& \frac{3\left[\max_{u \in \mathcal{C}}\|x^{0} - u\|^2\right]}{2\gamma K}  \notag\\
    &&\quad + \left(
			6\ell + \frac{18(\omega + 1) (\widehat \ell + \widetilde \ell)}{n}
		 + \frac{12(\omega+1) p \widetilde l}{\alpha n}\right)
\cdot \frac{\|x^0-x^{*,0}\|^2}{K} \notag\\
    &&\quad + \left(6 +  \gamma \left(
			6\ell + \frac{18(\omega + 1) (\widehat \ell + \widetilde \ell)}{n}
		 + \frac{12(\omega+1) p \widetilde l}{\alpha n}\right)\right) \frac{\gamma B \sigma_0^2}{\rho K} \notag.
\end{eqnarray*}
\end{theorem}
Applying Corollary~\ref{cor:main_result_monotone_coco}, we get the rate of convergence to the exact solution.
\begin{corollary}\label{cor:prox_VR_DIANA_convergence_monotone_coco}
    Let the assumptions of Theorem~\ref{thm:prox_VR_DIANA_convergence_monotone_coco} hold. Then $\forall K > 0$ one can choose $p = \tfrac{1}{m}$, $\alpha = \min\left\{\tfrac{1}{3m}, \tfrac{1}{1+\omega}\right\}$ and $\gamma$ as
    \begin{eqnarray*}
        \gamma &=& \min\Bigg\{\frac{1}{6\ell + \frac{18(\omega + 1) (\widehat \ell + \widetilde \ell)}{n}
		 + \frac{12(\omega+1) \max\{3m, 1+\omega\} \widetilde \ell}{m n}},\\
		 &&\quad\quad\quad\quad\quad\quad\quad \frac{\Omega_{0,\cC}\sqrt{n}}{\Omega_{0,\cC} \sqrt{2\max\{3m, 1+\omega\} (\omega + 1)(\widetilde {\ell} + \widehat{\ell})\ell}}\Bigg\}.
    \end{eqnarray*}
    This choice of $\alpha$ and $\gamma$ implies
   \begin{align*}
        \Exp\left[\text{Gap}_{\cC} \left(\frac{1}{K}\sum\limits_{k=1}^{K}  x^{k}\right)\right] &= \cO\Bigg(\frac{\left(
			\ell + \nicefrac{(\omega + 1) (\widehat \ell + \widetilde \ell)}{n}
		 + \nicefrac{(\omega+1) \max\{m, \omega\} \widetilde \ell}{m n}\right)(\Omega_{0,\cC}^2 + \Omega_0^2)}{K} \\\notag
		 &\quad\quad\quad\quad\quad\quad\quad+ \frac{\Omega^2_{0,\cC} \sqrt{\max\{m, \omega\}(\omega + 1)(\widetilde {\ell} + \widehat{\ell})\ell}}{\sqrt{n} K}\Bigg).\notag
    \end{align*}
\end{corollary}

\subsection{Discussion of the Results in the Monotone and Cocoercive Cases}

\citet{beznosikov2021distributed} also consider monotone case and derive the following rate for \algname{MASHA1} (neglecting the dependence on Lipschitz parameters and the quantities like $\Omega_{0,\cC}^2 = \max_{u\in \cC}\|x^0 - u\|^2$): $\cO\left(\sqrt{(m+\omega)(1+\nicefrac{\omega}{n})}\tfrac{1}{K}\right)$. In general, due to the term proportional to $\nicefrac{1}{\sqrt{K}}$ and due to the relation between (star-)cocoercivity constants and Lipschitz constants our rate\newline  $\cO\left(\tfrac{(1+\omega)}{nK} + \tfrac{(1+\omega)\max\{m,\omega\}}{mnK} + \tfrac{\sqrt{\max\{m,\omega\}(1+\omega)}}{\sqrt{n}K} + \frac{G_*}{\sqrt{K}}\right)$ our rate is worse than the one from \citet{beznosikov2021distributed} (even when $R(x)\equiv 0$, i.e., $G_* = 0$). However, when the difference between cocoercivity and Lipschitz constants is not significant, and $m, n$ or $\omega$ are sufficiently large, our result in the cocoercive case (Corollary~\ref{cor:prox_VR_DIANA_convergence_monotone_coco}) might be better. Moreover, we emphasize here that \citet{beznosikov2021distributed} do not consider \algname{SGDA} as the basis for their methods. To the best of our knowledge, our results are the first ones for distributed \algname{SGDA}-type methods with compression derived in the monotone case without assuming (quasi-)strong monotonicity.

\newpage

\section{COORDINATE \algname{SGDA}} \label{sec:coord_sgda}

In this section, we focus on the coordinate versions of \algname{SGDA}. To denote $i$-th component of the vector $x$ we use $[x]_i$. Vectors $e_1,\ldots, e_d \in \R^d$ form a standard basis in $\R^d$.

\subsection{\algname{CSGDA}}

\begin{algorithm}[h!]
   \caption{\algname{CSGDA}: Coordinate Stochastic Gradient Descent-Ascent}
   \label{alg:prox_CSGDA}
\begin{algorithmic}[1]
   \State {\bfseries Input:} starting point $x^0 \in \R^d$, stepsize $\gamma > 0$, number of steps $K$
   \For{$k=0$ {\bfseries to} $K-1$}
   \State Sample uniformly at random $j \in [d]$
   \State $g^k = d e_j [F(x^k)]_j$
   \State $x^{k+1} = \prox_{\gamma R}\left(x^k - \gamma g^k\right)$
   \EndFor
\end{algorithmic}
\end{algorithm}

\subsubsection{\algname{CSGDA} Fits Assumption~\ref{as:key_assumption}}
\begin{proposition}\label{thm:prox_CSGDA_convergence_appendix_1}
     Let $F$  be $\ell$-star-cocoercive. Then, \algname{CSGDA} satisfies Assumption~\ref{as:key_assumption} with
    \begin{gather*}
        A = d\ell,\quad D_1 = 2d \max_{x* \in X^*}\left[\left\| F(x^{*})\right\|^2\right],\quad \sigma_k^2 =0 , \quad B = 0,\quad C = 0,  \quad \rho =1,\quad D_2 = 0.
    \end{gather*}
\end{proposition}
\begin{proof}
    First of all, for all $a\in \R^d$ and for random index $j$ uniformly distributed on $[d]$ we have $\Exp_j[\|e_j[a]_j\|^2] = \tfrac{1}{d}\sum_{i=1}^d[a]_j^2 = \tfrac{1}{d}\|a\|^2$. Using this and $g^k = d e_j [F(x^k)]_j$, we derive
    \begin{eqnarray}
        \Exp_k\left[\|g^k - F(x^{*,k})\|^2\right] &=& 
        \Exp_k\left[\|d e_j [F(x^k) - F(x^{*,k})]_j + d e_j [F(x^{*,k})]_j - F(x^{*,k})\|^2\right]\notag \\
        &\leq& 2\Exp_k\left[\|d e_j [F(x^k) - F(x^{*,k})]_j \|^2\right] + 2\Exp_k\left[\|d e_j [F(x^{*,k})]_j - F(x^{*,k})\|^2\right]\notag \\
        &=& 2d\|F(x^k) - F(x^{*,k}) \|^2 + 2\Exp_k\left[\|d e_j [F(x^{*,k})]_j - \Exp_k[d e_j [F(x^{*,k})]_j]\|^2\right]\notag \\
        &\leq& 2d\|F(x^k) - F(x^{*,k}) \|^2 + 2\Exp_k\left[\|d e_j [F(x^{*,k})]_j\|^2\right]\notag \\
        &=& 2d\|F(x^k) - F(x^{*,k}) \|^2 + 2d\|F(x^{*,k})\|^2. \label{eq:bhsjcbhjcbjs}
\end{eqnarray}
Finally, the star-cocoercivity of $F$ implies
\begin{eqnarray*}
        \Exp_k\left[\|g^k - F(x^{*,k})\|^2\right] 
        &\leq& 2d\ell\langle F(x^k) - F(x^{*,k}), x^k - x^{*} \rangle + 2d\|F(x^{*,k})\|^2\\
        &\leq& 2d\ell\langle F(x^k) - F(x^{*,k}), x^k - x^{*} \rangle + 2d \max_{x* \in X^*}\left[\left\| F(x^{*})\right\|^2\right].
\end{eqnarray*}
\end{proof}

\subsubsection{Analysis of \algname{CSGDA} in the Quasi-Strongly Monotone Case}

Applying Theorem~\ref{thm:main_result} and Corollary~\ref{cor:main_result}, we get the following results.

\begin{theorem}\label{thm:prox_CSGDA_convergence_appendix}
     Let $F$  be $\mu$-quasi strongly monotone  and $\ell$-star-cocoercive, $0 < \gamma \leq \nicefrac{1}{2d \ell}$. Then for all $k \ge 0$
    \begin{equation*}
        \Exp\left[\|x^k - x^{*}\|^2\right] \leq \left(1 - \gamma\mu\right)^k \|x^0 - x^{*,0}\|^2 + \frac{2\gamma d}{\mu} \cdot \max_{x* \in X^*}\left[\left\| F(x^{*})\right\|^2\right]. \label{eq:prox_CSGDA_convergence_appendix}
    \end{equation*}
\end{theorem}

\begin{corollary}\label{cor:prox_CSGDA_convergence_appendix}
    Let the assumptions of Theorem~\ref{thm:prox_CSGDA_convergence_appendix} hold. Then, for any $K \ge 0$ one can choose $\{\gamma_k\}_{k \ge 0}$ as follows:
	\begin{eqnarray}
		\text{if } K \le \frac{2d \ell}{\mu}, && \gamma_k = \frac{1}{2d \ell},\notag\\
		\text{if } K > \frac{2d \ell}{\mu} \text{ and } k < k_0, && \gamma_k = \frac{1}{2d \ell},\notag\\
		\text{if } K > \frac{2d \ell}{\mu} \text{ and } k \ge k_0, && \gamma_k = \frac{2}{\mu(4d \ell + \mu(k - k_0))},\notag
	\end{eqnarray}
	where $k_0 = \left\lceil \nicefrac{K}{2} \right\rceil$. For this choice of $\gamma_k$ the following inequality holds:
	\begin{eqnarray*}
		\Exp[V_K] \le \frac{64d \ell}{\mu}\|x^0 - x^{*,0}\|^2\exp\left(-\frac{\mu K}{2d \ell}\right) + \frac{72 d }{\mu^2 K} \cdot \max_{x* \in X^*}\left[\left\| F(x^{*})\right\|^2\right].
	\end{eqnarray*}
\end{corollary}

\subsubsection{Analysis of \algname{CSGDA} in the Monotone Case}
Next, using Theorem~\ref{thm:main_result_monotone}, we establish the convergence of \algname{CSGDA} in the monotone case.

\begin{theorem}\label{thm:prox_CSGDA_convergence_monotone}
   Let $F$  be monotone, $\ell$-star-cocoercive and Assumptions~\ref{as:key_assumption}, \ref{as:boundness} hold. Assume that $\gamma \leq \nicefrac{1}{2d{\ell}}$. Then for $\text{Gap}_{\cC} (z)$ from \eqref{eq:gap} and for all $K\ge 0$ the iterates produced by \algname{CSGDA} satisfy
    \begin{eqnarray*}
    \Exp\left[\text{Gap}_{\cC} \left(\frac{1}{K}\sum\limits_{k=1}^{K}  x^{k}\right)\right] 
    &\leq& \frac{3\left[\max_{u \in \mathcal{C}}\|x^{0} - u\|^2\right]}{2\gamma K}  + \frac{8\gamma \ell^2 \Omega_{\mathcal{C}}^2}{K} +  \frac{5d\ell  \|x^0-x^{*,0}\|^2}{K}\\
    &&\quad +20\gamma d \cdot  \max_{x* \in X^*}\left[\left\| F(x^{*})\right\|^2\right].
\end{eqnarray*}
\end{theorem}

Applying Corollary~\ref{cor:main_result_monotone}, we get the rate of convergence to the exact solution.

\begin{corollary}\label{cor:prox_CSGDA_convergence_monotone_appendix}
    Let the assumptions of Theorem~\ref{thm:prox_CSGDA_convergence_monotone} hold. Then $\forall K > 0$ one can choose $\gamma$ as
    \begin{equation*}
        \gamma = \min\left\{\frac{1}{5d\ell}, \frac{\Omega_{0,\cC}}{G_* \sqrt{2dK  }}\right\}.
    \end{equation*}
    This choice of $\gamma$ implies
   \begin{align}
        \Exp\left[\text{Gap}_{\cC} \left(\frac{1}{K}\sum\limits_{k=1}^{K}  x^{k}\right)\right] = \cO\left(\frac{d \ell(\Omega_{0,\cC}^2 + \Omega_0^2) + \ell \Omega_{\cC}^2 }{K}  + \frac{\Omega_{0,\cC} G_*}{\sqrt{K}}\right).\notag
    \end{align}
\end{corollary}

\subsubsection{Analysis of \algname{CSGDA} in the Cocoercive Case}
Next, using Theorem~\ref{thm:main_result_monotone_coco}, we establish the convergence of \algname{CSGDA} in the cocoercive case.

\begin{theorem}\label{thm:prox_CSGDA_convergence_monotone_coco}
   Let $F$  be $\ell$-cocoercive and Assumptions~\ref{as:key_assumption}, \ref{as:boundness} hold. Assume that $\gamma \leq \nicefrac{1}{2d{\ell}}$. Then for $\text{Gap}_{\cC} (z)$ from \eqref{eq:gap} and for all $K\ge 0$ the iterates produced by \algname{CSGDA} satisfy
    \begin{eqnarray*}
    \Exp\left[\text{Gap}_{\cC} \left(\frac{1}{K}\sum\limits_{k=1}^{K}  x^{k}\right)\right] 
    &\leq& \frac{3\left[\max_{u \in \mathcal{C}}\|x^{0} - u\|^2\right]}{2\gamma K} +  \frac{9d\ell  \|x^0-x^{*,0}\|^2}{K}\\
    &&\quad +16 \gamma d \cdot  \max_{x* \in X^*}\left[\left\| F(x^{*})\right\|^2\right].
\end{eqnarray*}
\end{theorem}

Applying Corollary~\ref{cor:main_result_monotone_coco}, we get the rate of convergence to the exact solution.

\begin{corollary}\label{cor:prox_CSGDA_convergence_monotone_coco_appendix}
    Let the assumptions of Theorem~\ref{thm:prox_CSGDA_convergence_monotone_coco} hold. Then $\forall K > 0$ one can choose $\gamma$ as
    \begin{equation*}
        \gamma = \min\left\{\frac{1}{9d\ell}, \frac{\Omega_{0,\cC}}{G_* \sqrt{2dK  }}\right\}.
    \end{equation*}
    This choice of $\gamma$ implies
   \begin{align}
        \Exp\left[\text{Gap}_{\cC} \left(\frac{1}{K}\sum\limits_{k=1}^{K}  x^{k}\right)\right] = \cO\left(\frac{d \ell(\Omega_{0,\cC}^2 + \Omega_0^2)}{K} + \frac{\Omega_{0,\cC} G_*}{\sqrt{K}}\right).\notag
    \end{align}
\end{corollary}

\subsection{\algname{SEGA-SGDA}}

In this section, we consider a modification of \algname{SEGA} \citep{hanzely2018sega} -- the linearly converging coordinate method for composite optimization problems working even for non-separable regularizers.

\begin{algorithm}[h!]
   \caption{\algname{SEGA-SGDA}: \algname{SEGA} Stochastic Gradient Descent-Ascent \cite{hanzely2018sega}}
   \label{alg:prox_SEGA_SGDA}
\begin{algorithmic}[1]
    \State {\bfseries Input:} starting point $x^0\in \R^d$,  stepsize $\gamma > 0$, number of steps $K$
   \State Set $h^0 = 0$
   \For{$k=0$ {\bfseries to} $K-1$}
   \State Sample uniformly at random $j \in [d]$
   \State $h^{k+1} = h^k + e_j ([F(x^k)]_j - h^k_j)$
   \State $g^k = d e_j ([F(x^k)]_j - h^k_j) + h^k$
   \State $x^{k+1} = \prox_{\gamma R}\left(x^k - \gamma g^k\right)$
   \EndFor
\end{algorithmic}
\end{algorithm}

\subsubsection{\algname{SEGA-SGDA} Fits Assumption~\ref{as:key_assumption}}

The following result from \citet{hanzely2018sega} does not rely on the fact that $F(x)$ is the gradient of some function. Therefore, it holds in our settings as well.

\begin{lemma}[Lemmas A.3 and A.4 from \cite{hanzely2018sega}]\label{lem:sega}
    Let Assumption \ref{as:unique_solution} hold. Then for all $k \ge 0$ \algname{SEGA-SGDA} satisfies
    \begin{eqnarray*}
        \Exp_k \left[ \|g^k - F(x^{*}) \|^2\right] &\leq& 2d\|F(x^k) - F(x^*) \|^2 + 2d\sigma_k^2, \\
        \Exp_k \left[ \sigma^2_{k+1}\right] &\leq& \left( 1 - \frac{1}{d}\right) \sigma_k^2 + \frac{1}{d}\|F(x^k) - F(x^*) \|^2,
    \end{eqnarray*}
    where $\sigma_k^2 =  \|h^k - F(x^*) \|^2$.
\end{lemma}

The lemma above implies that Assumption~\ref{as:key_assumption} is satisfied with certain parameters. 

\begin{proposition}\label{thm:prox_SEGA_SGDA_convergence_appendix_1}
     Let $F$  be $\ell$-star-cocoercive and Assumption~\ref{as:unique_solution} holds. Then, \algname{SEGA-SGDA} satisfies Assumption~\ref{as:key_assumption} with $\sigma_k^2 = \|h^k - F(x^*) \|^2$ and 
    \begin{gather*}
        A = d \ell,\quad B = 2d,\quad D_1 = 0,\quad C = \frac{\ell}{2d}, \quad \rho = \frac{1}{d}, \quad D_2 = 0.
    \end{gather*}
\end{proposition}
\begin{proof}
The result follows from Lemma \ref{lem:sega} and star-cocoercivity of $F$.
\end{proof}

\subsubsection{Analysis of \algname{SEGA-SGDA} in the Quasi-Strongly Monotone Case}

Applying Theorem~\ref{thm:main_result} and Corollary~\ref{cor:main_result} with $M = 4d^2$, we get the following results.

\begin{theorem}\label{thm:prox_SEGA_SGDA_convergence_appendix}
     Let $F$  be $\mu$-quasi strongly monotone, $\ell$-star-cocoercive, Assumption~\ref{as:unique_solution} holds, and $0 < \gamma \leq \frac{1}{6d\ell}$. Then, for all $k \ge 0$ the iterates produced by \algname{SEGA-SGDA} satisfy
    \begin{equation*}
        \Exp\left[\|x^k - x^{*}\|^2\right] \leq \left(1 - \min\left\{\gamma\mu, \frac{1}{2d}\right\}\right)^k \cdot V_0 , \label{eq:prox_SEGA_SGDA_convergence_appendix}
    \end{equation*}
    where $V_0 = \|x^0 - x^*\|^2 + 4d^2\gamma^2 \sigma_0^2$.
\end{theorem}

\begin{corollary}\label{cor:prox_SEGA_convergence_appendix}
    Let the assumptions of Theorem~\ref{thm:prox_SEGA_SGDA_convergence_appendix} hold. Then, for $\gamma = \frac{1}{6d\ell}$ and any $K \ge 0$ we have
	\begin{equation}
        \Exp[\|x^k - x^{*}\|^2] \leq V_0 \exp\left(-\min\left\{\frac{\mu}{6d{\ell}}, \frac{1}{2d}\right\}K\right). \notag
    \end{equation}
\end{corollary}

\subsubsection{Analysis of \algname{SEGA-SGDA} in the Monotone Case}
Next, using Theorem~\ref{thm:main_result_monotone}, we establish the convergence of \algname{CSGDA} in the monotone case.

\begin{theorem}\label{thm:prox_SEGA_convergence_monotone}
   Let $F$  be monotone, $\ell$-star-cocoercive and Assumptions~\ref{as:key_assumption}, \ref{as:boundness}, \ref{as:unique_solution} hold. Assume that $\gamma \leq \nicefrac{1}{6d{\ell}}$. Then for $\text{Gap}_{\cC} (z)$ from \eqref{eq:gap} and for all $K\ge 0$ the iterates produced by \algname{SEGA-SGDA} satisfy
    \begin{align}
    \Exp\left[\text{Gap}_{\cC} \left(\frac{1}{K}\sum\limits_{k=1}^{K}  x^{k}\right)\right] 
    &\leq \frac{3\left[\max_{u \in \mathcal{C}}\|x^{0} - u\|^2\right]}{2\gamma K}  + \frac{8\gamma \ell^2 \Omega_{\mathcal{C}}^2}{K} + 13 d \ell  \cdot \frac{\|x^0-x^{*,0}\|^2}{K} \notag\\
    &\quad + \left(4 +  13 \gamma  d\ell \right) \frac{2d\gamma \sigma_0^2}{K} +9\gamma\cdot \max_{x^*\in X^*}\left[ \|F(x^*)\|^2\right].\notag
\end{align}
\end{theorem}

Applying Corollary~\ref{cor:main_result_monotone}, we get the rate of convergence to the exact solution.

\begin{corollary}\label{cor:prox_SEGA_convergence_monotone_appendix}
    Let the assumptions of Theorem~\ref{thm:prox_SEGA_convergence_monotone} hold. Then $\forall K > 0$ one can choose $\gamma$ as
    \begin{equation*}
        \gamma = \min\left\{\frac{1}{13 d \ell}, \frac{\Omega_{0,\cC}}{\sqrt{2} G^* d},  \frac{\Omega_{0,\cC}}{G_*\sqrt{K}}\right\}.
    \end{equation*}
    This choice of $\gamma$ implies
   \begin{align}
        \Exp\left[\text{Gap}_{\cC} \left(\frac{1}{K}\sum\limits_{k=1}^{K}  x^{k}\right)\right] = \cO\left(\frac{d \ell(\Omega_{0,\cC}^2 + \Omega_0^2) + \ell \Omega_{\cC}^2}{K} + \frac{d \Omega_{0,\cC} G_*}{ K} + \frac{\Omega_{0,\cC} G_*}{\sqrt{K}}\right).\notag
    \end{align}
\end{corollary}
\begin{proof}
The proof follows from the next upper bound $\widehat \sigma_0^2$ for $\sigma^2_0$ with initialization $h_0 = 0$
$$
\sigma^2_0 = \| h_0 - F(x^*)\|^2 = \| F(x^*)\|^2 \leq G^2_*.
$$
\end{proof}

\subsubsection{Analysis of \algname{SEGA-SGDA} in the Cocoercive Case}
Next, using Theorem~\ref{thm:main_result_monotone_coco_appendix}, we establish the convergence of \algname{CSGDA} in the cocoercive case.

\begin{theorem}\label{thm:prox_SEGA_convergence_monotone_coco}
   Let $F$  be $\ell$-cocoercive and Assumptions~\ref{as:key_assumption}, \ref{as:boundness}, \ref{as:unique_solution} hold. Assume that $\gamma \leq \nicefrac{1}{6d{\ell}}$. Then for $\text{Gap}_{\cC} (z)$ from \eqref{eq:gap} and for all $K\ge 0$ the iterates produced by \algname{SEGA-SGDA} satisfy
    \begin{align}
    \Exp\left[\text{Gap}_{\cC} \left(\frac{1}{K}\sum\limits_{k=1}^{K}  x^{k}\right)\right] 
    &\leq \frac{3\left[\max_{u \in \mathcal{C}}\|x^{0} - u\|^2\right]}{2\gamma K}  + 21 d \ell  \cdot \frac{\|x^0-x^{*,0}\|^2}{K} \notag\\
    &\quad + \left(6 +  21 \gamma  d\ell \right) \frac{2d\gamma \sigma_0^2}{K}.\notag
\end{align}
\end{theorem}

Applying Corollary~\ref{cor:main_result_monotone_coco}, we get the rate of convergence to the exact solution.

\begin{corollary}\label{cor:prox_SEGA_convergence_monotone_coco_appendix}
    Let the assumptions of Theorem~\ref{thm:prox_SEGA_convergence_monotone_coco} hold. Then $\forall K > 0$ one can choose $\gamma$ as
    \begin{equation*}
        \gamma = \min\left\{\frac{1}{21 d \ell}, \frac{\Omega_{0,\cC}}{\sqrt{2} G^* d}\right\}.
    \end{equation*}
    This choice of $\gamma$ implies
   \begin{align}
        \Exp\left[\text{Gap}_{\cC} \left(\frac{1}{K}\sum\limits_{k=1}^{K}  x^{k}\right)\right] = \cO\left(\frac{d \ell(\Omega_{0,\cC}^2 + \Omega_0^2)}{K} + \frac{d \Omega_{0,\cC} G_*}{ K}\right).\notag
    \end{align}
\end{corollary}

\subsection{Comparison with Related Work}

The summary of rates in the (quasi-) strongly monotone case is provided in Table~\ref{tab:coord_methods}. First of all, our results are the first convergence for solving \textit{regularized} VIPs via coordinate methods. In particular, \algname{SEGA-SGDA} is the first linearly converging coordinate method for solving regularized VIPs. Next, when $q = 2$ in \algname{zoVIA} from \citet{sadiev2020zeroth}, i.e., Euclidean proximal setup is used, our rate for \algname{SEGA-SGDA} is better than the one derived for \algname{zoVIA} in \citet{sadiev2020zeroth} since $\ell \leq \nicefrac{L^2}{\mu}$. Finally, \algname{zoscESVIA} might have better rate, but it is based on \algname{EG} and it uses approximation of each component of operator $F$ at each iteration, which makes one iteration of the method costly.

In the monotone and cocoercive cases, our result and the results from \citet{sadiev2020zeroth} are comparable modulo the difference between (star-)cocoercivity and Lipschitz constants.

\begin{table*}[h]
		\centering
		\caption{\small Summary of the complexity results for zeroth-order methods with two-points feedback oracles for solving \eqref{eq:VI}. By complexity we mean the number of oracle calls required for the method to find $x$ such that $\Exp[\|x - x^*\|^2] \leq \varepsilon$. By default, operator $F$ is assumed to be \textbf{$\mu$-strongly monotone} and, as the result, the solution is unique. Our results rely on \textbf{$\mu$-quasi strong monotonicity} of $F$ \eqref{eq:QSM}. Methods supporting $R(x) \not\equiv 0$ are highlighted with $^*$. Our results are highlighted in green. Notation: $q =$ the parameter depending on the proximal setup, $q = 2$ in Euclidean case and $q = +\infty$ in the $\ell_1$-proximal setup; $G_* = \max_{x* \in X^*}\left\| F(x^{*})\right\|$, which is zero when $R(x) \equiv 0$.}
		\label{tab:coord_methods}    
		\begin{threeparttable}
			\begin{tabular}{|c|c c c|}
			\hline
				Method & Citation & Assumptions & Complexity\\
				\hline
				\hline
				\algname{zoscESVIA}\tnote{\color{blue}(1)} & \citep{sadiev2020zeroth} & $F$ is $L$-Lip.\tnote{{\color{blue}(2)}} & $\widetilde{\cO}\left(d\frac{L}{\mu}\right)$\\
				\hline\hline
				\algname{zoVIA} & \citep{sadiev2020zeroth} & $F$ is $L$-Lip.\tnote{{\color{blue}(2)}} & $\widetilde{\cO}\left(d^{\nicefrac{2}{q}}\frac{L^2}{\mu^2}\right)$\\
			    \rowcolor{bgcolor2}\algname{CSGDA}$^*$ & \textbf{This paper} & $F$ is $\ell$-cocoer. & $\widetilde{\cO}\left(d\frac{\ell}{\mu} + \frac{dG_*^2}{\mu^2\varepsilon}\right)$\\
			    \rowcolor{bgcolor2}\algname{SEGA-SGDA}$^*$ & \textbf{This paper} & \begin{tabular}{c}
			         $F$ is $\ell$-cocoer.\\
			         As.~\ref{as:unique_solution}
			    \end{tabular} & $\widetilde{\cO}\left(d + d\frac{\ell}{\mu}\right)$\\
			    \hline
			\end{tabular}
			        {\small
					\begin{tablenotes}
					    \item [{\color{blue}(1)}] The method is based on Extragradient update rule. Moreover, at each step full operator is approximated.
					    \item [{\color{blue}(2)}] The problem is defined on a bounded set.
					\end{tablenotes}}
		\end{threeparttable}
\end{table*}

\end{document}